\documentclass{amsart}
\usepackage{amscd,amsmath,amssymb,amsfonts}
\usepackage[cmtip, all]{xy}

\newtheorem{thm}{Theorem}[section]
\newtheorem{prop}[thm]{Proposition}
\newtheorem{lem}[thm]{Lemma}
\newtheorem{cor}[thm]{Corollary}

\renewcommand{\theclaim}{\kern-3pt}

\newtheorem{IntroThm}{Theorem}

\theoremstyle{definition}
\newtheorem{definition}[thm]{Definition}

\theoremstyle{remark}
\newtheorem{rem}[thm]{Remark}
\newtheorem{rems}[thm]{Remarks}
\newtheorem{ex}[thm]{Example}

\numberwithin{equation}{section}

\newcommand{\sA}{{\mathcal A}}
\newcommand{\sB}{{\mathcal B}}
\newcommand{\sC}{{\mathcal C}}

\newcommand{\sE}{{\mathcal E}}
\newcommand{\sF}{{\mathcal F}}
\newcommand{\sG}{{\mathcal G}}

\newcommand{\sL}{{\mathcal L}}

\newcommand{\sO}{{\mathcal O}}

\newcommand{\sR}{{\mathcal R}}

\newcommand{\sT}{{\mathcal T}}
\newcommand{\sU}{{\mathcal U}}

\newcommand{\A}{{\mathbb A}}

\newcommand{\F}{{\mathbb F}}

\renewcommand{\H}{{\mathbb H}}

\newcommand{\N}{{\mathbb N}}
\renewcommand{\P}{{\mathbb P}}
\newcommand{\Q}{{\mathbb Q}}
\newcommand{\R}{{\mathbb R}}

\newcommand{\Z}{{\mathbb Z}}

\renewcommand{\L}{{\mathbb L}}

\renewcommand{\phi}{\varphi}

\newcommand{\an}{{\rm an}}

\newcommand{\CH}{{\rm CH}}

\newcommand{\Hom}{{\rm Hom}}

\newcommand{\Spec}{{\rm Spec \,}}
\newcommand{\Proj}{{\rm Proj \,}}

\newcommand{\0}{\emptyset}
\newcommand{\sHom}{{\mathcal{H}{om}}}

\newcommand{\id}{{\operatorname{id}}}
\newcommand{\Zar}{{\text{\rm Zar}}}

\newcommand{\Sch}{{\operatorname{\mathbf{Sch}}}}

\newcommand{\op}{{\text{\rm op}}}

\newcommand{\Sets}{{\mathbf{Sets}}}

\renewcommand{\max}{{\operatorname{\rm max}}}

\newcommand{\Spt}{{\mathbf{Spt}}}

\newcommand{\Sm}{{\mathbf{Sm}}}
\newcommand{\Prj}{{\mathbf{Proj}}}

\renewcommand{\lim}{\operatornamewithlimits{\varprojlim}}
\newcommand{\colim}{\operatornamewithlimits{\varinjlim}}

 \newcommand{\Ab}{{\mathbf{Ab}}}
\newcommand{\Tot}{{\operatorname{\rm Tot}}}

\newcommand{\eff}{{\operatorname{eff}}}

\newcommand{\DM}{{DM}}

\newcommand{\Nis}{{\operatorname{Nis}}}

\newcommand{\ds}{{/\kern-3pt/}}

\newcommand{\coker}{\operatorname{coker}}

\newcommand{\Cor}{{\mathop{Cor}}}

\newcommand{\Deg}{{\mathop{\rm{deg}}}}

\newcommand{\PSh}{{\mathop{PSh}}}
\newcommand{\Sh}{{\mathop{Sh}}}

\newcommand{\PST}{{\mathop{PST}}}

\newcommand{\equi}{{\mathop{equi}}}

\newcommand{\et}{{\operatorname{\acute{e}t}}}

\newcommand{\Aut}{{\operatorname{Aut}}}
\newcommand{\Mod}{{\operatorname{Mod}}}

\newcommand{\Cube}{{\mathbf{Cube}}}
\newcommand{\ECube}{{\mathbf{ECube}}}
\newcommand{\un}{\underline}
\newcommand{\dgn}{{\operatorname{degn}}}
\renewcommand{\dim}{\text{dim}}
\newcommand{\PTR}{{\operatorname{Pre-Tr}}}
\newcommand{\Cone}{{\operatorname{Cone}}}
\newcommand{\sgn}{{\operatorname{sgn}}}
\newcommand{\Sgn}{{\operatorname{Sgn}}}
\newcommand{\alt}{{\operatorname{alt}}}
\newcommand{\Alt}{{\operatorname{Alt}}}
\newcommand{\Obj}{{\operatorname{Obj}}}
\newcommand{\Iso}{{\operatorname{Iso}}}
\newcommand{\DGCor}{{\mathop{dgCor}}}
\newcommand{\Op}{{\mathop{Opn}}}
\newcommand{\Pt}{{\mathop{Pt}}}

\newcommand{\Otimes}{\displaystyle\mathop{\otimes}}
\newcommand{\Mor}{{\text{\rm Mor}}}
\newcommand{\DGPCor}{{\mathop{dgPrCor}}}
\newcommand{\DGPMot}{{\mathop{dgSmMot}}}
\newcommand{\PMot}{{\mathop{SmMot}}}
\newcommand{\DGTCor}{{\mathop{dgLCor}}}
\newcommand{\DGTMot}{{\mathop{dgTMot}}}
\newcommand{\PTMot}{{\mathop{DTMot}}}

\newcommand{\chow}{{\operatorname{chow}}}
\newcommand{\SW}{{\operatorname{S-W}}}
\newcommand{\CM}{{\mathop{ChMot}}}
\newcommand{\Loc}{{\operatorname{Loc}}}
\newcommand{\DTM}{{\operatorname{DTM}}}

\begin{document}

\title{Smooth motives}
\author{Marc Levine}
\address{
Department of Mathematics\\
Northeastern University\\
Boston, MA 02115\\
USA}
\email{marc@neu.edu}


 \thanks{The  author gratefully acknowledges the support of the Humboldt Foundation, and the support of the NSF via the grant DMS-0457195}
\renewcommand{\abstractname}{Abstract}
\begin{abstract}  Following ideas of Bondarko, we construct a DG category whose homotopy category is equivalent to the full subcategory of motives over a base-scheme $S$ generated by the motives of smooth projective $S$-schemes, assuming that $S$ is itself smooth over a perfect field.
\end{abstract}
\maketitle
\tableofcontents

\section*{Introduction}  Recently, Bondarko \cite{Bondarko} has given a construction of a DG category of motives over a field $k$, built out of ``higher finite correspondences" between smooth projective varieties over $k$. Assuming resolution of singularities, the homotopy category of this DG category is equivalent to Voevodsky's category of effective geometric motives. The main goal in this paper is to extend this construction to the case of motives over a base-scheme $S$. For simplicity, we restrict to the case of a regular $S$, essentially of finite type over a field, although many aspects of the construction should be possible in a more general setting. 

If $S$ has  positive Krull dimension, one would not expect that  the motive of an arbitrary smooth $S$-scheme be expressible in terms of motives of smooth projective $S$-schemes. Thus,  the category of motives we construct represents a special type of motive over $S$. Since the Betti realization of our motives will land in the derived category of local systems on $S^\an$ rather than in the derived category of constructible sheaves, we call our category $\PMot^\eff(S)$ the category of {\em smooth effective motives over $S$}. We have as well a version with the Tate motive inverted, $\PMot(S)$. Both $\PMot^\eff(S)$ and $\PMot(S)$ are constructed by taking the homotopy category of suitable DG categories, and then taking an idempotent completion.

We were not able to construct directly a tensor structure on $\PMot^\eff(S)$ or $\PMot(S)$. However, after passing to $\Q$-coefficients, we replace our cubical construction with alternating cubes, which makes possible a tensor structure on  DG categories whose homotopy categories are equivalent to  $\PMot^\eff(S)_\Q$ and $\PMot(S)_\Q$ (up to idempotent completion).

Our main comparison result involves the categories of motives $\DM^\eff(S)$ and $\DM(S)$ constructed by Cisinski-D\'eglise. We construct exact functors
\[
\rho^\eff_S:\PMot^\eff(S)\to \DM^\eff(S);\ \rho_S:\PMot(S)\to \DM(S)
\]
and we show
\begin{IntroThm} Let $k$ be a field. Suppose that $S$ is a smooth $k$-scheme, essentially of finite type over $k$. Then $\rho_S$ is a fully faithful embedding.
\end{IntroThm}
This of course implies that $\rho^\eff_S$ is a faithful embedding, but due to the lack of a cancellation theorem in $\DM^\eff(S)$, we do not know if $\rho^\eff_S$ is full. 

Our main technical tool is an extension of the Friedlander-Lawson-Voevodsky moving lemmas to the case of cycles on a smooth projective scheme over a regular semi-local base scheme $B$, with $\sO_B$ containing a field. This enables us to extend the fundamental duality theorem for equi-dimensional cycles on smooth projective varieties to smooth projective schemes over a regular semi-local base (over a field).  We pass from the semi-local case to an arbitrary regular base (over a field) by making a Zariski sheafification; we were not able to extend the available techniques beyond the semi-local case, so the Zariski sheafification is forced upon us. We do not know if the duality theorem over a general base holds before making the Zariski sheafification.

As hinted above, our interest in these constructions arose from our desire to construct a refined realization functor on the subcategory of $\DM(S)$ generated by smooth projective $S$-schemes. One example, given above, is that we should have a realization functor
\[
\Re_B:\PMot(S)\to D^b(\Loc/S^\an),
\]
where $\Loc/S^\an$ is the abelian category of local systems of abelian groups on $S^\an$, refining the usual Betti realization of $\DM_{gm}(S)$ into the derived category of constructible sheaves. Similarly, one should have realizations of $\PMot(S)$ to the derived categories of smooth $l$-adic \'etale sheaves on $S^\et$ or variations of mixed Hodge structures on $S^\an$. By our main theorem, we can view the triangulated category $\DTM(S)$ of mixed Tate motives over $S$ as the full subcategory of $\PMot(S)$ generated by the Tate twists of the motive of $S$. Our construction of $\PMot(S)$ as the homotopy of a DG category (after taking an idempotent completion) gives a similar DG description of $\DTM(S)$. Thus, we can hope to refine the realization functors for $\PMot(S)$ even further if we restrict to $\DTM(S)$. This should give us a Betti realization functor on $\DTM(S)$ to the derived category of uni-potent local systems on $S^\an$, an \'etale realization functor to relatively uni-potent \'etale sheaves on $S^\et$ and a Hodge realization to uni-potent variations of mixed Hodge structures on $S^\an$. The paper of Deligne-Goncharov \cite{DeligneGoncharov} and our work with Esnault \cite{EsnaultLevine}, giving constructions of the mixed Tate fundamental group for some types of schemes $S$, gave us the motivation for the construction of categories of smooth motives and refined realization functors.
As this paper is long enough already, we will postpone the construction of these realization functors to a future work.

The paper is organized as follows. We begin with a resum\'e of the cubical category and cubical constructions. This is a more convenient setting for constructing commutative DG structures than the simplicial one; we took the opportunity here of collecting a number of useful results on cubical constructions that are scattered throughout the literature. We also discuss a useful refinement of cubical structures involving the extended cubical category. This variation on the cubical theme adds the cubical analog of the simplicial degeneracy maps; many of the most useful results on cubical objects that arise in nature actually use the extended cubical structure, so we thought it would be useful to give an abstract discussion.

The next section deals with various versions of Kapranov's construction of complexes over a DG category and the associated triangulated homotopy category. As detailed verifications of the fundamental properties of these constructions are not available in the literature, we thought it would be a good idea to give a complete treatment of this useful construction, with the hope that our total signed contribution to the current level of sign errors would be negative.

In section \S\ref{sec:DGMotives} we apply this machinery to the category of correspondences, endowed with the algebraic $n$-cubes as a cubical object. This leads to our construction of the DG category of higher correspondences, $\DGCor_S$, the full DG subcategory $\DGPCor_S$ of correspondences on smooth projective $\S$-scheme,  the Zariski sheafified version $R\Gamma(S, \un{\DGPCor}_S)$, the DG category of motivic complexes $\DGPMot^\eff_S:=C^b(R\Gamma(S, \un{\DGPCor}_S))$, and finally the triangulated  category of smooth effective motives over $S$, $\PMot^\eff(S)_S$, defined by taking the idempotent completion of the  homotopy category $K^b(R\Gamma(S, \un{\DGPCor}_S))$. We also define the $\Q$-version with alternating cubes, $\DGPMot^\eff_{S\Q}$, which is a DG tensor category, leading to the triangulated tensor category   $\PMot^\eff(S)_\Q$. Finally, we consider versions of these categories, $\DGPMot_S$, $\DGPMot_{S\Q}$, $\PMot(S)$ and $\PMot(S)_\Q$, formed by inverting the Lefschetz motive. 

In \S\ref{sec:Duality} we state our main duality theorem for equi-dimensional cycles over a semi-local base (theorem~\ref{thm:Duality}), as well as the the projective bundle formula (theorem~\ref{thm:PBF}). We derive the consequences of these results for duality in the categories $\PMot^\eff(S)$ and 
 $\PMot(S)$. In  \S\ref{sec:Motives}, we briefly recall some aspects of the definition of the Cisinski-D\'eglise categories of motives over a base, $\DM^\eff(S)$ and $\DM(S)$, define exact functors 
\begin{align*}
&\rho_S^\eff:\PMot^\eff(S)\to \DM^\eff(S)\\
&\rho_S:\PMot(S)\to \DM(S),
\end{align*}
and prove our main result (corollary~\ref{cor:Main}). Finally, in section  \S\ref{sec:FLVMov} we prove our extension of the Friedlander-Lawson-Voevodsky moving lemmas and give the proofs of theorems \ref{thm:Duality} and \ref{thm:PBF}. 

\section{Cubical objects and DG categories} 

\subsection{Cubical objects} We recall some notions discussed in e.g. \cite{Additive}.
We introduce the ``cubical category" $\Cube$. This is the subcategory of $\Sets$ with objects  $\un{n}:=\{0,1\}^n$, $n=0,1,2,\ldots$, and morphisms generated by\\
\\
1. Inclusions: $\eta_{n,i,\epsilon}:\un{n}\to \un{n+1}$, $\epsilon=0,1$, 
$i=1,\ldots, n+1$
\[
\eta_{n,i,\epsilon}(y_1,\ldots, y_{n-1})=(y_1,\ldots, y_{i-1},\epsilon,y_i,\ldots, y_{n-1})
\]
2. Projections: $p_{n,i}:\un{n}\to\un{n-1}$, $i=1,\ldots, n$.\\
3. Permutations of factors: 
$(\epsilon_1,\ldots, \epsilon_n)\mapsto (\epsilon_{\sigma(1)},\ldots, \epsilon_{\sigma(n)})$ for $\sigma\in S_n$.\\
4. Involutions: $\tau_{n,i}$ exchanging $0$ and $1$ in the $i$th factor of $\un{n}$.\\
\\
Clearly all the Hom-sets in $\Cube$ are finite. For a category $\sA$, we call a functor $F:\Cube^\op\to \sA$ a {\em cubical object} of $\sA$ and a functor $F:\Cube\to \sA$ a {\em co-cubical object} of $\sA$.

\begin{rem} The permutations and involutions in (3) and (4) give rise to a subgroup of $\Aut_{\Sets}(\un{n})$ isomorphic to the semi-direct product $F_n:=(\Z/2)^n\ltimes\Sigma_n$, where 
$\Sigma_n$ acts on $(\Z/2)^n$ by permuting the factors.

We extend the standard sign representation of $\Sigma_n$ to  the sign representation
\[
\sgn:F_n\to \{\pm1\}
\]
by
\[
\sgn(\epsilon_1,\ldots, \epsilon_n,\sigma):=(-1)^{\sum_j\epsilon_j}\sgn(\sigma).
\]
\end{rem}

\begin{ex}  Let $S$ be a scheme, set   $\A_S^1:=\Spec_S\, \sO_S[y]$. We set $\square_S^n:=(\A^1_S)^n$. $S_n$ acts on $\square^n_S$ by permuting the factors. We let $\Z/2$ act on $\A^1_S$ by $x\mapsto 1-x$. This gives us an action of $F_n$ on $\square^nS$.

Letting $p_{n,i}:  ({\A}_S^{1})^n\to{\A}_S^{1}$ be the  $i$th projection, we use the  coordinate system 
$(y_1, \cdots , y_n)$ on $\square_S^n$, with $y_i:=y\circ  p_{n,i}$.

Let $\eta_{n,i,\epsilon}:\square_S^{n-1}\to \square_S^n$ be the inclusion
\[
\eta_{n,i,\epsilon}(y_1,\ldots, y_{n-1})=(y_1,\ldots, y_{i-1},\epsilon,y_i,\ldots, y_{n-1})
\]
This gives us the  co-cubical object $n\mapsto \square_S^n$ in $\Sm/S$.

A {\em face} of $\square_S^n$ is  a subscheme $F$ defined by equations of the form
 \[
 y_{i_1}=\epsilon_1, \ldots,  y_{i_s}=\epsilon_s;\ \epsilon_j\in\{0,1\}.
 \]
\end{ex}

\subsection{Cubical objects in a pseudo-abelian category}
Let $\sA$ be a pseudo-abelian category, $\un{A}:\Cube^\op\to \sA$ a cubical object. For $\epsilon\in\{0,1\}$, let  $\pi^\epsilon_{n,i}:\un{A}(\un{n})\to \un{A}(\un{n})$ be the endomorphism $p_{n,i}^*\circ \eta_{n-1,i,\epsilon}^*$, and set
\[
\pi_n:=(\id-\pi^1_{n,n})\circ\ldots\circ (\id-\pi^1_{n,1}).
\]
Note that the $\pi^\epsilon_{n,i}$ are commuting idempotents, and that the subobject $(\id-\pi^\epsilon_{n,i})^*(\un{A}(\un{n}))\subset \un{A}(\un{n})$ is a kernel for $\eta_{n,i,\epsilon}$. Since $\sA$ is pseudo-abelian, the objects
\[
\un{A}(\un{n})^0:=\cap_{i=1}^n\ker \eta_{n-1,i,1}^*\subset \un{A}(\un{n})
\]
and
\[
\un{A}(\un{n})^\dgn:=\sum_{i=1}^np_{n,i}^*(\un{A}(\un{n-1}))\subset \un{A}(\un{n})
\]
are well-defined.  

Let $(\un{A}_*,d)$ be the complex with $\un{A}_n:=\un{A}(\un{n})$ and 
with
\[
d_n:=\sum_{i=1}^n(-1)^i(\eta_{n,i,1}^*-\eta_{n  ,i,0}^*):\un{A}_{n+1}\to \un{A}_n.
\]
Write $\un{A}^0_n, \un{A}^\dgn_n$ for $\un{A}^0(\un{n}), \un{A}^\dgn(\un{n})$, respectively.

The following result is the basis of all ``cubical" constructions; the proof is elementary and is left to the reader.
\begin{lem}\label{lem:Cube0}
 Let $\un{A}:\Cube^\op\to\sA$ be a cubical object in a  pseudo-abelian category $\sA$. Then\\
\\
1. For each $n$, $\pi_n$ maps $\un{A}_n$ to $\un{A}_n^0$ and defines a splitting
\[
\un{A}_n=\un{A}_n^\dgn\oplus \un{A}_n^0.
\]
2. $d_n(\un{A}_n^\dgn)=0$, $d_n(\un{A}_n^0)\subset \un{A}_{n-1}^0$
\end{lem}

\begin{definition}
 Let $\un{A}:\Cube^\op\to\sA$ be a cubical object in a pseudo-abelian category $\sA$. Define
 the complex $(A_*,d)$ to be the quotient complex  
\[
\un{A}_*/\un{A}_*^\dgn
\]
 of $\un{A}_*$.
\end{definition}

Lemma~\ref{lem:Cube0} shows that $A_*$ is well-defined and is isomorphic to the subcomplex $\un{A}_*^0$ of $\un{A}_*$. We often use cohomological notation, with $A^n:=A_{-n}$, etc.

\subsection{Products} Suppose we have two cubical objects 
\[
\un{A},\un{B}:\Cube^\op\to\sA
\]
in a tensor category $(\sA, \otimes)$. Form the diagonal cubical object $\un{A\otimes B}$ by
\[
\un{A\otimes B}(\un{n}):=\un{A}(\un{n})\otimes \un{B}(\un{n})
\]
and on  morphisms by
\[
\un{A\otimes B}(f):=\un{A}(f)\otimes\un{B}(f).
\]

Let $p^1_{n,m}:\un{n+m}\to\un{n}$, $p^2_{n,m}:\un{n+m}\to\un{m}$ be the projections on the first $n$ and last $m$ factors, respectively. Let
\[
\cup_{A,B}^{n,m}:\un{A}(\un{n})\otimes\un{B}(\un{m})\to \un{A}(\un{n+m})\otimes\un{B}(\un{n+m})
\]
be the map $\un{A}(p^1_{n,m})\otimes\un{B}(p^2_{n,m})$. One easily checks that the direct sum of the maps $\cup_{A,B}^{n,m}$ defines a map of complexes
\begin{equation}\label{eqn:Cup}
\cup_{A,B}:\un{A}^*\otimes\un{B}^*\to \un{A\otimes B}^*.
\end{equation}
It is easy to see that we have an associativity property
\begin{equation}\label{eqn:Assoc}
\cup_{A\otimes B,C}\circ(\cup_{A,B}\otimes\id_{\un{C}^*})=\cup_{A,B\otimes C}\circ(\id_{\un{A}^*}\otimes\cup_{B,C})
\end{equation}
but not in general a commutativity property. 

\subsection{Alternating cubes} Recall the semi-direct product $F_n:=(\Z/2)^n\ltimes\Sigma_n$ and the sign representation $\Sgn:F_n\to\{\pm1\}$. If $\sA$ is a pseudo-abelian category and $M$ an $F_n$-module in $\sA$ (i.e., we are given a homomorphism $F_n\to \Aut_\sA(M)$), we let $M^\Sgn$ be the largest subobject of $M$ on which $F_n$ acts by the sign representation:
\[
M^\Sgn:=\cap_{g\in F_n}\ker((g-\Sgn(g)\id_M).
\]
Similarly, if $\Sigma_n$ acts on $M$, we let $M^\sgn$ be the subobject of $M$
\[
M^\sgn:=\cap_{g\in \Sigma_n}\ker((g-\sgn(g)\id_M).
\]

Let $\un{A}:\Cube^\op\to \sA$ be a cubical object in a pseudo-abelian category $\sA$.  For each $n=0,1, 2,\ldots$, define the subobject 
$\un{A}^\alt(\un{n})$ of $\un{A}(\un{n})$ by 
\[
\un{A}^\alt(\un{n}):=\un{A}(\un{n})^{\sgn}
\]
Similarly, let $A^\alt(n):=A(\un{n})^\sgn$.

\begin{lem} 1. $n\mapsto \un{A}^\Alt(\un{n})$ defines a sub-cubical object of $\un{A}$.\\
\\
2. $n\mapsto A^\alt(\un{n})$ defines a sub-complex of $A_*$.\\
\\
3. Suppose $2\times\id$ is invertible on all the objects $\un{A}(\un{n})$. Then the map $\un{A}_*\to A_*$ induces an isomorphism  of complexes $\un{A}^\Alt_*\to A^\alt_*$.
\end{lem} 

\begin{proof} This is straightforward, noting that the degenerate subcomplex is killed by the idempotent
\[
\frac{1}{2^n}\prod_{i=1}^n(1-\tau_i)
\]
where $\tau_i$ is the involution in the $i$th factor of $\un{n}$.
\end{proof}

\subsection{Extended cubes} We note that product makes $\Sets$ a symmetric monoidal category, and that $\Cube$ is a symmetric monoidal subcategory. Let $\ECube$ be the smallest symmetric monoidal subcategory of $\Sets$ having the same objects as $\Cube$, containing $\Cube$ and containing the morphism
\[
\mu:\un{2}\to\un{1}
\]
defined by the multiplication of integers:
\[
\mu((1,1))=1;\ \mu(a,b)=0\text{ for }(a,b)\neq(1,1).
\]
An {\em extended cubical object} in a category $\sC$ is a functor $F:\ECube^\op\to\sC$.

Let $\un{F}:\Cube^\op\to\sA$ be a cubical object in a pseudo-abelian category. Let $NF(\un{n})\subset \un{F}(\un{n})$ be the subobject
\[
NF(\un{n}):= \cap_{i=2}^n\ker(\eta^*_{n,i,0})\cap \cap_{i=1}^n\ker(\eta^*_{n,i,1}).
\]
This defines the {\em normalized subcomplex} $NF^*$ of $\un{F}^*$. Note that $NF^*$ is a subcomplex of  $\un{F}(\un{*})_0$.

\begin{lem}\label{lem:Normalized} Let $\un{F}:\ECube^\op\to\sA$ be an extended cubical object in a pseudo-abelian category $\sA$. Then the inclusion $i:NF^*\to \un{F}(\un{*})_0$ is a homotopy equivalence.
\end{lem}

\begin{proof} Let 
\[
N^MF_n:=\begin{cases}
\cap_{i=n-M}^n\ker(\eta^*_{n,i,0})\cap\cap_{i=1}^n\ker(\eta^*_{n,i,1})&\text{ for }n-M>2\\
NF^n&\text{ for }n-M\le 2
\end{cases}
\]
For each $M$, the subobjects $N^MF_n\subset \un{F}_n^0$ form a subcomplex 
$N^MF_*$ of $\un{F}_*^0$ which contains $NF_*$ and agrees with $NF_*$ in degrees $n\le M+2$. Since $N^{-1}F_*= \un{F}_*^0$, it thus suffices to show that the inclusion
\[
i^M:N^MF_*\to N^{M-1}F_*
\]
is a homotopy equivalence, such that the chosen homotopy inverse $p^M$ and  chosen homotopy $h^M$ between $i^M\circ p^M$ and $\id$ satisfy:
\begin{enumerate}
\item $p^M\circ i^M=\id$
\item  $h_n^M:N^{M-1}F_{n-1}\to N^{M-1}F_n$ is the zero map for $n\le M$ 
\end{enumerate}
Indeed, in this case, the infinite composition
\[
p:=\ldots p^M\circ p^{M-1}\circ\ldots\circ p^0
\]
makes sense, as does the infinite sum
\[
h:=\sum_M i^0\circ\ldots i^{M-1}\circ h^M\circ p^{M-1}\circ\ldots p^{0}.
\]
The map  $p$ gives a homotopy inverse to $i$, with $pi=\id$ and  $h$ defines a homotopy between $ip$ and the identity. We proceed to define the maps $p^M$ and $h^M$.

Define the map $q:\un{2}\to\un{1}$ by
\[
q(x,y)=1-\mu(1-x,1-y)=1-(1-x)(1-y).
\]
For $1\le i\le n-1$, let $q_{n,i}:\un{n}\to\un{n-1}$ be the map
\[
q_{n,i}(x_1,\ldots, x_n):=(x_1, x_2,\ldots, x_{i-1},q(x_i, x_{i+1}),x_{i+2},\ldots, x_n).
\]
Then
 \begin{align}\label{align:Relations}
&q_{n,i}\circ\eta_{n,i,0}=q_{n,i}\circ\eta_{n,i+1,0}=\id_{\un{n-1}}\\
&q_{n,i}\circ\eta_{n,i,1}= q_{n,i}\circ\eta_{n,i+1,1}= \eta_{n,i,1}\circ p_{n-1,i}\notag\\
&q_{n,i}\circ\eta_{n,j,\epsilon}=
\begin{cases}
\eta_{n-1,j,\epsilon}\circ q_{n-1,i-1}&\text{ for }1\le j<i\\
\eta_{n-1,j-1,\epsilon}\circ q_{n-1,i}&\text{ for }i+1< j\le n\\
\end{cases}\notag
\end{align}
 Defining $q_{n,j}^*=0$ for $j\le 0$ or $j\ge n$, this implies that, for $i\ge1$, the maps
\[
p^M_n:= \id-q_{n,n-M-1}^*\circ\eta_{n-1,n-M,0}^*
\]
define a map of complexes $p^M:\un{F}_*^0\to \un{F}_*^0$ which restricts to the inclusion
$N^MF_*\to  \un{F}_*^0$ on $N^MF_*$, and maps $N^{M-1}F_*$ to $N^MF_*$. We let 
\[
p^M:N^{M-1}F_*\to N^MF_*
\]
be the restriction. 

Let $h_n^M:\un{F}^0_{n-1}\to \un{F}_n^0$ be the map $(-1)^{n-M}q^*_{n,n-M-1}$.  The relations \eqref{align:Relations} imply that $h_n^M$ restricts to a map 
\[
h_n^M:N^{M-1}F_{n-1}\to N^{M-1}F_{n}
\]
and, on $\un{F}_n^0$,  we have
\begin{align*}
d_{n}&h^{n+1}_M+h^n_Md_{n-1}\\&=(-1)^{n-M+1}\left[
\sum_{j=1}^{n+1}(-1)^{j}\eta_{n,j,0}^*\circ q_{n+1,n-M}^*-\sum_{l=1}^n(-1)^{l}q_{n,n-M-1}^*\circ\eta_{n-1,l,0}^*\right]\\
&=(-1)^{n-M+1}
\sum_{j=1}^{n-M-1}(-1)^{j}q_{n,n-M-1}^*\circ\eta_{n-1,j,0}^*\\
&\hskip20pt +(-1)^{n-M+1}\sum_{j=n-M+2}^{n+1}(-1)^{j}q_{n,n-M}^*\circ\eta_{n-1,j-1,0}^*\\
&\hskip 40pt -(-1)^{n-M+1}\sum_{l=1}^n(-1)^{l}q_{n,n-M-1}^*\circ\eta_{n-1,l,0}^*\\
&=q_{n,n-M-1}^*\eta^*_{n-1,n-M,0}\\
&\hskip20pt +(-1)^{n-M}\sum_{j=n-M+1}^n(-1)^{j}(q_{n,n-M}^*+q_{n,n-M-1}^*)\circ\eta_{n-1,j,0}^*
\end{align*}
Since $\eta_{n-1,j,0}^*=0$ on $N^{M-1}F_n$ for $j\ge n-M+1$, the $h_n^M$ give the desired homotopy.
\end{proof}

Let $\un{F}:\ECube^\op\to\sA$ be an extend cubical object in an abelian category $\sA$. Then we have the following description of $H_n(NF_*)$:
\[
H_n(NF_*)=\frac{\cap_{i=1}^n\ker \eta^*_{n-1,i,0}\cap \cap_{i=1}^n\ker \eta^*_{n-1,i,1}}
{\eta^*_{n,1,0}[\cap_{i=2}^{n+1}\ker \eta^*_{n,i,0}\cap \cap_{i=1}^{n+1}\ker \eta^*_{n,i,1}]}
\]
From this description, we see that the symmetric group $S_n$ acts on $H_n(NF_*)$ through the permutation action on $\un{n}$ and the action $\sigma\mapsto \id\times F(\sigma)$ on $\un{F}(\un{n+1})$.  Via lemma~\ref{lem:Normalized}, this gives us an $S_n$-action on $H_n(F_*)$.

\begin{prop}\label{prop:Alt} $\un{F}:\ECube^\op\to\sA$ be an extended cubical object in an abelian category $\sA$. Suppose that the Hom-groups in $\sA$ are $\Q$-vector spaces. Then the inclusion
\[
F^\alt_*\to F_*
\]
is a quasi-isomorphism.
\end{prop}

\begin{proof} Since the Hom-groups in $\sA$ are $\Q$-vector spaces, the natural map
\[
H_*(F^\alt)\to H_*(F)^\alt
\]
is an isomorphism. Thus, we need only show that the symmetric group $S_n$ acts on $H_n(F_*)$ by the sign representation. 

Fix an element in $Z_n(NF_*)$ representing a class $[z]\in H_n(NF_*)$, i.e.
\[
\eta_{n-1,i,\epsilon}^*(z)=0
\]
for all $i$ and for $\epsilon=0,1$. Let $\tau:\un{n}\to\un{n}$ be the permutation exchanging the first two factors. Let $h_n:\un{n+1}\to\un{n}$ be the map
\[
h_n(x_1,x_2, x_3,\ldots, x_{n+1}):=(x_2, q(x_1, x_3), x_4,\ldots, x_{n+1}),
\]
and let $b:=h_n^*(z)$. Then
\begin{align*}
&h_n\circ \eta_{n,1,0}=\id\\
&h_n\circ\eta_{n,2,0}(x_1,\ldots, x_n)=(0, q(x_1,x_2), x_3,\ldots,x_n)=\\
&h_n\circ\eta_{n,3,0}=\tau\\
&h_n\circ \eta_{n,j,0}=\eta_{n-1,j-1,0}\circ h_{n-1}\text{ for }j\ge 4.
\end{align*}
Similarly, $h_n\circ\eta_{n,j,1}=\eta_{n-1,j',1}\circ f_{n,j}$ for some $j'$ and some map $f_{n,j}:\un{n}\to \un{n-1}$. Thus
\[
db=z+\tau^*(z)
\]
proving the result.
\end{proof}

\subsection{Cubical enrichments and DG categories} For a complex $C\in C(\Ab)$, we have the group of cycles in degree $n$, $Z^nC$ and the cohomology $H^nC$. For complexes $C,C'$, we have the Hom-complex $\sHom_{C(\Ab)}(C,C')^*$, with
\[
\sHom_{C(\Ab)}(C,C')^n:=\prod_p\Hom_\Ab(C^p,C^{\prime n+p},
\]
with differential
\[
d_{C',C}f:=d_{C'}\circ f-(-1)^{\Deg f}f\circ d_C.
\]
We have as  well as the group of maps of complexes 
\[
\Hom_{C(\Ab)}(C,C'):=Z^0\sHom_{C(\Ab)}(C,C')^*.
\]

For us a {\em DG category} is simply a category enriched in complexes of abelian groups, possessing finite direct sums. Concretely, for objects $X,Y$ in a DG category $\sC$, one has the {\em Hom complex} $\sHom_\sC(X,Y)^*\in C(\Ab)$, and for $X,Y, Z$ in $\sC$, a composition law
\[
\circ_{X,Y,Z}:\sHom_\sC(Y,Z)^*\otimes\sHom_\sC(X,Y)^*\to \sHom_\sC(X,Z).
\]
The map $\circ_{X,Y,Z}$ is a   map of complexes; equivalently, we have the Leibniz rule:
\[
d(f\circ g)=df\circ g+(-1)^{\Deg f}f\circ dg.
\]
One has associativity of composition and an identity morphism $\id_X\in \sHom_\sC(X,X)^0$ with $d\id_X=0$.

For a DG category $\sC$, one has the additive category $Z^0\sC$, with the same objects as $\sC$ and with
\[
\Hom_{Z^0\sC}(X,Y):=Z^0\sHom_\sC(X,Y).
\]
We also have the {\em homotopy category} $H^0\sC$, the additive category with the same objects as $\sC$ and with
\[
\Hom_{H^0\sC}(X,Y):=H^0\sHom_\sC(X,Y).
\]
Clearly each DG functor $F:\sC\to \sC'$ induces functors of additive categories $H^0F:H^0\sC\to H^0\sC'$ and $Z^0F:H^0\sC\to Z^0\sC'$, making $H^0$ and $Z^0$ functors from DG categories to additive categories. 

\begin{definition} Let $\sC$ be an additive category. A {\em cubical enrichment} of $\sC$ is a functor
\[
\un{\sHom}:\sC^\op\times\sC\times\Cube^\op\to \Ab
\]
together with an associative composition law
\[
\un{\circ}_{X,Y,Z,n,m}:\un{\sHom}(X,Y,n)\otimes\un{\sHom}(Y,Z,m)\to \un{\sHom}(X,Z,n+m)
\]
such that
\begin{enumerate}
\item $\sHom(X,Y,0)=\Hom_\sC(X,Y)$. Also,  
\[
\id_X\un{\circ}_{X,Y,0,m}-=\id;\ -\un{\circ}_{X,Y,n,0}\id_Y=\id.
\]
\item The maps $\circ_{X,Y,Z,n,m}$ give rise to a map of complexes
\[
\un{\circ}_{X,Y,Z}:\un{\sHom}(X,Y)^*\otimes\un{\sHom}(Y,Z)^*\to \un{\sHom}(X,Z)^*
\]
which descends to a well-defined map of complexes
\[
\circ_{X,Y,Z}:\sHom(X,Y)^*\otimes \sHom(Y,Z)^*\to \sHom(X,Z)^*.
\]
\item The assignment $(X,Y)\mapsto :\un{\sHom}(X,Y)^*$ defines a DG category $\sC^*$ and the identity of (1) defines a functor  of DG categories $\sC\to \sC^*$.
\end{enumerate}
\end{definition}

\begin{definition}\label{def:comult}
Let $n\mapsto \square^n$ be a co-cubical object (denoted $\square^*$) of a tensor category $(\sC,\otimes)$ such that $\square^0$ is the unit object with respect to $\otimes$.

  A {\em co-multiplication} $\delta^*$ on $\square^*$ is a morphism of co-cubical objects
\[
\delta^*:\square^*\to \square^*\otimes\square^*,
\]
where $\square^*\otimes\square^*$ is the diagonal co-cubical object $n\mapsto \square^n\otimes\square^n$,  which  is
\begin{enumerate}
\item {\em co-associative}: $(\delta^*\otimes \id_{\square^*})\circ\delta^*=
(\id_{\square^*}\otimes \delta^*)\circ \delta^*$.
\item {\em co-unital}: $\square^0\xrightarrow{\delta(0)} \square^0\otimes\square^0\xrightarrow{\mu}\square^0=\id_{\square^0}$ where $\mu$ is the unit isomorphism for $\otimes$.
\item {\em symmetric}: Let $t$ be the commutativity constraint in $(\sC,\otimes)$. Then $t_{\square^*,\square^*}\circ\delta^*=\delta^*$.
\end{enumerate}
\end{definition}
Given a co-multiplication on a co-cubical object $\square^*$, define
\[
\un{\sHom}(X,Y,n):=\sHom_\sC(X\times\square^n, Y),
\]
giving us the cubical object $n\mapsto \un{\sHom}(X,Y,n)$ of $\Ab$; we denote the associated complex by $ \un{\sHom}(X,Y)^*$. The co-multiplication gives us the map of cubical objects
\[
\hat{\circ}_{X,Y,Z}: \un{\sHom}(Y,Z,*)\otimes  \un{\sHom}(X,Y,*)\to  \un{\sHom}(X,Z,*)
\] sending $f:X\otimes\square^n\to Y$ and $g:Y\otimes \square^n\to Z$ to the composition
\[
X\otimes\square^n\xrightarrow{\delta(n)}X\otimes\square^n\otimes\square^n
\xrightarrow{f\otimes\id}Y\otimes\square^n\xrightarrow{g}Z
\]
Using the cup product map \eqref{eqn:Cup}, the map $\hat{\circ}_{X,Y,Z}$ gives rise to the map of complexes
\[
\un{\circ}_{X,Y,Z}:\un{\sHom}(Y,Z)^*\otimes  \un{\sHom}(X,Y)^*\to  \un{\sHom}(X,Z)^*
\]
by
\[
\un{\circ}_{X,Y,Z}:=\hat{\circ}_{X,Y,Z}\circ\cup_{\un{\sHom}(Y,Z), \un{\sHom}(X,Y)}.
\]

The following proposition is proved by a straightforward computation.

\begin{prop} Let $\square^*:\Cube\to\sC$ be a co-cubical object in a tensor category $\sC$, with a co-multiplication $\delta$. Then $(X,Y,n)\mapsto \sHom_\sC(X\times\square^n, Y)$, with the composition law $\un{\circ}_{X,Y}$ defined above, defines a cubical enrichment of $\sC$.
\end{prop}

We denote the DG category formed by the cubical enrichment described above by $(\sC,\otimes,\square^*,\delta)$, or just $(\sC,\square^*)$ when the context makes the meaning clear.

Suppose now that, in addition to the assumptions used above, the Hom-groups in $\sC$ are $\Q$-vector spaces; we call such a category {\em $\Q$-additive}. We may then define the {\em alternating projection}
\[
\alt:\sHom_\sC(X\times\square^n, Y)\to \sHom_\sC(X\times\square^n, Y)^\alt
\]
by applying the idempotent In the rational group ring $\Q[F_n]$ corresponding to the sign representation. 

\begin{definition}  Suppose that $\sC$ is $\Q$-additive. Define the sub-DG category $(\sC,\otimes,\square^*,\delta)^\alt$ of $(\sC,\otimes,\square^*,\delta)$, with the same objects as $(\sC,\otimes,\square^*,\delta)$, and with complex of morphisms given by the subcomplex
\[
\sHom_\sC(X,Y,n)^\alt:=\sHom_\sC(X\times\square^n, Y)^\alt\subset \sHom_\sC(X,Y,n).
\]
The composition law is defined by the composition
\[
\sHom_\sC(Y,Z)^{\alt*}\otimes   \sHom_\sC(X,Y)^{\alt*}\xrightarrow{\circ}  \sHom_\sC(X,Z)^*
\xrightarrow{\alt}\sHom_\sC(X,Z)^{\alt*}
\]
\end{definition}

\begin{prop} \label{prop:AltDG} Let  $\square^*:\Cube\to\sC$ be a co-cubical object in a tensor category $\sC$, with a co-multiplication $\delta$. Suppose that $\square^*$ extends to a functor
\[
\square^*:\ECube\to \sC
\]
and that $\sC$ is $\Q$-additive. Then the natural inclusion functor
\[
(\sC,\square^*)^\alt\to (\sC,\square^*)
\]
is a homotopy equivalence of DG categories, i.e., for each pair of objects $X,Y\in\Obj(\sC,\square^*)^\alt =\Obj(\sC,\square^*)$, the inclusion
\[
\sHom(X,Y)^{\alt*}\to \sHom(X,Y)^*
\]
is a quasi-isomorphism (see definition~\ref{definition:DGHEquiv} for the general case).
\end{prop}

\begin{proof} This follows immediately from proposition~\ref{prop:Alt}.
\end{proof}

\subsection{Tensor structure} There is a natural tensor structure on the DG category $(\sC,\otimes,\square^*,\delta)^\alt$, which we now describe.

Given $f:X\times\square^n\to Y$, $f':X'\times\square^{n'}\to Y'$, define
\[
f\tilde{\otimes} g:X\otimes X'\otimes\square^{n+n'}\to Y\otimes Y'
\]
as the composition
\begin{multline*}
X\otimes X'\otimes\square^{n+n'}\xrightarrow{\id\otimes\delta_{n,n'}}
X\otimes X'\otimes\square^n\otimes\square^{n'}\\\xrightarrow{\tau_{X',\square^n}}
X\otimes \square^n\otimes X'\otimes \square^{n'}
\xrightarrow{f\otimes f'} Y\otimes Y'.
\end{multline*}
Assuming that $\sC$ is $\Q$-additive, we define $f\otimes_\square g$ by applying the alternating projection:
\[
f\otimes_\square g:=(f\tilde{\otimes} g)\circ(\id_{X\otimes X'}\otimes\alt_{n+n'}).
\]

\begin{prop} \label{prop:tensor}Let $\sC$ be a $\Q$-tensor category, $\square^*$ a co-cubical object of $\sC$ and $\delta:\square\to\square\otimes\square$ a co-multiplication. Then $((\sC,\otimes,\square^*,\delta)^\alt, \otimes_\square)$ is a   DG tensor category, with commutativity constraints induced by the commutativity constraints in $\sC$.
\end{prop}

\begin{proof} One checks easily that the integral operation  $\tilde{\otimes}$ satisfies the  Leibniz rule:
\[
\partial(f\tilde{\otimes}g)=\partial f\tilde{\otimes}g+(-1)^{\Deg f}f\otimes\partial g.
\]

Let $\delta_{n,m}:\square^{n+m}\to \square^n\otimes\Delta^m$ denote the composition $(p_1^{n,m}\otimes p_2^{n,m})\circ\delta^{n+m}$.  Let $\pi_\Alt^N:\square^N\to\square^N$ be the map induced by the alternating idempotent in $\Q[F_N]$.

It follows from the properties of co-associativity and symmetry of $\delta$, together with the fact that $\delta$ is a map of co-cubical objects, that
\begin{multline*}
(\id_{\square^n}\otimes t_{m,n'}\otimes\id_{\square^{m'}})(\delta_{n,m}\otimes\delta_{n',m'})\circ\delta_{n+m,n'+m'}\\
=
(\delta_{n,n'}\otimes\delta_{m,m'})\circ\delta_{n+n',m+m'}\circ \square(\id_{\un{n}}\times\tau_{m,n'}\times\id_{m'}),
\end{multline*}
where $\tau_{m,n'}:\un{m}\times\un{n'}\to \un{n'}\times\un{m}$ is the symmetry in $\Cube$, and $t_{**}$ is the symmetry in $\sC$. Composing on the right with $\pi_\Alt^{n+m+n'+m'}$ yields the identity
\begin{multline*}
(\id_{\square^n}\otimes t_{m,n'}\otimes\id_{\square^{m'}})(\delta_{n,m}\otimes\delta_{n',m'})\circ\delta_{n+m,n'+m'}\circ\pi_\Alt^{n+m+n'+m'}\\
=
(-1)^{mn'}(\delta_{n,n'}\otimes\delta_{m,m'})\circ\delta_{n+n',m+m'}\circ \pi_\Alt^{n+m+n'+m'}.
\end{multline*}
The identity
\[
(f\otimes_\square g)\circ(f'\otimes_\square g')=(-1)^{\Deg g\Deg f'}ff'\otimes_\square gg'
\]
follows directly from this.

One shows by a similar argument that, for $f\in\Hom_\sC(X\otimes\square^p,Y)^\Alt$, $g\in\Hom_\sC(X'\otimes\square^q,Y')^\Alt$, we have
\[
t_{Y,Y'}\circ(f\otimes_\square g)=(-1)^{pq}(g\otimes_\square f)\circ (t_{X,X'}\otimes\id_{\square^{p+q}}),
\]
completing the proof.
\end{proof}

\begin{rem}\label{rem:action} One could hope that the operations $\tilde{\otimes}$ define at least a monoidal structure on $(\sC,\otimes,\square^*,\delta)$, but this is  in general not the case. In fact, as we noted in the proof of proposition~\ref{prop:tensor}, sending $f,g$ to $f\tilde{\otimes}g$ does satisfy the Leibniz rule, but we do not have the identity
\[
(f\tilde{\otimes}g)\circ(f'\tilde{\otimes}g')=(-1)^{\Deg f'\Deg g}ff'\tilde{\otimes}gg'
\]
in general: the cubes on the two sides of this equation  are in a different order.

In spite of this, the operation $\otimes$ {\em does} extend to an action of $\sC$ on $(\sC,\otimes,\square^*,\delta)$. Consider $\sC$ as a DG category with all morphisms of degree zero (and zero differential). Define
\[
\otimes:\sC\otimes (\sC,\otimes,\square^*,\delta)\to (\sC,\otimes,\square^*,\delta)
\]
to be the same as $\otimes$ on objects, and on morphisms by 
\[
\otimes:\Hom_\sC(A,B)\otimes\Hom_\sC(X\otimes\square^n,Y)\to
\Hom_\sC(A\otimes X\otimes\square^n,B\otimes Y)
\]
\end{rem}

\section{Complexes over a DG category} We review a version of Kapranov's construction \cite{Kapranov} of complexes over a DG category.

\subsection{The category $\PTR(\sC)$}  

\begin{definition} Let $\sC$ be a DG category. The DG category $\PTR(\sC)$ has objects  $\sE$ consisting of the following data:
\begin{enumerate}
\item A finite collection of objects of $\sC$, $\{E_i, N\le i\le M\}$ ($N$ and $M$ depending on $\sE$).
\item Morphisms $e_{ij}:E_j\to E_i$ in $\sC$  of degree $j-i+1$, for $N\le j,i\le M$, satisfying
\[
(-1)^ide_{ij}+\sum_ke_{ik}e_{kj}=0.
\]
\end{enumerate}
For $\sE:=\{E_i, e_{ji}\}$, $\sF:=\{F_i, f_{ji}\}$, a morphism $\phi:\sE\to \sF$ of degree $n$ is a collection of morphisms $\phi_{ij}:E_j\to F_i$ in $\sC$, such that   $\phi_{ij}$ has degree $n+j-i$. The composition of morphisms $\phi:\sE\to \sF$, $\psi:\sF\to \sG$ is defined by
\[
(\psi\circ\phi)_{ij}:=\sum_k\psi_{ik}\circ\phi_{kj}.
\]

Given a morphism $\phi:\sE\to \sF$ of degree $n$, define
\[
\partial_{\sF,\sE}(\phi)\in\Hom_{\PTR(\sC)}(\sE,\sF)^{n+1}
\]
to be the collection $\partial_{\sF,\sE}(\phi)_{ij}:E_j\to F_i$ with
\[
\partial_{\sF,\sE}(\phi)_{ij}:=(-1)^id(\phi_{ij})+\sum_kf_{ik}\phi_{kj}-(-1)^n\sum_k\phi_{ik}e_{kj}.
\]
\end{definition}

\begin{rem} We take the opportunity to give a detailed, although tedious, verification that $\PTR(\sC)$ is indeed a DG category, filling a much needed gap in the literature.
\end{rem}

\begin{lem} $\partial_{\sF,\sE}^2=0$.
\end{lem}

\begin{proof} Take $\phi:=\{\phi_{ij}:E_j\to F_i\}$ of degree $n$. Then
\[
\partial^2(\phi)_{ij}=(-1)^id(\partial\phi)_{ij}+\sum_kf_{ik}(\partial\phi)_{kj}-(-1)^{n+1}
\sum_k(\partial\phi)_{ik}e_{kj}.
\]
We have
\begin{align*}
(-1)^id(\partial\phi)_{ij}&=(-1)^id[(-1)^id\phi_{ij}+\sum_lf_{il}\phi_{lj}-(-1)^n\phi_{il}e_{lj}]\\
&=(-1)^i\sum_l[df_{il}\phi_{lj}+(-1)^{l-i+1}f_{il}d\phi_{lj}-(-1)^nd\phi_{il}e_{lj}\\
&\hskip 150pt-(-1)^{l-i}\phi_{il}de_{lj}]\\
&=-\sum_{l,k}f_{ik}f_{kl}\phi_{lj}-\sum_l(-1)^{l}f_{il}d\phi_{lj}-(-1)^{n+i}\sum_ld\phi_{il}e_{lj}\\
&\hskip 150pt+
\sum_{l,k}\phi_{il}e_{lk}e_{kj},
\end{align*}
\[
\sum_kf_{ik}(\partial\phi)_{kj}=\sum_k(-1)^kf_{ik}d\phi_{kj}+\sum_{k,l}f_{ik}f_{kl}\phi_{lj}
-(-1)^n\sum_{k,l}f_{ik} \phi_{kl}e_{lj},
\]
and
\[
-(-1)^{n+1}\sum_k(\partial\phi)_{ik}e_{kj}=
(-1)^{n+i}\sum_k d\phi_{ik}e_{kj}+(-1)^n\sum_{k,l}f_{il}\phi_{lk}e_{kj}-\sum_{l,k}\phi_{il}e_{lk}e_{kj},
\]
which proves the result.
\end{proof}

\begin{lem} For $\phi:\sE\to \sF$, and $\psi:\sF\to\sG$, we have
\[
\partial_{\sE,\sG}(\psi\circ\phi)=\partial_{\sF,\sG}(\psi)\circ\phi+(-1)^{\Deg\psi}\psi\circ \partial_{\sE,\sF}(\phi).
\]
\end{lem}

\begin{proof} Write $\sE=\{E_i, e_{ij}\}$, $\sF=\{F_i, f_{ij}\}$, $\sG=\{G_i, g_{ij}\}$, and suppose $\phi$ has degree $n$ and $\psi$ has degree $m$. Then
\[
\partial_{\sE,\sG}(\psi\circ\phi)_{ij}=(-1)^id(\psi\circ\phi)_{ij}+\sum_kg_{ik}(\psi\circ\phi)_{kj}-(-1)^{n+m}\sum_k(\psi\circ\phi)_{ik}e_{kj}.
\]
We have
\begin{align*}
(-1)^id(\psi\circ\phi)_{ij}&=(-1)^id(\sum_l\psi_{il}\phi_{lj})\\
&=\sum_l(-1)^id\psi_{il}\phi_{lj}+(-1)^{l+m}\psi_{il} d\phi_{lj}.
\end{align*}
Thus
\begin{align*}
\partial_{\sE,\sG}(\psi\circ\phi)_{ij}&=
\sum_l(-1)^id\psi_{il}\phi_{lj}+\sum_{k,l}g_{ik}\psi_{kl}\phi_{lj}-(-1)^m\sum_{k,l} \psi_{ik}f_{kl}\phi_{lj}\\
&\hskip 20pt+(-1)^m\sum_{k,l} \psi_{ik}f_{kl}\phi_{lj}+
\sum_l(-1)^{l+m}\psi_{il}d\phi_{lj}\\
&\hskip 40pt -(-1)^{n+m}\sum_{k,l}\psi_{il}\phi_{lk}e_{kj}\\
&=\sum_l[(-1)^id\psi_{il} +\sum_{k}g_{ik}\psi_{kl} -(-1)^m\sum_{k} \psi_{ik}f_{kl}]\phi_{lj}\\
&\hskip 20pt +(-1)^m\sum_l\psi_{il}[(-1)^{l}d\phi_{lj}+\sum_{k} f_{lk}\phi_{kj}
 -(-1)^{n}\sum_{k}\phi_{lk}e_{kj}]\\
&=\sum_l(\partial\psi)_{il}\phi_{lj}+(-1)^m\sum_l\psi_{il}(\partial\phi)_{lj}\\
&=[\partial_{\sF,\sG}(\psi)\circ\phi+(-1)^{\Deg\psi}\psi\circ \partial_{\sE,\sF}(\phi)]_{ij}.
\end{align*}
\end{proof}

With the help of these two lemmas, it is straightforward to show that $\PTR(\sC)$ is a DG category.

\begin{rem}\label{rem:additive} Suppose that $\sC$ is an additive category, which we consider as a DG category with all morphisms in degree zero, and zero differential. Let $\sE=\{E_i, e_{ij}\}$ be in $\PTR(\sC)$. Since $e_{ij}$ has degree $j-i+1$, the forces $e_{ij}=0$ unless $i=j+1$. Writing $d^j:=e_{j+1,j}$, the condition $(-1)^ide_{ij}+\sum_ke_{ik}e_{kj}=0$ is just $d^{j+1}d^j=0$, so $\sE$ is just a complex, in the usual sense, with differential $d^j:E_j\to E_{j+1}$.

Similarly,  $\Hom(\sE,\sF)^p$ reduces to the usual  group of degree $p$ maps of complexes: $\phi=\{\phi^n=\phi_{n,n}\}$ and the differential is the standard one:
\[
\partial\phi=d_\sF\circ\phi-(-1)^n\phi\circ d_\sE.
\]
Thus, $\PTR(\sC)=C^b(\sC)$, the category of bounded complexes in the additive category $\sC$.
\end{rem}

\begin{rem}\label{rem:Nonpos} Let $\sC$ be a DG category. Call $\sC$ {\em non-positive} if $\Hom_\sC(X,Y)^p=0$ for $p>0$. If $\sC$ is non-positive and $\{E_i, e_{ij}\}$ is in $\PTR(\sC)$, then the degree restriction forces $e_{ij}=0$ unless $j<i$. Similarly, if $\phi:\sE\to \sF$ is a degree $p$ morphism in  $\PTR(\sC)$, and we decompose $\phi$ into its components $\phi_{ij}:\sE_j\to \sF_i$, then $i\ge j+p$ for $\phi_{ij}\neq0$. In particular, for a degree zero morphisms $\phi$, we have $\phi_{ij}=0$ if $i<j$. Bondarko has used these properties to define a {\em weight structure} on   $\PTR(\sC)$; concretely, the conditions imposed on the objects and morphisms in $\PTR(\sC)$ by the non-positivity of $\sC$ allow one to define the ``stupid truncation" as a functor on $\PTR(\sC)$. See \S\ref{sec:Truncation} for details.
\end{rem}

For $E$ in $\sC$, we have the object $i_0(E)$ with $i_0(E)_0=E$, $i_0(E)_i=0$ for $i\neq0$ and $e_{00}=0$. For a degree $p$ morphism $f:E\to F$ in $\sC$, define $i_0(f):i_0(E)\to i_0(F)$ by $i_0(f)_{00}=f$. This defines a fully faithful embedding of DG categories
\[
i_0:\sC\to \PTR(\sC).
\]

\subsection{Translation and cone sequence}
For $\sE=\{E_i, e_{ij}:E_j\to E_i\}$ in $\PTR(\sC)$, and $n\in\Z$, define $\sE[n]$ by
\[
(\sE[n])_i:=E_{i+n},\ e[n]_{ij}:=(-1)^ne_{i+n,j+n}:(\sE[n])_j\to (\sE[n])_i.
\]
For $\phi\in\Hom(\sE,\sF)^p$, define $\phi[n]\in\Hom(\sE[n],\sF[n])^p$ by setting
\[
\phi[n]_{ij}:(\sE[n])_j\to (\sF[n])_i
\]
equal to $(-1)^{np}\phi_{i+n,j+n}$. It follows directly from the definitions that $(\sE,\phi)\mapsto 
(\sE[n],\phi[n])$ defines a DG autoisomorphism   with inverse   the shift by $-n$.

Let $\phi:\sE\to \sF$ be a degree 0 morphism in $\PTR(\sC)$ with $\partial \phi=0$. We define the {\em cone} of $\phi$, $\Cone(\phi)$, as the object in $\PTR(\sC)$ with
\[
\Cone(\phi)_i:=F_i\oplus E_{i+1}
\]
and with morphisms $\Cone(\phi)_{ij}:\Cone(\phi)_j\to \Cone(\phi)_i$ given by the usual matrix
\[
\begin{pmatrix}
f_{ij}&0\\
\phi_{ij}&-e_{ij}
\end{pmatrix}
\]
The inclusions $F_i\to F_i\oplus E_{i+1}$ define the morphism $i_\phi:\sF\to \Cone(\phi)$ and the projections $F_i\oplus E_{i+1}\to E_{i+1}$ define the morphism $p_\phi:\Cone(\phi)\to \sE[1]$. Note that $\partial(i_\phi)=0=\partial(p_\phi)$. This gives us the {\em cone sequence} in $Z^0\PTR(\sC)$:
\[
\sE\xrightarrow{\phi}\sF\xrightarrow{i_\phi}\Cone(\phi)\xrightarrow{p_\phi}\sE[1].
\]

\subsection{Tensor structure} 
Now suppose that $(\sC,\otimes)$  is a  DG tensor category. We give $\PTR(\sC)$ a tensor structure as follows. On objects $\sE=\{E_i, e_{ij}\}$, $\sF=\{F_i, f_{ij}\}$, let $\sE\otimes \sF$ be the object with terms 
\[
(\sE\otimes\sF)_i:=\oplus_kE_k\otimes F_{i-k}
\]
and maps $g_{ij}:(\sE\otimes\sF)_j\to (\sE\otimes\sF)_i$ given by the sum of the maps  
\[
(-1)^{(l-k+1)n}e_{kl}\otimes\id_{F_n}\text{ or } (-1)^{k }\id_{E_k}\otimes  f_{mn}.
\]

\begin{lem} Given $\sE=\{E_i, e_{ij}\}$, $\sF=\{F_i, f_{ij}\}$ in $\PTR(\sC)$, the collection of objects 
$(\sE\otimes\sF)_i:=\oplus_kE_k\otimes F_{i-k}$ and maps $g_{ij}:(\sE\otimes\sF)_i\to
(\sE\otimes\sF)_i$ the direct sums of the maps 
$(-1)^{(l-k+1)n}e_{kl}\otimes\id_{F_n}$ and $(-1)^{k}\id_{E_k}\otimes  f_{mn}$ does in fact define an object of $\PTR(\sC)$.
\end{lem}

\begin{proof} We compute: Since
\[
(-1)^{(l-k+1)n+n}=(-1)^{(l-q+1)n}\cdot (-1)^{(q-k+1)n}
\]
we have
\begin{multline*}
d((-1)^{(l-k+1)n}e_{kl}\otimes\id_{F_n})=(-1)^{(l-k+1)n+k+1}\sum_qe_{kq}e_{ql}\otimes \id_{F_n}\\
=(-1)^{n+k+1}\sum_q[(-1)^{(q-k+1)n}e_{kq}\otimes \id_{F_n}]\circ 
[(-1)^{(l-q+1)n}e_{ql}\otimes \id_{F_n}].
\end{multline*}
Similarly, 
\[
d((-1)^{k}\id_{E_k}\otimes  f_{mn})=(-1)^{k+m+1}\sum_q[(-1)^{k}\id_{E_k}\otimes  f_{mq})]\circ
[(-1)^{k}\id_{E_k}\otimes  f_{qn})].
\]
Finally, there are two compositions of terms of the form $(-1)^{l-k+1)n}e_{kl}\otimes\id_{F_n}$ and $(-1)^{k}\id_{E_k}\otimes  f_{mn}$ that define maps from $E_l\otimes F_n$ to $E_k\otimes F_m$ with $k>l$ and $m>n$, namely
\[
[(-1)^{(l-k+1)m}e_{kl}\otimes\id_{F_m}]\circ [(-1)^{l}\id_{E_l}\otimes  f_{mn}]
\]
and
\[
[(-1)^{k}\id_{E_k}\otimes  f_{mn}]\circ [(-1)^{(l-k+1)n}e_{kl}\otimes\id_{F_n}].
\]
These are both $\pm e_{kl}\otimes f_{mn}$, the first one having sign $(-1)^{(l-k+1)m+l}$ and the second having sign $(-1)^{(l-k+1)n+k+(m-n+1)(l-k+1)}$. As these two signs are opposite, these terms cancel. These three computations yield the identity
\[
dg_{ij}=(-1)^{i+1}\sum_qg_{iq}g_{qj}
\]
as desired.
\end{proof}

We also need to define a cup product map
\[
\cup:\Hom(\sE,\sF)\otimes\Hom(\sE',\sF')\to \Hom(\sE\otimes\sE',\sF\otimes \sF')
\]
For this, suppose $\sE=\{E_i, e_{ij}\}$, $\sF=\{F_i, f_{ij}\}$, $\sE=\{E'_i, e'_{ij}\}$, $\sF'=\{F'_i, f'_{ij}\}$. Given $\phi_{ij}:E_j\to E'_i$ of degree $j-i+p$, $\psi_{kl}:F_l\to F'_k$ of degree $l-k+q$ (the degree taken in $\sC$), define
\[
\phi_{ij}\cup\psi_{kl}:=(-1)^{(j-i+p)k+qj}\phi_{ij}\otimes\psi_{kl}:E_j\otimes F_l\to E'_i\otimes F'_k.
\]
Taking the sum over all components defines the graded map
\[
\cup:\Hom_{\PTR(\sC)}(\sE,\sF)\otimes \Hom_{\PTR(\sC)}(\sE',\sF')\to 
 \Hom_{\PTR(\sC)}(\sE\otimes\sE',\sF\otimes \sF').
\]

\begin{lem} The map $\cup$ is a map of complexes.
\end{lem}

\begin{proof} Since the structure maps in $\sE\otimes\sE'$ and $\sF\otimes \sF'$ both involve two different types of maps, we have in total five terms in the differential in the Hom-complex
$\Hom_{\PTR(\sC)}(\sE\otimes\sE',\sF\otimes \sF')$. As the computation is rather involved, we handle each term separately. We compute with given morphisms $\phi=\{\phi_{ij}\}\in\Hom(\sE,\sF)^p$ and 
$\psi=\{\psi_{kl}\}\in\Hom(\sE,\sF)^q$. We denote with subscript $ij,kl$ the component of a map $\sE\otimes\sF\to \sE'\otimes \sF'$ that maps $E_j\otimes F_l$ to $E'_i\otimes F'_k$.

First we look at the terms involving the differential in $\sC$. In the $ij,kl$-component, this is
\begin{align*}
(-1)^{i+k}d[\phi_{ij}&\cup\psi_{kl}]\\
&=(-1)^{i+k+(j-i+p)k+qj}d(\phi_{ij}\otimes \psi_{kl})\\
&=(-1)^{i+k+(j-i+p)k+qj}[d\phi_{ij}\otimes \psi_{kl}+(-1)^{j-i+p}\phi_{ij}\otimes d\psi_{kl}]\\
&=(-1)^{i+k+(j-i+p)k+qj}(-1)^{j-i+p+1)k+qj}d\phi_{ij}\cup \psi_{kl}\\
&\hskip 20pt+
(-1)^{i+k+(j-i+p)k+qj+j-i+p}(-1)^{(j-i+p)k+(q+1)j}\phi_{ij}\cup d\psi_{kl}\\
&=[(-1)^id\phi_{ij}]\cup\psi_{kl}+(-1)^p\phi_{ij}\cup[(-1)^kd\psi_{kl}].
\end{align*}

Next, the terms involving composition on the right with the maps $f_{**}$.  In the $ij,kl$-component, this is a sum of terms of the form
\begin{align*}
-(-1)^{p+q}\phi_{ij}\cup\psi_{kl'}&\circ((-1)^j\id_{E_j}\otimes f_{l'l})\\
&=
(-1)^{(j-i+p)k+(q+1)j+p+q}(\phi_{ij}\otimes\psi_{kl'})\circ(\id_{E_j}\otimes f_{l'l})\\
&=-(-1)^{(j-i+p)k+(q+1)j+p+q}\phi_{ij}\otimes(\psi_{kl'}\circ  f_{l'l})\\
&=-(-1)^{p+q}\phi_{ij}\cup(\psi_{kl'}\circ  f_{l'l})\\
&=(-1)^p[\phi_{ij}\cup(-(-1)^q\psi_{kl'}\circ f_{l'l})].
\end{align*}
The $ij,kl$-component of the terms  involving composition on the right with the maps $e_{**}$ are
\begin{align*}
-(-1&)^{p+q}(\phi_{ij'}\cup\psi_{kl})\circ((-1)^{j-j'+1)l}e_{j'j}\otimes\id_{F_l})\\
&=
-(-1)^{(j'-i+p)k+qj'+(j-j'+1)l+p+q}(\phi_{ij'}\otimes\psi_{kl})\circ((-1)^{j-j'+1)l}e_{j'j}\otimes\id_{F_l})\\
&=-(-1)^{(j'-i+p)k+qj'+(j-j'+1)l+p+q}(-1)^{(l-k+q)(j-j'+1)}\phi_{ij'}e_{j'j}\otimes\psi_{kl}\\
&=-(-1)^{(j'-i+p)k+qj'+ p+q+(k+q)(j-j'+1)}\phi_{ij'}e_{j'j}\otimes\psi_{kl}\\
&=-(-1)^{(j'-i+p)k+qj'+ p+q+(k+q)(j-j'+1)}(-1)^{(j-i+p+1)k+qj}\phi_{ij'}e_{j'j}\cup\psi_{kl}\\
&=[-(-1)^p\phi_{ij'}e_{j'j}]\cup\psi_{kl}.
\end{align*}

For left composition with the maps $f'_{**}$, the $ij,kl$-component is a sum of terms of the form
\begin{align*}
((-1)^i&\id_{E_i}\otimes f'_{kl'})\circ(\phi_{ij}\cup\psi_{l'k})\\
&=(-1)^{(j-i+p)l'+qj+i}(\id_{E_i}\otimes f'_{kl'})\circ(\phi_{ij}\otimes\psi_{l'k})\\
&=(-1)^{(j-i+p)l'+qj+i}(-1)^{(l'-k+1)(j-i+p)}\phi_{ij}\otimes f'_{kl'}\circ \psi_{l'k}\\
&=(-1)^{(j-i+p)l'+qj+i}(-1)^{(l'-k+1)(j-i+p)}(-1)^{(j-i+p)k+(q+1)j}\phi_{ij}\cup f'_{kl'}\circ \psi_{l'k}\\
&=(-1)^p\phi_{ij}\cup f'_{kl'}\circ \psi_{l'k}.
\end{align*}
Finally, left composition with the maps $e'_{**}$ has $ij,kl$-component  a sum of terms of the form
\begin{align*}
(-1)^{j'-i+1}(e'_{ij'}\otimes\id_{F_k})\circ(\phi_{j'j}\cup\psi_{kl})
&=(-1)^{j'-i+1}(-1)^{(j-j'+p)k+qj}e_{ij'}\circ\phi_{j'j}\otimes \psi_{kl}\\
&=(-1)^{(j-i+p+1)k+qj}e_{ij'}\circ\phi_{j'j}\otimes \psi_{kl}\\
&=e_{ij'}\phi_{j'j}\cup\psi_{kl}.
\end{align*}

Putting these five computations together gives
\[
\partial_{\sE\otimes\sF,\sE'\otimes\sF'}(\phi\cup\psi)=
\partial_{\sE,\sF}(\phi)\cup\psi+(-1)^{\Deg\phi}\phi\cup\partial_{\sE',\sF'}(\psi),
\]
as claimed.
\end{proof}

\begin{lem} We have $(\phi\cup\psi)\circ(\phi'\cup\psi')=(-1)^{\Deg\psi\cdot\Deg\phi'}(\phi\circ\phi')\cup
(\psi\circ\psi')$.
\end{lem}

\begin{proof} It suffices to check for $\phi$, $\phi'$, $\psi$ and $\psi'$ each consisting of a single component. Suppose  $\phi$, $\phi'$, $\psi$ and $\psi'$ have respective degrees $p,p',q$ and $q'$. Then
\begin{align*}
(\phi_{im}\cup\psi_{kn})&\circ(\phi'_{mj}\cup\psi'_{nl})\\
&=(-1)^{(m-i+p)k+qm+(j-m+p')n+q'j}(\phi_{im}\otimes\psi_{kn})\circ(\phi'_{mj}\otimes\psi'_{nl})\\
&= (-1)^{(m-i+p)k+qm+(j-m+p')n+q'j+(n-k+q)(j-m+p')}\phi_{im}\phi'_{mj}\otimes\psi_{kn}\psi'_{nl}\\
&=(-1)^{(j-i+p+p')k+(q+q')j+qp'}\phi_{im}\phi'_{mj}\otimes\psi_{kn}\psi'_{nl}\\
&=(-1)^{qp'}\phi_{im}\phi'_{mj}\cup\psi_{kn}\psi'_{nl}.
\end{align*}
\end{proof}

Finally, we need a commutativity constraint. Let $t_{E,F}:E\otimes F\to F\otimes E$ be the commutativity constraint in $\sC$. Define
\[
\tau_{\sE,\sF}:\sE\otimes \sF\to \sF\otimes \sE
\]
as the collection of maps
\[
(-1)^{jl}t_{E_i,F_l}:E_i\otimes F_l\to F_l\otimes E_i.
\]

\begin{lem} 1. The collection of maps $\{(-1)^{jl}t_{E_i,F_l}\}$ does in fact define a morphism
$\tau_{\sE,\sF}:\sE\otimes \sF\to \sF\otimes \sE$\\
\\
2. For $\phi\in\Hom_{\PTR(\sC)}(\sE,\sE')^p$, 
$\psi\in \Hom_{\PTR(\sC)}(\sF,\sF')^q$, we have
\[
\psi\cup \phi\circ \tau_{\sE,\sF}=(-1)^{pq}\tau_{\sE',\sF'}\circ \phi\cup\psi.
\]
3. We have $\tau_{\sE,\sF\otimes\sG}=(\id_\sF\cup \tau_{\sE,\sG})\circ(\tau_{\sE,\sF}\cup\id_\sG)$.
\end{lem}

\begin{proof} The first assertion follows from the identities
\begin{align*}
&((-1)^l\id_{F_l}\otimes e_{ij})\circ((-1)^{jl}t_{E_j,F_l})=((-1)^{il}t_{E_i,F_l})\circ ((-1)^{(j-i+1)l}e_{ij}\otimes \id_{F_l})\\
&((-1)^{(l-k+1)j}f_{kl}\otimes \id_{E_j})\circ((-1)^{jl}t_{E_j,F_l})=
((-1)^{jk}t_{E_j,F_k})\circ((-1)^j\id_{E_j}\otimes f_{kL}).
\end{align*}
For the second, we have 
\[
(-1)^{pq}\tau \circ \phi_{ij}\cup\psi_{kl}=(-1)^{pq}(-1)^{ik}(-1)^{(j-i+p)k+qj}t_{E_i,F_k}\circ(\phi_{ij}\otimes\psi_{kl})
\]
while
\begin{align*}
(\psi_{kl}\cup\phi_{ij})\circ\tau&=(-1)^{(l-k+q)i+pl}(-1)^{jl}(\psi_{kl}\otimes\phi_{ij})\circ t_{E_j,F_l}\\
&=(-1)^{(l-k+q)i+pl+jl}(-1)^{(j-i+p)(l-k+q)}t_{E_i,F_k}\circ(\phi_{ij}\otimes\psi_{kl})
\end{align*}
The result follows easily from this.

Finally, since the maps $\tau_{F_i,G_j}$ send $F_i\otimes G_j$ to $G_j\otimes F_i$, we have 
\[
\id_{F_i}\cup (\tau_{\sE,\sG})_{ik}=(-1)^{ik}\id_{F_i}\otimes t_{E_i,G_k}
\] and similarly
\[
(\tau_{\sE,\sF})_{ij}\cup\id_{G_k}=(-1)^{ij}t_{E_i,F_j}\otimes\id_{G_k},\
(\tau_{\sE,\sF\otimes\sG})_{ijk}=(-1)^{i(j+k)}t_{E_i,F_j\otimes G_k}.
\]
Since
\[
t_{E_i,F_j\otimes G_k}=(\id_{F_k}\otimes t_{E_i,G_k})\circ (t_{E_i,F_j}\otimes\id_{G_k}),
\]
this proves (3).
\end{proof}

We have shown
\begin{prop} Let $(\sC,\otimes)$ be a DG tensor category. Then the operations $(\sE,\sF)\mapsto \sE\otimes\sF$, $(\phi,\psi)\mapsto \phi\cup\psi$ and commutativity constraints $\tau_{**}$ makes $\PTR(\sC)$ into a DG tensor category. 
\end{prop}

\begin{rem} The tensor structure is compatible with the translation functor, in that one has {\em identities}
\begin{align*}
&\sE[1]\otimes \sF=(\sE\otimes\sF)[1]\\
&\phi[1]\cup\psi=(\phi\cup\psi)[1]
\end{align*}
Using the commutativity constraint gives us a canonical isomorphism
\[
\sE\otimes\sF[1]\xrightarrow{\tau_{\sF,\sE}[1]^{-1}\circ\tau_{\sE,\sF[1]}}(\sE\otimes\sF)[1]
\]
\end{rem}

\begin{rem} Let $\sC$ be a tensor category, which we consider as a DG tensor category with all morphisms in degree zero. Then $\PTR(\sC)$ is equal to the category of bounded complexes $C^b(\sC)$, and the tensor structure we have defined on $\PTR(\sC)$ is the standard one on $C^b(\sC)$ arising from the tensor structure on $\sC$.
\end{rem}

\begin{rem} If $\sC$ is a DG tensor category, the DG functor $i_0:\sC\to\PTR(\sC)$ is a tensor functor.
\end{rem}

\subsection{The category $K^b(\sC)$} 

\begin{definition} Let $\sC$ be a DG category. \\
\\
1. The DG category $C^b(\sC)$ is the smallest full subcategory of $\PTR(\sC)$ containing $i_0(\sC)$, and closed under translation, isomorphism and the operation $u\mapsto \Cone(u)$.\\
\\
2. The additive category $K^b(\sC)$ to be the homotopy category of $\PTR(\sC)$:
\[
K^b(\sC):=H^0\PTR(\sC).
\]
\end{definition}

\begin{rem} \label{rem:Cb1} 
Suppose that $\sC$ is a DG tensor category. Then $C^b(\sC)$ is a tensor subcategory of $\PTR(\sC)$. 
\end{rem}

\begin{thm} \label{thm:DGMain} Let $\sC$ be a DG category.\\
\\
1. The translation functor on $C^b(\sC)$ induces a translation functor  on $K^b(\sC)$. Defining a distinguished triangle in $K^b(\sC)$ to be a triangle that is isomorphic to the image of a cone sequence in $C^b(\sC)$ makes $K^b(\sC)$ into a triangulated category.\\
\\
2. If $\sC$ is a DG tensor category, the induced tensor structure on $C^b(\sC)$ defines a tensor structure on $K^b(\sC)$, making $K^b(\sC)$ a tensor triangulated category.\\
\\
3. Let $F:\sC\to \sC'$ be a DG functor. Then $K^bF:K^b(\sC)\to K^b(\sC')$ is exact. If $F$ is a DG tensor functor of DG tensor categories, then $K^bF$ is an exact tensor functor.
\end{thm}

The reader may find proofs of these facts in \cite[Part 2, Chap. II, proposition 2.1.6.4, proposition 2.1.7 and  theorem 2.2.2]{MixMot}. The context in \cite{MixMot} is slightly different, as there we work with DG categories having a {\em translation structure}. This is a device that allows one to avoid keeping explicit track of all the signs that turn up in the approach used here. However, the proofs cited from  \cite{MixMot} go through without any essential change to prove theorem~\ref{thm:DGMain}.

\begin{rem} If $\sC$ is an additive category, then $K^b(\sC)$ is the usual homotopy category of the category of bounded complexes over $\sC$.
\end{rem}

\subsection{Unbounded complexes} One can extend the construction of $\PTR(\sC)$ by relaxing the finiteness condition, if one imposes a ``local finiteness" condition for the structure maps $e_{ij}$ and morphisms $\phi_{ij}$. 

\begin{definition} For a DG category $\sC$, let $\PTR^\infty(\sC)$ be the DG category with objects
\[
E=\{E_i, i\in\Z, e_{ij}:E_j\to E_i,\ \Deg e_{ij}=j-i+1\}
\]
such that
\begin{enumerate}
\item There is an integer $N_E$ such that  $e_{ij}=0$  if $j-i+1>N_E$.
\item For each $i,j$, $(-1)^ide_{ij}+\sum_ke_{ik}e_{kj}=0$.
\end{enumerate}
Note that (1) implies that the sum in (2) is finite.

For $E=\{E_i, e_{ij}\}$, $F=\{F_i, i\in\Z, f_{ij}\}$ and $n\in\Z$, let $\sHom(E,F)^n$ be the subset of the product $\prod_{i,j}\Hom_\sC(E_j,F_i)^{j-i+n}$ consisting of collections $\phi_{ij}:E_j\to F_i$ such that there is an integer $N_\phi$ such that  $\phi_{ij}=0$ for $j-i+n>N_\phi$.

Make $\sHom(E,F)^*$ a complex by defining the differential $\partial_{F,E}$ as
\[
\partial_{F,E}(\phi)_{ij}:=d_\sC(\phi_{ij})+\sum_kf_{ik}\phi_{kj}-(-1)^{\Deg\phi}\sum_k\phi_{ik}e_{kj}
\]
Note that the sums are all finite and $N_{\partial\phi}\le \max(N_E,N_F,1)+N_\phi$, so in particular, the differential is well-defined.

The composition law
\[
\sHom(F,G)^*\otimes\sHom(E,F)^*\to\sHom(E,G)^*
\]
is given by
\[
(\psi\circ\phi)_{ij}=\sum_k\psi_{ik}\circ\phi_{kj}
\]
Note that sum is finite and    $N_{\psi\circ\phi}\le N_\psi+N_\phi$, so the composition law is well-defined.
\end{definition}

The fact that $\PTR^\infty(\sC)$ is a DG category follows by verifying the same identities that show that $\PTR(\sC)$ is a DG category. Similarly, the expressions for the translation functor and cone sequence in $\PTR(\sC)$ extend directly to $\PTR^\infty(\sC)$.

\begin{definition} Let $\PTR^+(\sC)$ be the full DG subcategory of $\PTR^\infty(\sC)$ with objects those $E=\{E_i, e_{ij}\}$ such that there is an $i_0$ with $E_i=0$ for  $i<i_0$; define the full DG subcategory $\PTR^-(\sC)$ similarly as having objects $E=\{E_i, e_{ij}\}$ such that there is an $i_0$ with $E_i=0$ for  $i>i_0$. 
\end{definition}

\begin{rem} If $\sA$ is an additive category, then $\PTR^\infty(\sA)$ is the category of unbounded complexes $C(\sA)$. Similarly $\PTR^+(\sA)=C^+(\sA)$ and $\PTR^-(\sA)=C^-(\sA)$. If $\sA$ is a tensor category, these last two identities are identities of DG tensor categories.
\end{rem}

The translation functor and cone sequences on $\PTR^\infty(\sC)$ restrict to give these structures on $\PTR^+(\sC)$ and  $\PTR^-(\sC)$.
If $\sC$ is a DG tensor category, then the same expressions used to define the tensor product of objects and morphisms make  $\PTR^+(\sC)$ and  $\PTR^-(\sC)$ into DG tensor categories

 To unify the notation, we sometimes denote $\PTR(\sC)$ by $\PTR^b(\sC)$.

\subsection{The total complex} There is a {\em total complex functor}
\[
\Tot:\PTR(\PTR(\sC))\to \PTR(\sC)
\]
defined as follows: if $\sE=\{\sE^l, \phi^{kl}:\sE^l\to \sE^k\}$, where $\sE^l=\{E^l_i, e^l_{ij}:E^l_j\to E^l_i\}$ and $\phi^{kl}=\{\phi^{kl}_{ij}:E^l_j\to E^k_i\}$, then $\Tot(\sE)$ is given by the collection of objects
\[
\Tot(\sE)_n:=\oplus_{j+l=n}E^l_j,
\]
together with the  morphisms $\Tot(\sE)_{mn}:\Tot(\sE)_n\to \Tot(\sE)_m$ defined as the sum of terms $(-1)^le^l_{ij}$ and $\phi^{kl}_{ij}$.

Using exactly the same formulas, the total complex functor  extends to
\[
\Tot:\PTR(\PTR^?(\sC))\to\PTR^?(\sC);\ ?=+,-,\infty
\]
and to 
\begin{align*}
&\Tot^+:\PTR^+(\PTR^+(\sC))\to\PTR^+(\sC);\\
&\Tot^-:\PTR^-(\PTR^-(\sC))\to\PTR^-(\sC),
\end{align*}
 which is an equivalence of DG categories with inverse the evident extension of the inclusion functor $i_0$. If $\sC$ is a DG tensor category,  then $\Tot$ is a DG tensor equivalence for $?=+,-$.

In particular, if $\sA$ is an additive category, we have an equivalence of DG categories
\[
\Tot:\PTR(C^{\pm}(\sA))\to C^\pm(\sA).
\]

\begin{prop}\label{prop:Tot} Let $\sC$ be a DG category.\\
\\
1. For $?=b, +, -, \infty$, the functor $\Tot:\PTR(\PTR^?(\sC))\to \PTR^?(\sC)$ is an equivalence of DG categories. The quasi-inverse is the inclusion functor $i_0:\PTR^?(\sC)\to \PTR(\PTR^?(\sC))$. Similarly, the functors $\Tot^\pm$ are equivalences of DG categories.
\\
2. $\Tot$, $\Tot^+$ and $\Tot^-$ intertwine the respective translation functors, and send cone sequences to cone sequences.\\
\\
3. Suppose that $\sC$ is a DG tensor category. Then for $?=b, +,-$
\[
\Tot:\PTR(\PTR^?(\sC))\to\PTR^?(\sC)
\]
 is a tensor functor, and is in fact an equivalence of DG tensor categories, with quasi-inverse $i_0$. The same is true for $\Tot^+$ and $\Tot^-$.
\end{prop}

\begin{proof} One checks by direct computation that $\Tot$ is a DG functor. The composition $\Tot\circ i_0$ is the identity, so to complete the proof of (1), we need to construct a natural isomorphism $\id\cong i_0\circ\Tot$. If  $\sE=\{\sE^l, \phi^{kl}:\sE^l\to \sE^k\}$, where $\sE^l=\{E^l_i, e^l_{ij}:E^l_j\to E^l_i\}$ and $\phi^{kl}=\{\phi^{kl}_{ij}:E^l_j\to E^k_i\}$, then $\Tot(\sE)_n=\oplus_{i+l=n}E^l_i$. Define the map
\[
\rho^{k0}:i_0\circ\Tot(\sE)^0\to \sE^k
\]
in $\PTR(\sC)$ to be the collection of maps
\[
\rho^{k0}_{ij}:i_0\circ\Tot(\sE)^0_j=\oplus_{l+n=j}E^l_n\to E^k_i
\]
with $\rho^{k0}_{ij}=0$ if $k+i\neq j$, and with $\rho^{k0}_{ij}$ the projection on the summand 
$E^k_i$ if $k+i=j$. Then $\rho$ is a degree zero map in $\PTR(\PTR(\sC))$ with $\partial\rho=0$. The inverse $\lambda$ to $\rho$ is given by the collection of inclusions $\lambda^{0l}_{n,n-l}:E^l_{n-l}\to \oplus_{k+i=n}E^k_i$. It is easy to check that $\rho$ and $\lambda$ are natural.

A direct computation shows that $\Tot$ intertwines the translation functors an maps a cone sequence in 
$\PTR(\PTR(\sC))$ to a cone sequence in $\PTR(\sC)$, and that $\Tot$ is a tensor functor if $\sC$ is a DG tensor category.  This completes the proof of (1)-(3); (4) follows directly, noting that both $\Tot$ and $i_0$ are compatible (up to isomorphism) with translation and cone sequences.
\end{proof}

\subsection{Non-positive DG categories and  truncation functors}\label{sec:Truncation}
Recall the canonical truncation $\tau_{\le n}$ operation on the category $C(\sA)$ of complexes over an abelian category $\sA$:
\[
(\tau_{\le n}C)^m=\begin{cases}C^m&\text{ for }m<n\\0&\text{for }m>n\\\ker(d^n:C^n\to C^{n+1})&\text{ for }m=n.
\end{cases}
\]
For $f\in Z^0\sHom_{C(\sA)}(C,C')$,  we let $\tau_{\le n}f:\tau_{\le n}C\to \tau_{\le n}C'$ be the map induced by $f$.

We have the evident inclusion $\tau_{\le n}C\to C$, giving a natural endofunctor  $\tau_{\le n}$ of $Z^0C(\sA)$, which descends to a natural endofunctor of $K(\sA)$. Since $\tau_{\le n}C\to C$ induces an isomorphism on $H^m$ for $m\le n$,  $\tau_{\le n}$ passes to an endofunctor on the derived category $D(\sA)$.

If $\sA$ is an abelian tensor category, then for complexes $C, C'$, the map $\tau_{\le n}C\otimes\tau_{\le n'}C'\to C\otimes C'$ induced by the inclusions $\tau_{\le n}C\to C$, $\tau_{\le n'}C'\to C'$ factors through $\tau_{\le n+n'}(C\otimes C')\to C\otimes C'$; this follows from the Leibniz rule for $d_{C\otimes C'}$.

Now let $\sC$ be a DG category. We let $\tau_{\le 0}\sC$ be the DG category with the same objects as $\sC$, and with the Hom-complexes given by
\[
\sHom_{\tau_{\le 0}\sC}(X,Y):=\tau_{\le 0}\sHom_\sC(X,Y).
\]
The composition law for $\tau_{\le 0}\sC$ is given by the commutative diagram
\[
\xymatrix{
\tau_{\le0}\sHom_\sC(Y,Z)\otimes\tau_{\le 0}\sHom_\sC(X,Y)\ar@{^{(}->}[r]\ar[d]_{\circ}&
\sHom_\sC(Y,Z)\otimes \sHom_\sC(X,Y)\ar[d]^\circ\\
\tau_{\le0}\sHom(X,Z)\ar@{^{(}->}[r]&\sHom_\sC(X,Z)
}
\]

\begin{rem}
For $\sC$ an arbitrary DG category, the inclusion $\tau_{\le 0}\sC\to \sC$ is universal for DG functors of a non-positive DG category $\sC'$ to $\sC$.

If $\sC$ is a DG tensor category, then $\tau_{\le 0}\sC$ is a DG tensor subcategory, similarly universal for DG functors from a non-positive DG tensor category to $\sC$.
\end{rem}

Let $\sC$ be a non-positive DG category. We have the functor of ``stupid" truncation $\sigma_{\le N}$ on $\PTR^\infty(\sC)$: For $Y=\{Y^i, e_{ij}:Y^j\to Y^i\}$, $\sigma_{\le N}Y$ is the object given by the collection $\{Y^i, i\le N, e_{ij}, i,j\le N\}$. We define $\sigma_{\le N}$ on morphisms similarly (but note that $\sigma_{\le N}$ is {\em not} compatible with the differential). It is easy to check that $\sigma_{\le N}$ is a well-defined endofunctor of the underlying graded additive category, and the evident collections of identity maps and zero maps defines a natural transformation
\[
\pi_{\le N}(Y):Y\to \sigma_{\le N}Y
\]
with $\pi_{\le N}\circ\pi_{\le M}=\pi_{\le M}\circ\pi_{\le N}=\pi_{\le N}$ for $N\le M$. Similarly, we have the stupid truncation in the other direction $\sigma^{\ge N}Y$, and natural transformation
\[
\iota^{\ge N}(Y):\sigma^{\ge N}Y\to  Y.
\]
with  $\iota^{\ge M}\circ\iota^{\ge N}=\iota^{\ge N}\circ\iota^{\ge M}=\iota^{\ge N}$ for $N\ge M$.

\begin{rem} Beware that $\sigma_{\le N}$ and $\sigma^{\ge N}$ do {\em not} define DG functors, and hence do not pass to functors on the homotopy category.
\end{rem} 

\begin{rem}\label{rem:Cone} Let $\sC$ be a non-positive DG category.  Let $Y=\{Y^i, e_{ij}\}$ be in $\PTR(\sC)$, and suppose that $Y=\sigma_{\le N}Y$. Then 
the collection of maps $e_{Nj}:Y^j\to Y^N$, for $j<N$, give rise to a  map
\[
e_{N,*<N}:\sigma_{\le N-1}Y\to Y^N[1-N]
\]
in $Z^0\sC$, with $Y$ is isomorphic to $\Cone(e_{N,*<N})[-1]$. By induction on $N$, this shows that $\PTR(\sC)=C^b(\sC)$.
\end{rem} 

This motivates the following definition:
\begin{definition} For a DG category $\sC$, let $C^\infty(\sC)$ be the smallest full DG subcategory of $\PTR^\infty(\sC)$ that contains $\PTR^\infty(\tau_{\le0}\sC)$ and is closed under translation, isomorphism and taking cones. Let $C^+(\sC)$ and $C^-(\sC)$ be similarly defined closures of $\PTR^+(\tau_{\le0}\sC)$ and $\PTR^-(\tau_{\le0}\sC)$ in $\PTR^+(\sC)$ and $\PTR^-(\sC)$, respectively.

For $?=b, +,-,\infty$, let $K^?(\sC)$ be the homotopy category of $C^?(\sC)$.
\end{definition}

\begin{rem} $C^+(\sC)$ and $C^-(\sC)$ are also the ``closures" of  $\PTR^+(\tau_{\le0}\sC)$ and $\PTR^-(\tau_{\le0}\sC)$ in $\PTR^\infty(\sC)$.
\end{rem}

The analog of theorem~\ref{thm:DGMain} remains true for $C^?(\sC)$, $?=\infty, +, -$. The proofs  are exactly the same as in the bounded case. For reference purposes, we record this result here:

\begin{thm} \label{thm:DGMain+} Let $\sC$ be a DG category, $?=b,+,-,\infty$\\
\\
1. The translation functor on $C^?(\sC)$ induces a translation functor  on $K^?(\sC)$. Defining a distinguished triangle in $K^?(\sC)$ to be a triangle that is isomorphic to the image of a cone sequence in $C^?(\sC)$ makes $K^?(\sC)$ into a triangulated category.\\
\\
2. If $\sC$ is a DG tensor category, then for $?=b,+,-$, the induced tensor structure on $C^?(\sC)$ defines a tensor structure on $K^?(\sC)$, making $K^?(\sC)$ a tensor triangulated category.\\
\\
3. Let $F:\sC\to \sC'$ be a DG functor. Then $K^?F:K^?(\sC)\to K^?(\sC')$ is exact. If $F$ is a DG tensor functor of DG tensor categories, then $K^?F$ is an exact tensor functor for $?=b,+,-$.
\end{thm}

Similarly, we have

\begin{prop} For $?=b,+,-,\infty$,  the equivalence $\Tot:\PTR(\PTR^?(\sC))\to\PTR^?(\sC)$ induce an equivalence 
\[
\Tot:C^b(C^?(\sC))\to C^?(\sC)
\]
with quasi-inverse $i_0:C^?(\sC)\to C^b(C^?(\sC))$.  If $\sC$ is a DG tensor category, and $?=b,+,-$, then $\Tot$ is an equivalence of DG tensor categories. The analogous statements hold for 
\begin{align*}
\Tot^+:C^+(C^+(\sC))\to C^+(\sC)\\
\Tot^-:C^-(C^-(\sC))\to C^-(\sC)
\end{align*}
\end{prop}

\begin{cor}\label{cor:Universality} Let $\sC$ be a DG category, $\sA$ an additive category, and 
\[
F:\sC\to C^?(\sA)
\]
a DG functor, $?=b,+,-,\infty$.  Then $F$ defines a canonical exact functor
\[
K^b(F):K^b(\sC)\to K^?(\sA);
\]
if $F$ is a DG tensor functor, then $K^?(F)$ is an exact tensor functor for $?=b,+,-$.
\end{cor}

\begin{proof} $F$ induces the exact (tensor) functor
\[
K^b(F):K^b(\sC)\to K^b(C^?(\sA));
\]
we apply the exact functor (exact tensor functor for $?=b,+,-$)
\[
\Tot: K^b(C^?(\sA))\to K^?(\sA)
\]
to give the result.
\end{proof}

\subsection{Homotopy equivalence of DG categories}

\begin{definition}\label{definition:DGHEquiv} A functor of DG category $F:\sC\to\sC'$ is called a {\em homotopy equivalence} if 
\begin{enumerate}
\item $H^0F$ induces an isomorphism $\Obj(H^0\sC)/\Iso\to \Obj(H^0\sC')/\Iso$.
\item For each pair of objects $X, Y\in\sC$, the map 
\[
F:\sHom_\sC(X,Y)^*\to \sHom_{\sC'}(F(X),F(Y))^*\]
 is a quasi-isomorphism.
\end{enumerate}
\end{definition}

\begin{lem}\label{lem:Limit} Let $\sC$ be a non-positive DG category. For $Y$ in $\PTR^\infty(\sC)$, the limits $\lim_N\sigma_{\le N}Y$ and $\colim_N\iota^{\ge N}Y$ both exist, and the natural maps
\[
Y\to \lim_N\sigma_{\le N}Y;\quad \colim_N\sigma^{\ge N}Y\to Y
\]
induced by the $\pi_{\le N}$ and $\iota^{\ge N}$ are isomorphisms. In addition, we have a natural short exact sequence
\begin{multline*}
0\to R^1\lim_NH^{n-1}(\Hom_{\PTR^\infty(\sC)}(X, \sigma_{\le N}Y)\to
H^n(\Hom_{\PTR^\infty(\sC)}(X, Y)^*)\\
\to  \lim_NH^n(\Hom_{\PTR^\infty(\sC)}(X, (\sigma_{\le N}Y)^*)\to0.
\end{multline*}
The analogous result holds for  $\lim_N\sigma_{\le N}(-)$ in  $\PTR^+(\sC)$ and
$\colim_N\sigma^{\ge N}(-)$ in $\PTR^-(\sC)$.
\end{lem}

\begin{proof}  For all $X\in \PTR^\infty(\sC)$, 
\[
\Hom_{\PTR^\infty(\sC)}(X, \sigma_{\le N}Y)^n=\prod_{i\le N, j}\Hom_{\sC}(X^j,Y^i)^{j-i+n}
\]
and hence the maps $\pi_{\le N}$ induce an isomorphism
\[
\Hom_{\PTR^\infty(\sC)}(X, Y)^*\to \lim_N\Hom_{\\PTR^\infty(\sC)}(X, (\sigma_{\le N}Y)^*
\]
Thus $Y$ represents the projective limit. In addition, the maps in the projective system of Hom-complexes are just projections, hence satisfy the Mittag-Leffler conditions for a projective system of complexes. This gives us the short exact sequence
\begin{multline*}
0\to R^1\lim_NH^{n-1}(\Hom_{\PTR^\infty(\sC)}(X, (\sigma_{\le N}Y))\to
H^n(\Hom_{\PTR^\infty(\sC)}(X, Y)^*)\\
\to  \lim_NH^n(\Hom_{\PTR^\infty(\sC)}(X, (\sigma_{\le N}Y)^*)\to0,
\end{multline*}

The proof for the colimit of the $\sigma^{\ge N}Y$ is similar.
\end{proof}

\begin{lem}\label{lem:NonPosEquiv} Let $F:\sC\to \sC'$ be a homotopy equivalence of non-positive DG categories.   Then
for $?=b,+,-,\infty$
\[
K^?(F): K^?(\sC)\to K^?(\sC')
\]
is an equivalence of triangulated categories.
\end{lem}

\begin{proof} We first give the proof for $?=+$. We first show that $K^+(F)$ is fully faithful.

Take $X, Y\in\PTR^+(\sC)$.   By translating, may assume that $X=\{X^i, i\ge0, f_{ij}\}$, $Y=\{Y^i, i\ge0, e_{ij}\}$. If $X,Y$ are in $\sC$, then by assumption
\[
\Hom_{K^+(\sC)}(X[-n], Y[-m])\xrightarrow{F} \Hom_{K^+(\sC')}(F(X)[-n], F(Y)[-m])
\]
is an isomorphism for all $n,m$. In general, since we may express $\sigma_{\le M}X$, and $\sigma_{\le N}Y$ as cones (see remark~\ref{rem:Cone})
\begin{align*}
&\sigma_{\le M}X\cong \Cone(f_{M,*<M}:\sigma_{\le M-1}X\to X^M[1-M])[-1]\\
&\sigma_{\le N}Y\cong \Cone(e_{N,*<N}:\sigma_{\le N-1}Y\to Y^N[1-N])[-1]
\end{align*}
By induction on $N,M$, it follows that  $F$ induces a quasi-isomorphism
\[
\sHom_{\PTR^+(\sC)}(\sigma_{\le M}X,\sigma_{\le N}Y)^*\to
\sHom_{\PTR^+(\sC')}(\sigma_{\le M}F(X),\sigma_{\le N}F(Y))^*
\]
for every $N,M$. Additionally,  the map
\[
\pi_{\le M}(X)^*:\sHom_{\PTR^+(\sC)}(\sigma_{\le M}X,\sigma_{\le N}Y)^n\to
\sHom_{\PTR^+(\sC)}(X,\sigma_{\le N}Y)^n
\]
is an isomorphism for $M\ge N-n$, and similarly in $\PTR^+(\sC)$, so $F$ induces a quasi-isomorphism
\[
\sHom_{\PTR^+(\sC)}(X,\sigma_{\le N}Y)^*\to
\sHom_{\PTR^+(\sC')}(F(X),\sigma_{\le N}F(Y))^*
\]
for every $N$. Applying lemma~\ref{lem:Limit} shows that $F$ induces a quasi-isomorphism
\[
\sHom_{\PTR^+(\sC)}(X, Y)^*\to
\sHom_{\PTR^+(\sC')}(F(X),F(Y))^*
\]
as desired.

It remains to show that $F$ induces a surjection on isomorphism classes. Take $X=\{X^i, f_{ij}:X^j\to X^i\}$ in  $\PTR^+(\sC')$; we may assume that $X^n=0$ for $n<0$. By assumption, $F$ induces an isomorphism from the isomorphism classes in $H^0\sC$ to those in $H^0\sC'$, so there is a $Y^0\in\sC$ with $F(Y^0)\cong X^0$ in $H^0\sC'$. Assume by induction that we have an integer $N\ge0$, and for $0\le n\le N$, we have an object $Y_n$ in $\PTR^+(\tau_{\le  0}\sC)$ and a map $\phi_n:F(Y_n)\to \sigma_{\le n}X$ such that
\begin{enumerate}
\item[a.] $Y_n=\sigma_{\le N}Y_N$ and $\phi_n=\pi_{\le n}\phi_N$.
\item[b.] $\phi_n$ is an isomorphism in $K^b(\sC')$.
\end{enumerate}
Choose an object $Y^{N+1}$ in $\sC$ with an isomorphism $\phi^{N+1}: F(Y^{N+1})\to X^{N+1}$ in $H^0\sC'$.  
We have the map
\[
f:=f_{N+1,*<N+1}:\sigma_{\le N}X\to X^{N+1}[-N] 
\]
in $Z^0\PTR(\sC')$ and the isomorphism
\[
\sigma_{\le N+1}X\cong \Cone(f_{N+1,*<N+1})[-1].
\]
Since $K^+(F)$ is fully faithful, there is map
\[
e:=e_{N+1,*<N+1}:Y_N\to Y^{N+1}[-N]
\]
in $Z^0\PTR(\sC)$ such that
\[
\phi^{N+1}[-N]\circ F(e)=f\circ \phi_N
\]
in $K^+(\sC')$. 

We lift this relation to an identity in $\PTR(\sC')$. There is a degree -1 map
\[
a:F(Y_N)\to X^{N+1}[-N]
\]
in $\PTR(\sC')$ with
\[
\partial a=f\circ \phi_N-\phi^{N+1}[-N]\circ F(e)
\]

Let $Y_{N+1}:=\Cone(e)[-1]$. Define the degree 0 map
\[
\phi_{N+1}:F(Y_{N+1})\to \sigma_{\le N+1}X
\] 
via the matrix
\[
\phi_{N+1}:=\begin{pmatrix}\phi_N&0\\ a&\phi^{N+1}[-N-1]\end{pmatrix}
\]
where we make the identification
\[
\sHom(F(Y_N) X^{N+1}[-N-1])^0=\sHom(F(Y_N) X^{N+1}[-N])^1.
\]
Now, for any degree map $\Phi:F(Y_{N+1})\to \sigma_{\le N+1}X$ given as a matrix
\[
\Phi=\begin{pmatrix}\Phi_N&0\\ G&\Phi^{N+1} \end{pmatrix}
\]
with $\Phi_N:F(Y_N)\to F(Y_N)$, $\Phi^{N+1}:F(Y^{N+1})[-N-1]\to X^{N+1}[-N-1]$ and $G:F(Y_N)\to X^{N+1}[-N-1]$ degree zero morphisms, 
we have
\[
\partial_{\sigma_{\le N+1}X,F(Y_{N+1})}\Phi=
\begin{pmatrix}\partial\Phi_N&0\\ \partial G +f\circ \Phi_N-\Phi^{N+1}\circ F(e)&\partial\Phi^{N+1}\end{pmatrix}
\]
Thus
\begin{multline*}
\partial\phi_{N+1}=\begin{pmatrix}\partial\phi_N&0\\ \partial a-\phi^{N+1}[-N-1]\circ F(e)+f\circ \phi_N&\partial(\phi^{N+1}[-N-1])\end{pmatrix}\\=
\begin{pmatrix}0&0\\
-f\circ \phi_N+\phi^{N+1}[-N-1]\circ F(e)+f\circ \phi_N-\phi^{N+1}[-N-1]\circ F(e)&0\end{pmatrix}\\=0,
\end{multline*}
so $\phi_{N+1}:F(Y_{N+1})\to \sigma_{\le N+1}X$ defines a morphism in $K^+(\sC')$. 

By construction, 
 we have $\sigma_{\le N}Y_{N+1}=Y_N$ and $\sigma_{\le N}\phi_{N+1}=\phi_N$. In addition, we have the diagram in $\PTR(\sC')$
 \[
 \xymatrix{
F(Y^{N+1})[-N-1] \ar[r]^-{F(i_Y)}\ar[d]_{\phi^{N+1}[-N-1]}&F(Y_{N+1})\ar[r]^-{F(\pi_{N+1})}\ar[d]_{\phi_{N+1}}&F(Y_N)\ar[d]_{\phi_N}\ar[r]^-{F(e)}&F(Y^{N+1})[-N]\ar[d]^{\phi^{N+1}[-N]}\\
X^{N+1}\ar[r]_-{i_X}&\sigma_{\le N+1}X\ar[r]_-{\pi_{N+1}}&\sigma_{\le N}X\ar[r]_-{f}&X^{N+1}[-N].
}
\]
The left-hand and middle squares are commutative, while the right-hand square commutes up to the homotopy given by the degree -1 map $a$. Thus, the above diagram yields a commutative diagram in $K^b(\sC')$, with rows distinguished triangles (up to sign). Since $\phi^{N+1}$ and $\phi_N$ are isomorphisms in $K^b(\sC')$, it follows that $\phi_{N+1}$ is also an isomorphism in $K^b(\sC')$, and the induction goes through.

Let $Y=\{Y^i, i\ge0, e_{ij}\}$. Then $Y_N=\sigma_{\le N}Y$ for all $N\ge0$, and the sequence of maps $\phi_N:F(Y_N)\to \sigma_{\le N}X$ give a well-defined map $\phi:F(Y)\to X$ in $Z^0\PTR^+(\sC')$, such that
\[
\sigma_{\le N}\phi=\phi_N
\]
for each $N\ge0$. We claim that $\phi$ is an isomorphism in $K^+(\sC')$. For this, take an object $Z=\{Z^i\}$ in $K^+(\sC')$ and suppose $Z^i=0$ for $i<i_0$. For each $N\ge0$ we have the commutative diagram
\[
\xymatrix{
F(Y)\ar[d]_{F(\pi_N)}\ar[r]^\phi&X\ar[d]^{\pi_N}\\
F(\sigma_{\le N}Y)\ar[r]_{\phi_N}&\sigma_{\le N}X
}
\]
and the map $\phi_N$ is an isomorphism in $K^+(\sC')$. Thus, we have the commutative diagram
\[
\xymatrix{
0\ar[d]&0\ar[d]\\
R^1\lim_N H^{n-1}(\sHom(Z,F(\sigma_{\le N}Y))^*)\ar[d]\ar[r]_{\phi_{N*}}
&R^1\lim_N H^{n-1}(\sHom(Z,\sigma_{\le N}X)^*)\ar[d]\\
H^n(\sHom(Z,F(Y))^*)\ar[d]\ar[r]_{\phi_*}
&H^n(\sHom(Z,X)^*)\ar[d]\\
\lim_N H^n(\sHom(Z,F(\sigma_{\le N}Y))^*)\ar[d]\ar[r]_{\phi_{N*}}
&\lim_N H^n(\sHom(Z, \sigma_{\le N}X)^*)\ar[d]\\
0&0
}
\]
with exact  columns (see lemma~\ref{lem:Limit}). Since $\phi_{N*}$ is an isomorphism, this  verifies our claim. This completes the proof in case $?=+$.

The proof for $?=b$ is the same, but easier, as we can omit the limit argument. For $?=-$, we use a dual proof, replacing the system of projections $\pi_{\le N}$ with the system of inclusions  $\iota^{\ge N}$.

Now consider the case $?=\infty$. Take $X,Y\in C^\infty(\sC)$. From the case $?=-$, $F$ induces a quasi-isomorphism
\[
\sHom_{K^\infty(\sC)}(\sigma_{\le N}X,\sigma_{\le N}Y)\to 
\sHom_{K^\infty(\sC')}(\sigma_{\le N}F(X),\sigma_{\le N}F(Y))
\]
for all $N$. Arguing as in the case $?=+$, we see that $F$ is fully faithful.

Take $X\in C^\infty(\sC')$. From the case $?=-$, there is a $Y_0\in C^-(\sC)$ and an isomorphism $\phi_0:Y_0\to \sigma_{\le 0}X$. Arguing as in the case $?=+$, we construct a sequence of compatible objects $Y_N$ in $C^-(\sC)$, and maps $\phi_N:\F(Y_N)\to \sigma_{\le N}X)$ which are isomorphisms in $K^-(\sC')$. Taking the limit gives us a $Y\in C^\infty(\sC)$ and a map $\phi:F(Y)\to X$ which is an isomorphism in $K^\infty(\sC')$, completing the proof.
\end{proof}

\begin{thm}\label{thm:DGHomEq}  Let $F:\sC\to \sC'$ be a homotopy equivalence of  DG categories.   Then
for $?=b,+,-,\infty$
\[
K^?(F): K^?(\sC)\to K^?(\sC')
\]
is an equivalence of triangulated categories. If $\sC$ and $\sC'$ are DG tensor categories and  $F:\sC\to \sC'$ is a homotopy equivalence and a DG tensor functor, then $K^?(-)$ is an equivalence of tensor triangulated categories for $?=b,+,-$.
\end{thm}

\begin{proof}The statement for a DG tensor functor follows from the assertion for a homotopy equivalence. 

Define the  full additive subcategory $K^?(\sC)_N$ of $K^?(\sC)$ inductively, with 
$K^?(\sC)_0=K^?(\tau_{\le 0}\sC)$, and $K^?(\sC)_N$ having objects those $Y\in K^?(\sC)$ that fit into a distinguished triangle
\[
Y_1\to Y_2\to Y\to Y_1[1]
\]
in $K^?(\sC)$, with $Y_1, Y_2$ in $K^?(\sC)_{N-1}$. Define $K^?(\sC')_N$ similarly. Since
\[
K^?(\sC)=\cup_NK^?(\sC)_N
\]
it suffices to show that $K^?(F)_N:K^?(\sC)_N\to K^?(\sC')_N$ is an equivalence of additive categories for each $N$.

The case $N=0$ follows from  lemma~\ref{lem:NonPosEquiv}, so take $N\ge1$ and assume the result for $N-1$.

It follows easily by induction on $N$ that $K^?(F)_N$ is fully faithful. If $X$ is in $K^?(\sC')_N$, choose a distinguished triangle $X_1\xrightarrow{f} X_2\to X\to X_1[1]$ with $X_i\in K^?(\sC')_{N-1}$, $i=1,2$. By induction, there is a morphism $Y_1\xrightarrow{g} Y_2$ in $K^?(\sC)_{N-1}$ and a commutative diagram in $K^?(\sC')_{N-1}$ 
\[
\xymatrix{
F(Y_1)\ar[r]^{F(g)}g\ar[d]_{\phi_1}&F(Y_2)\ar[d]^{\phi_2}\\
X_1\ar[r]_f&X_2
}
\]
with $\phi_1, \phi_2$ isomorphisms. Complete $Y_1\xrightarrow{g} Y_2$ to a distinguished triangle 
$Y_1\xrightarrow{g} Y_2\to Y\to Y_1[1]$ and extend the commutative diagram to a commutative diagram
\[
\xymatrix{
F(Y_1)\ar[r]^{F(g)}\ar[d]_{\phi_1}&F(Y_2)\ar[r]\ar[d]^{\phi_2}&F(Y)\ar[d]^\phi\ar[r]&Y_1[1]\ar[d]^{\phi_1[1]}\\
X_1\ar[r]_f&X_2\ar[r]&X\ar[r]&X_1[1]
}
\]
Then $\phi$ is an isomorphism, and hence $K^?(F)_N$ is a bijection on isomorphism classes. Thus, the induction goes through.
\end{proof}

\subsection{Cohomologically non-positive DG categories}

\begin{definition} A DG category $\sC$ is {\em cohomologically non-positive} if for all $X, Y$ in $\sC$
\[
H^n(\sHom_\sC(X,Y)^*)=0
\]
 for $n>0$.
\end{definition}

\begin{prop} Let $\sC$ be a cohomologically non-positive DG category. Then the inclusions $
\tau_{\le 0}\sC\to\sC$ induces an equivalence of triangulated categories $K^?(\tau_{\le 0}\sC)\to K^?(\sC)$ for $?=b,+,-,\infty$. If $\sC$ is a DG tensor category, then $K^?(\tau_{\le 0}\sC)\to K^?(\sC)$ is an equivalence of tensor triangulated categories for $?=b,+,-$. 
\end{prop}

\begin{proof} This follows from theorem~\ref{thm:DGHomEq}.\end{proof}

\section{Sheafification of DG categories}  We show how to construct a DG category $R\Gamma\sC$ out of a presheaf of DG categories $U\mapsto \sC(U)$ so that the Hom complexes in $R\Gamma\sC$ compute the hypercohomology of the presheaf of Hom complexes for $\sC$.

\subsection{Sheafifying}
Fix a Grothendieck topology $\tau$ on some full subcategory $\Op^\tau_S$ of $\Sch_S$; we assume that 
$\Op^\tau_S$ is closed under fiber product over $S$, that $\tau$ is subcanonical, and that $\tau$ has a conservative set of points $\Pt_\tau(S)$. Let $\Sh_\tau(S)$ be the category of sheaves of abelian groups on $\Op_S$ for the topology $\tau$, and $\PSh_\tau(S)$ the category of presheaves on $\Op_S$. We let
\[
i:\prod_{x\in\Pt_\tau(S)}\Ab\to \Sh_\tau(S)
\]
be the canonical map of topoi, i.e, we have the functor $i^*:\Sh_\tau(S)\to \prod_{x\in\Pt_\tau(S)}$ (the factor $i^*_x$ sends $f$ to the stalk $f_x$ at $x$) and the right adjoint of $i^*$, $i_*:
\prod_{x\in\Pt_\tau(S)}\to \Sh_\tau(S)$. Let $\sG^0:=i_*\circ i^*$, $\sG^n:=(\sG^0)^n$. Combining the unit $\epsilon:\id\to i_*i^*$ and the co-unit $\eta:i^*i_*\to \id$ of the adjunction gives us the cosimplicial object in the category of endofunctors of $\Sh_\tau(S)$:
\[
\sG^0\ \begin{matrix}\to\\\leftarrow\\\to\end{matrix}\ \sG^1\ 
\begin{matrix}\to\\\leftarrow\\\to\\\leftarrow\\\to\end{matrix}\ \ldots
\]
with augmentation $\epsilon:\id\to\sG^0$. For a sheaf $\sF$, we let 
\[
\sF\xrightarrow{\epsilon}\sG^*\sF
\]
be the augmented complex associated to the augmented  cosimplicial sheaf $n\mapsto \sG^n(\sF)$; we extend this construction to complexes by taking the total complex of the associated double complex.

\begin{lem} For each complex of sheaves $\sF$, the augmentation $\epsilon:\sF\to \sG^*(\sF)$ is a quasi-isomorphism, and $\sG^n(\sF)$ is acyclic for the functor $R^n\Gamma(S,-)$.
\end{lem}

Thus, we have a functorial  $\Gamma(S,-)$-acyclic resolution for complexes of sheaves  of abelian groups on $S$. In addition, $i^*$ is a tensor functor, and the isomorphism
\[
i^*(i_*(\sA)\otimes i_*(\sB))\to i^*i_*(\sA)\otimes i^*i_*(\sB)
\]
followed by the co-unit 
\[
i^*i_*(\sA)\otimes i^*i_*(\sB)\to \sA\otimes \sB
\]
induces a natural transformation $i_*(\sA)\otimes i_*(\sB)\to i_*(\sA\otimes \sB)$. This gives us a natural map
\[
\sG^0(\sF)\otimes\sG^0(\sF')\to \sG^0(\sF\otimes\sF')
\]
and thus by iteration a map of cosimplicial sheaves
\[
\mu:[n\mapsto \sG^n(\sF)\otimes\sG^n(\sF')]\to [n\mapsto  \sG^n(\sF\otimes\sF')].
\]
We have the Alexander-Whitney map 
\[
 AW:\sG^*(\sF)\otimes\sG^*(\sF')\to [n\mapsto \sG^n(\sF)\otimes\sG^n(\sF')]^*
 \]
 sending $\sG^p(\sF)\otimes \sG^q(\sF')\to \sG^{p+q}(\sF)\otimes\sG^{p+q}(\sF')$ by $\sG(i_1^{p,q})\otimes\sG(i_2^{p,q})$, where
 \[
 i_1^{p,q}:\{0,\ldots, p\}\to \{0,\ldots, p+q\};\quad  i_2^{p,q}:\{0,\ldots, q\}\to \{0,\ldots, p+q\}
 \]
 are the maps
 \[
 i_1^{p,q}(j)=j,\ i_2^{p,q}(j)=j+p.
 \]
 The composition 
 \[
 \mu^*_{\sF,\sF'}:=\mu\circ AW:\sG^*(\sF)\otimes\sG^*(\sF')\to \sG^*(\sF\otimes\sF')
 \]
 makes $\sF\mapsto \sG^*(\sF)$ a weakly monoidal functor, i.e., we have the associativity
 \[
 \mu^*_{\sF\otimes\sF',\sF''}\circ(\mu^*_{\sF,\sF'}\otimes\id_{\sG^*(\sF'')})=
  \mu^*_{\sF,\sF'\otimes\sF''}\circ(\id_{\sG^*(\sF)}\otimes\mu^*_{\sF',\sF''}).
  \]

  These facts enable the following construction: Let $U\mapsto \sC(U)$ be a presheaf of DG categories on $\Op^\tau_S$. For objects $X$ and $Y$ in $\sC(S)$, we have the presheaf
 \[
[f:U\to S]\mapsto \sHom_{\sC(U)}(f^*(X), f^*(Y));
\]
we denote the associated sheaf by $\un{\sHom}^\tau_{\sC}(X,Y)$.
  
Let $R\Gamma(S,\sC)$ be the DG category with the same objects as $\sC(S)$, and with Hom-complex
\[
\sHom_{R\Gamma(S,\sC)}(X,Y)^*:=\sG^*(\un{\sHom}^\tau_{\sC}(X,Y))(S).
\]
Our remarks on the Godement resolution show that the composition law for the sheaf of DG categories $\sC$ defines canonically an associative composition law
\[
\circ:\sHom_{R\Gamma(S,\sC)}(Y,Z)^*\otimes \sHom_{R\Gamma(S,\sC)}(X,Y)^*\to
\sHom_{R\Gamma(S,\sC)}(X,Z)^*
\]
for $R\Gamma(S,\sC)$, making $R\Gamma(S,\sC)$ a DG category. In addition we have
\begin{prop} Let $\sC$ be a presheaf of DG categories on $\Op^\tau_S$. Then for each pair of objects $X,Y$ of $\sC(S)$, we have a canonical isomorphism
\[
\H^n(S_\tau,\un{\sHom}^\tau_{\sC}(X,Y))\cong H^n(\sHom_{R\Gamma(S,\sC)}(X,Y)).
\]
\end{prop}
Indeed, since  $\sF\to\sG^*(\sF)$ is a $\Gamma(S,-)$-acyclic resolution of $\sF$ for any complex of sheaves $\sF$, we have a canonical isomorphism $\H^n(S_\tau,\sF)\cong H^n(\sG^*(\sF)(S))$ for any complex of presheaves $\sF$ on $\Op^\tau_S$.
  
\subsection{Thom-Sullivan co-chains} \label{sec:ThomSullivan}
 Suppose now that we are given a presheaf of DG tensor categories $U\mapsto \sC(U)$ on $\Op_S^\tau$.  The Alexander-Whitney maps gives us a natural associative pairing of Hom-complexes
\[
\otimes:\sHom_{R\Gamma(S,\sC)}(X,Y)^*\otimes
\sHom_{R\Gamma(S,\sC)}(X',Y')^*\to
\sHom_{R\Gamma(S,\sC)}(X\otimes X',Y\otimes Y')^*
\]
satisfying the Leibniz rule and compatible with the tensor product operation on $\sC(S)$, however, this map is not commutative. For this, we need assume that $U\mapsto \sC(U)$ is a sheaf of $\Q$-DG category; we can then  modify the construction of $R\Gamma(S,\sC)$, giving a homotopy equivalent DG category $R\Gamma(S,\sC)^\otimes$ with a well-defined DG tensor structure.  The construction uses the {\em Thom-Sullivan cochains}; we have taken this material from \cite{HinSch}.  

Let $|\Delta_n|$ be the real $n$-simplex
\[ |\Delta_n|:=\{(t_0,\ldots,t_n)\in\R^{n+1}\ |\ \sum_{i=0}^nt_i=1, 0\le t_i\},
\] and let $\Delta_n$ be the simplicial set
$\Hom_\Delta(-,[n]):\Delta^\op\to\Sets$. For $f:[m]\to[n]$ in $\Delta$,
let $|f|:|\Delta_m|\to|\Delta_n|$ denote the corresponding affine-linear map.

For a cosimplicial abelian group $G$, we have the associated complex $G^*$, and the
{\em normalized} subcomplex, $NG^*$, with
$NG^p$ the subgroup of those $g\in G([p])$ such 
$G(f)(g)=0$ for all $f:[p]\to[q]$ in $\Delta$ which are not injective. The inclusion
of
$NG^*\hookrightarrow G^*$ is a homotopy equivalence. In particular, we have the
complex of (simplicial) cochains of $\Delta_n$, and the subcomplex of normalized
cochains  $Z^*(\Delta_n)$.

Let $\Omega^*(|\Delta_n|)$ denote the complex of  $\Q$-polynomial differential forms 
on $|\Delta_n|$:
\[
\Omega^*(|\Delta_n|):=\Omega^*_{\Q[t_0,\ldots,t_n]/\sum_{i=0}^nt_i-1}.
\]

Sending $[n]$ to $Z^*(\Delta_n)$, $\Omega^*(|\Delta_n|)$ determines functors
\[
Z^*:\Delta^\op\to C^{\ge 0}(\Ab),\quad 
\Omega^*:\Delta^\op\to C^{\ge 0}(\Mod_\Q),
\] where $C^{\ge 0}(\Mod_\Q)$ is the category of complexes of $\Q$-vector spaces 
concentrated in degrees $\ge0$, and $C^{\ge 0}(\Ab)$ is the integral version. There
is a natural  homotopy equivalence
\[
\int:\Omega^*\to Z^*\otimes\Q
\] defined by
\[
\int(\omega)(\sigma)=\int_{|\sigma|}\omega
\] for $\omega\in\Omega^m(|\Delta_n|)$ and 
$\sigma$ an $m$-simplex of $\Delta_n$.

We have the category  $\Mor(\Delta)$, with objects the morphisms in $\Delta$, where
 a morphism $f\to g$ is a commutative diagram
\[
\vcenter{\xymatrix{ {\bullet}\ar[r]^-f&{\bullet}\ar[d]\\
{\bullet}\ar[u]\ar[r]_-g&{\bullet\hbox to0pt{\ .\hss}} }}
\] 
For functors
$F:\Delta^\op\to C^{\ge 0}(\Ab)$ and $G:\Delta\to \Ab$, we have the functor 
\begin{gather*} F(\text{domain})\otimes G(\text{range}):\Mor(\Delta)\to C^{\ge0}\\
f\mapsto F(\text{domain}(f))\otimes G(\text{range}(f));
\end{gather*} let $F\Otimes_{\leftarrow} G$ be the projective limit
\[ F\Otimes_\leftarrow G:=\lim_{\Mor(\Delta)}
F(\text{domain})\otimes G(\text{range}).
\] 
Explicitly, an element $\epsilon$ of $(F\Otimes_\leftarrow G)^m$ is given by a 
collection
$p\mapsto \epsilon_p\in F^m([p])\otimes G([p])$ such that
\[ F(f)\otimes G(\id)(\epsilon_p)=F(\id)\otimes G(f)(\epsilon_q)
\] for each $f:[q]\to [p]$ in $\Delta$. The operation $\Otimes_\leftarrow$ is
functorial and respects homotopy equivalence.

For a cosimplicial abelian group $G$, we have the well-defined map
$e:Z^*\Otimes_\leftarrow G\to NG^*$ defined by sending 
$\epsilon:=(\ldots \epsilon_p\ldots)\in (Z^*\Otimes_\leftarrow G)^q$ to 
$\epsilon_q(\id_{[q]})\in G([q])$. In \cite[Lemma 3.1]{HinSch} it is shown that
this map is well-defined,  lands in $NG^*$, and gives an isomorphism of
complexes. Thus, we have the natural homotopy equivalences
\[
\int\Otimes_\leftarrow\id:\Omega^*\Otimes_\leftarrow G\to NG^*\otimes\Q\to
G^*\otimes\Q.
\]

The operation of wedge product makes $\Omega^*$ into a simplicial commutative differential graded
algebra. Suppose we have two cosimplicial abelian groups $G, G':\Delta\to \Ab$, giving us the diagonal cosimiplicial abelian group $G\otimes G':\Delta\to\Ab$. We have the map
\[
\cup:(\Omega^*\Otimes_\leftarrow G)\otimes(\Omega^*\Otimes_\leftarrow G')\to 
\Omega^*\Otimes_\leftarrow (G\otimes G')
\]
induced by the map 
\[
(\omega\otimes g)\otimes(\omega'\otimes g')\mapsto
(\omega\wedge\omega')\otimes  (g\otimes g'), 
\]
for $\omega\otimes g\in \Omega^q(|\Delta_p|)\otimes G([p])$, $\omega'\otimes g'\in
\Omega^{q'}(|\Delta_p|)\otimes G([p])$. It is easy to check that this gives a
well-defined functorial  product of cochain complexes and that $\cup$ is associative and commutative, in the evident sense.

We have as well the product map
\[
\cup_{AW}:G^*\otimes G^{\prime*}\to (G\otimes G')^*
\]
induced by the Alexander-Whitney maps $i_1^p:[p]\to[p+q]$, $i_2^q:[q]\to[p+q]$
\[
(g\in G^p)\otimes(g'\in G^{\prime q})\mapsto G(i_1^p)(g)\otimes G'(i_2^q)(g')\in G^{p+q}\otimes G^{\prime p+q}
\]
It follows easily from the change of variables formula for integration that the diagram
\begin{equation}\label{eqn:CommuteProduct}
\xymatrix{
(\Omega^*\Otimes_\leftarrow G)\otimes(\Omega^*\Otimes_\leftarrow G')\ar[r]^-\cup\ar[d]_{\int\otimes\int}&
\Omega^*\Otimes_\leftarrow (G\otimes G')\ar[d]^\int\\
(G^*\otimes\Q)\otimes(G^{\prime*}\otimes \Q)\ar[r]_-{\cup_{AW}}&(G\otimes G')^*\otimes\Q
}
\end{equation}
commutes.

We extend the operation $\Omega^*\Otimes_\leftarrow(-)$ to functors
$G^*:\Delta^\op\to C^+(\Ab)$ by taking the total complex of the double complex
\[
\ldots\to \Omega^*\Otimes_\leftarrow G^0\to \Omega^*\Otimes_\leftarrow G^1\to\ldots\ .
\] All the properties of $\Omega^*\Otimes_\leftarrow(-)$ described above extend to
the case of complexes.

\subsection{Sheafifying DG tensor categories} Let   $U\mapsto \sC(U)$ be a presheaf of $\Q$-DG tensor categories on $\Op_S^\tau$. 

Define 
\[
\sHom_{R\Gamma(S,\sC)^\otimes}(X,Y)^*:= \Omega^*\Otimes_\leftarrow\sG^*(\sHom^\tau_{\sC}(X,Y))(S)
\]
Using the functoriality of the Godement resolution, the composition law for $U\mapsto \sC(U)$ together with the product $\cup$ defined in \S\ref{sec:ThomSullivan} defines a composition law
\[
\circ:\sHom_{R\Gamma(S,\sC)^\otimes}(Y,Z)^*\otimes
\sHom_{R\Gamma(S,\sC)^\otimes}(X,Y)^*\to \sHom_{R\Gamma(S,\sC)^\otimes}(X,Z)^*
\]
giving us the DG category $R\Gamma(S,\sC)^\otimes$. By the commutativity of \eqref{eqn:CommuteProduct} the maps
\[
\int\Otimes_\leftarrow\id:\sHom_{R\Gamma(S,\sC)^\otimes}(X,Y)^*\to \sHom_{R\Gamma(S,\sC)}(X,Y)^*
\]
define a DG functor
\[
\int:R\Gamma(S,\sC)^\otimes\to R\Gamma(S,\sC)
\]
which is a homotopy equivalence.

Finally, the tensor product operation on $\sC$ gives rise to maps of cosimplicial objects
\[
\otimes:\sG^*(\sHom^\tau_{\sC}(X,Y))(S)\otimes
\sG^*(\sHom^\tau_{\sC}(X',Y'))(S)\to \sG^*(\sHom^\tau_{\sC}(X\otimes X',Y\otimes Y'))(S)
\]
Applying the Thom-Sullivan cochain construction gives the map
\[
\tilde{\otimes}:\sHom_{R\Gamma(S,\sC)^\otimes}(X,Y)^*\otimes
\sHom_{R\Gamma(S,\sC)^\otimes}(X',Y')^*\to
\sHom_{R\Gamma(S,\sC)^\otimes}(X\otimes X',Y\otimes Y')^*
\]
that makes $R\Gamma(S,\sC)^\otimes$ a DG tensor category, with commutativity constraint induced from $\sC(S)$.

\begin{rem}\label{rem:TensorStructure} Let $U\mapsto \sC(U)$ be a presheaf of $\Q$-DG tensor categories on $\Op_S^\tau$. The homotopy equivalence $\int:R\Gamma(S,\sC)^\otimes\to R\Gamma(S,\sC)$ induces an equivalence of triangulated categories
\[
K^b(\int):K^b(R\Gamma(S,\sC)^\otimes)\to K^b(R\Gamma(S,\sC))
\]
(theorem~\ref{thm:DGHomEq}). In addition, the DG tensor structure we have defined on 
$R\Gamma(S,\sC)^\otimes$ endows $C^b(R\Gamma(S,\sC)^\otimes)$ with the structure of a DG tensor category, and makes $K^b(R\Gamma(S,\sC)^\otimes)$ a triangulated tensor category.
\end{rem}
 
\subsection{Idempotent completion} Recall that an additive category $\sA$ is {\em pseudo-abelian} if each idempotent endomorphism admits a kernel and cokernel. If $\sA$ is an additive category, one has the {\em idempotent completion} $\sA\to \sA^\natural$ of $\sA$, which is the universal functor of $\sA$ to a pseudo-abelian category; $\sA^\natural$ has objects $(M,p)$, where $M$ is an object of $\sA$ and $p:M\to M$ is an idempotent endomorphism,
\[
\Hom_{\sA^\natural}((M,p),(N,q)):=p^*q_*\Hom_\sA(M,N)=q_*p^*\Hom_\sA(M,N)\subset 
\Hom_\sA(M,N), 
\]
and the composition law in $\sA^\natural$ is induced by that of $\sA$. If $\sA$ is a tensor category, $\sA^\natural$ inherits the tensor structure, making $\sA\to \sA^\natural$ is a tensor functor.

One extends the definition to DG categories and DG tensor categories in the evident manner: if $\sC$ is a DG category, then $\sC^\natural$ has objects $(M,p)$ with $p:M\to M$ an idempotent endomorphism in $Z^0\sC$ The Hom-complex is given by
\[
\sHom_{\sC^\natural}((M,p),(N,q))^*:=p^*q_*\sHom_\sC(M,N)=q_*p^*\sHom_\sC(M,N).
\]

A theorem of Balmer-Schlichting \cite{BalmerSchlichting} tells us that, for $\sA$ a triangulated category, $\sA^\natural$ has a canonical structure of a triangulated category for which $\sA\to \sA^\natural$ is exact; the same holds in the setting of triangulated tensor categories.

\section{DG categories of motives}\label{sec:DGMotives} Let $S$ be a fixed base-scheme; we assume that $S$ is a regular scheme of finite Krull dimension. Let $\Sm/S$ denote the category of smooth $S$-schemes of finite type, $\Prj/S\subset \Sm/S$ be full subcategory of $\Sm/S$ consisting of the smooth projective $S$-schemes. We form the DG category of correspondences $\DGCor_S$ from a cubical enhancement of the category of finite correspondences $\Cor_S$, using the algebraic $n$-cubes as the cubical object; restricting to smooth projective $S$-schemes gives us our basic DG category $\DGPCor_S$. We sheafify the construction over $S$, then take the homotopy category of complexes over the sheafified DG category to form the triangulated category $K^b(R\Gamma(S,\un{\DGPCor}_S))$. Taking the idempotent completion gives us our category of smooth  effective motives $\PMot^\eff(S)$; inverting tensor product with the Lefschetz motive gives us the category of smooth motives $\PMot(S)$.

We also have a parallel version with $\Q$-coefficients; using alternating cubes endows everything with a tensor structure.

\subsection{DG categories of correspondences} We begin by recalling the definition of the category of finite correspondences.

\begin{definition} For $X, Y\in\Sm/S$, $\Cor_S(X,Y)$ is the free abelian group on the integral closed subschemes $W\subset X\times_SY$ such that the projection $W\to X$ is finite and surjective onto an irreducible component of $X$.
\end{definition}

Now let $X, Y$ and $Z$ be in $\Sm/S$. Take generators $W\in \Cor_S(X, Y)$, $W'\in\Cor_S(Y,Z)$. As in \cite[Chap. V]{FSV}, each component $T$ of the intersection $W\times_SZ\cap X\times_SW'\subset X\times_SY\times_SZ$ is finite over $X\times_SZ$ and over $X$ (via the projections) and the map
\[
T\to X
\]
is surjective over some irreducible component of $X$. Thus, letting $p_{XY}$, $p_{YZ}$, etc., denote the projections from the triple product $X\times_SY\times_SZ$, and $\cdot_{XYZ}$ the intersection product of cycles on $X\times_SY\times_SZ$, the expression
\[
W\circ W':=p_{XZ*}(p_{XY}^*(W)\cdot_{XYZ}p_{YZ}^*(W'))
\]
gives a well-defined and associative composition law
\[
\circ:\Cor_S(X,Y)\otimes\Cor_S(Y,Z)\to \Cor_S(X,Z),
\]
defining the pre-additive category $\Cor_S$. $\Cor_S$ is an additive category, with disjoint union being the direct sum, and product over $S$ makes $\Cor_S$ into a tensor category.

Sending $X\in\Sm/S$ to $X\in\Cor_S$ and sending $f:X\to Y$ to the graph of $f$, $\Gamma_f\subset X\times_SY$, defines a faithful embedding $i_S:\Sm/S\to \Cor_S$.  Making $\Sm/S$ a symmetric monoidal category using the product over $S$ makes $i_S$ a symmetric monoidal functor.

Set $\square^1_S=(\A_S^1,0_S,1_S)$. For each $n=0,1,\ldots$, we have the $n$-cube $\square^n_S:=(\square^1)^n$, giving us the co-cubical object $\square_S^*$,
\[
n\mapsto \square^n_S.
\]
Indeed, we send  $p_i$ to the $i$th projection $p_i:\square^n_S\to\square^{n-1}_S$, $\eta_{n,i,\epsilon}$ to the $\epsilon$-section to $p_i$ ($\epsilon=0,1$), $\tau_{n,i}$ to the involution of $\square^n_S$ sending  the $i$th coordinate $t_i$ to $1-t_i$ and acting by the identity on the other factors, and having the permutation group $\Sigma_n$ act on $\square^n$ by permuting the factors.

In fact, $\square^*_S$ extends to a functor
\[
\square^*_S:\ECube\to\Sm/S
\]
by sending the multiplication map $\mu:\un{2}\to\un{1}$ to the usual multiplication
\[
\mu_S:\square^2_S\to\square^1_S;\ \mu_S(x,y)=xy.
\]

Define the co-multiplication $\delta:\square^*_S\to \square^*_S\otimes\square^*_S$ by taking the collection of diagonal maps $\delta^n:=\delta_{\square^n}:\square^n_S\to \square^n_S\times_S\square^n_S$. One easily verifies the properties of definition~\ref{def:comult}.

\begin{definition} \label{def:DGMot} Let $\DGCor_S$ denote the DG category $(\Cor_S, \otimes, \square^*_S,\delta)$. We denote the DG tensor category $((\Cor_{S\Q},\otimes,\square_S^*,\delta)^\alt, \otimes_\square)$ by $\DGCor^\alt_S$.
\end{definition}

\begin{prop}\label{prop:CorHE} 1. The DG categories $\DGCor_S$ and $\DGCor^\alt_S$ are non-positive.\\
\\
2. The functor $\DGCor^\alt_S\to \DGCor_{S\Q}$  is a homotopy equivalence.
\end{prop}

\begin{proof} Indeed, (1) is obvious, and (2) follows from proposition~\ref{prop:AltDG}.
\end{proof}

\begin{definition} Let $\DGPCor_S$ be the full DG subcategory of $\DGCor_S$ with objects $X\in\Prj/S$.
$\DGPCor^\alt_S$ is similarly defined as the full DG subcategory of $\DGCor^\alt_S$ with objects 
$X\in\Prj/S$.
\end{definition}
Note that $\DGPCor^{\alt}_S$ is a DG tensor subcategory of $\DGCor^\alt_S$, and the functor
$\DGPCor^{\alt}_S\to \DGPCor_{S\Q}$ induced by $\DGCor^\alt_S\to \DGCor_{S\Q}$  is a homotopy equivalence.

\subsection{Functoriality} Let $f:S'\to S$ be a $k$-morphism of regular schemes. We have the well-defined pull-back functor $f^*:\Cor_S\to \Cor_{S'}$, with $f^*(X)=X\times_SS'$ for $X\in\Sm/S$ and using Serre's intersection multiplicity formula to define the pull-back of cycles 
\[
f^*:\Cor_S(X,Y)\to \Cor_{S'}(f^*(X), f^*(Y))
\]
Note that the cycle pull-back is always well-defined for finite correspondences, using the isomorphism
\[
(X\times_SS')\times_{S'}(Y\times_SS')\cong (X\times_SS')\times_SY.
\]

Since $f^*(X)\times_{S'}\square^n_{S'}\cong f^*(X\times_S\square^n_S)$, the pull-back extends to the map of complexes
\[
f^*:\sHom_{\DGCor_S}(X,Y)\to \sHom_{\DGCor_{S'}}(f^*(X),f^*(Y)),
\]
defining the functor $S\mapsto \DGCor_S$ from regular schemes to DG categories. We have as well the sub-functor $S\mapsto \DGPCor_S$.

A similar construction defines the functor $S\mapsto \DGCor^\alt_S$,  from regular schemes to DG tensor categories, and the subfunctor  $S\mapsto \DGPCor^{\alt}_S$.

\subsection{Complexes and smooth motives} 

\begin{definition} Let $S$ be a regular scheme.  Define the Zariski presheaf  $\un{\DGPCor}_S$  by
\[
U\mapsto \un{\DGPCor}_S(U):= \DGPCor_U.
\]
The DG category of smooth effective motives over $S$, $\DGPMot^\eff_S$, is defined as
\[
\DGPMot^\eff_S:=C^b(R\Gamma(S, \un{\DGPCor}_S)).
\]
The triangulated category, $\PMot^\eff(S)$,  of smooth effective motives over $S$ is defined as the idempotent completion of the homotopy category 
\[
K^b(R\Gamma(S, \un{\DGPCor}^\eff_S))=H^0\DGPMot^\eff_S.
\]

The Zariski  presheaf, $\un{\DGPCor}^\alt_S$, is defined by
\[
U\mapsto \un{\DGPCor}^\alt_S(U):= \DGPCor^\alt_U.
\]
The DG tensor category of smooth effective motives over $S$ with $\Q$-coefficients, $\DGPMot^\eff_{S\Q}$, is defined as
\[
\DGPMot^\eff_{S\Q}:=C^b(R\Gamma(S, \un{\DGPCor}^\alt_S)^\otimes).
\]
The triangulated tensor category, $\PMot^\eff(S)_\Q$, of smooth effective motives over $S$ with $\Q$-coefficients  is defined as the idempotent completion of the homotopy category $K^b(R\Gamma(S, \un{\DGPMot}^\alt_S)^\otimes)=H^0\DGPMot^\eff_{S\Q}$.
\end{definition}

\begin{prop} The functor $\un{\DGPCor}^\alt_S\to \un{\DGPCor}_S\otimes\Q$ defined by the natural functors
\[
\DGPCor^\alt_U\to \DGPCor_U\otimes\Q
\]
gives rise to a functor of DG categories
\[
R\Gamma(S, \un{\DGPCor}^\alt_S)^\otimes\to R\Gamma(S, \un{\DGPCor}_S)\otimes\Q,
\]
which in turn induces an equivalence of $\PMot^\eff(S)_\Q$ with the idempotent completion of $\PMot^\eff(S)\otimes\Q$, as triangulated categories.
\end{prop}

\begin{proof} This follows from proposition~\ref{prop:CorHE} and remark~\ref{rem:TensorStructure}.
\end{proof}

\begin{rem}\label{rem:Functoriality}
We note that the DG categories $R\Gamma(S, \un{\DGPCor}_S)$ and $\DGPMot^\eff_S$ are functorial in the regular scheme $S$, as is the triangulated category $\PMot^\eff(S)$. The same holds for the DG tensor categories $R\Gamma(S, \un{\DGPCor}^\alt_S)^\otimes$, $\DGPMot^\eff_{S\Q}$ and the triangulated tensor category $\PMot^\eff(S)_\Q$.
\end{rem}

\subsection{Lefschetz motives} In $\Prj/S$, we have the idempotent endomorphism $\alpha$ of $\P^1_S$ defined as the composition 
\[
\P^1_S\xrightarrow{p} S\xrightarrow{i_\infty}\P^1_S.
\]
Since $\Cor_S$ is a tensor category, we have the object $\L:=(\P^1_S,1-\alpha)$ in $\Cor_S^\natural$, as well as the $n$th tensor power $\L^n$ of $\L$ and, for each $X\in \Sm/S$, the object $X\otimes \L^n$.
We thus have these objects in the DG categories of correspondences $\DGPCor_S^\natural$ and $R\Gamma(S, \un{\DGPCor}_S)^\natural$.

\begin{definition} \label{definition:EffectiveTate} The DG category of Lefschetz motives over $S$, $\DGTCor^\eff_S$, is defined to be the full DG subcategory of $R\Gamma(S, \un{\DGPCor}_S)^\natural$ with objects $\L^d$, $d\ge0$. The DG category of effective mixed Tate motives is
\[
\DGTMot^\eff_S:=C^b(\DGTCor^\eff_S).
\]
The triangulated category of  effect mixed Tate motives over $S$, $\PTMot^\eff(S)$, is  the homotopy category $K^b(\DGTCor^\eff_S)=H^0\DGTMot^\eff_S$.

We have parallel definitions of DG tensor categories and a triangulated tensor category with $\Q$-coefficients. The DG tensor category of Lefschetz motives over $S$ with $\Q$-coefficients, $\DGTCor^\eff_{S\Q}$, is  the full DG subcategory of $R\Gamma(S, \un{\DGPMot}^\alt_S)^{\otimes\natural}$ with objects finite direct sums of the $\L^d$, $d\ge0$,  The DG tensor category of effective mixed Tate motive with $\Q$-coefficients is
\[
\DGTMot^\eff_{S\Q}:=C^b(\DGTCor^\eff_{S\Q}),
\]
 and the  triangulated tensor category of  effect mixed Tate motives over $S$ with , $\PTMot^\eff(S)_\Q$, is  the homotopy category $K^b(\DGTCor^\eff_{S\Q})=H^0\DGTMot^\eff_{S\Q}$. 
\end{definition}
 
\section{Duality}\label{sec:Duality}

In   section~\ref{sec:FLVMov}, we will give an extension of the moving lemmas of Friedlander-Lawson to smooth projective schemes over a regular, semi-local base. Before we do this, we give in this section the applications apply this to define a twisted duality for Hom-complexes in $R\Gamma(S, \un{\DGPMot}_S)$, and extend this to a duality in various categories of motives.

In this section $S$ will be a regular scheme over a fixed base-field $k$.

\subsection{Equi-dimensional cycles}
\begin{definition} Let $X$ and $Y$ be smooth over $S$, $r\ge0$ an integer. The group $z^S_\equi(Y,r)(X)$ is the free abelian group on the integral subschemes $W\subset X\times_SY$ such that the projection $W\to X$ dominates an irreducible component $X'$ of $X$, and such that, for each $x\in X$,  the fiber $W_x$ over $x$ has pure dimension $r$ over $k(x)$, or is empty.

We let $z^S_\equi(Y,r)^\eff(X)\subset z^S_\equi(Y,r)(X)$ be the submonoid of effect cycles, that is, the free monoid on the generators $W$ for $z^S_\equi(Y,r)(X)$  described above. 
\end{definition}

For an $S$-morphism $f:X'\to X$, and for $W\in z^S_\equi(Y,r)(X)$, the pull-back cycle $(f\times\id_Y)^*(W)$ is well-defined and in $z^S_\equi(Y,r)(X')$. Thus $z^S_\equi(Y,r)$ is a presheaf on $\Sm/S$. For $x\in X\in\Sm/k$, and for $W\in z^S_\equi(Y,r)(X)$, we denote the pull-back $i_x^*(W)$ by the inclusion $i_x:x\to X$ by $W_x$.

Suppose $Y$ is projective over $S$, with a fixed embedding $Y\hookrightarrow \P^N_S$. Then we have a well-defined degree homomorphism
\[
\Deg:z^S_\equi(Y,r)(X)\to H^0(X_\Zar,\Z)
\]
which sends a cycle $W\in z^S_\equi(Y,r)(X)$ to the locally constant function on $X$ 
\[
x\mapsto \Deg(W_x),
\]
where $\Deg(W_x)$ is the usual degree of the cycle $W_x$ in $\P^N$.  

For each integer $e\ge1$, we let 
$z^S_\equi(Y,r)^\eff_e(X)\subset z^S_\equi(Y,r)^\eff(X)$ be the subset of $z^S_\equi(Y,r)^\eff(X)$ consisting of those $W$ with $\Deg(W)$ the constant function with value $e$ on $X$. We set
\[
z^S_\equi(Y,r)^\eff_{\le e}(X):=\amalg_{1\le e'\le e}z^S_\equi(Y,r)^\eff_{e'}(X).
\]
We let $z^S_\equi(Y,r)_{\le e}(X)\subset z^S_\equi(Y,r)(X)$ be the subgroup generated by the set  $z^S_\equi(Y,r)^\eff_{\le e}(X)$. Thus we have the presheaves of abelian monoids 
$z^S_\equi(Y,r)^\eff_{e}(X)$ and $z^S_\equi(Y,r)^\eff_{\le e}(X)$, and the presheaf of abelian groups 
$z^S_\equi(Y,r)_{\le e}(X)$.

\begin{definition} For $X, Y$ in $\Sm/S$, define the {\em cubical Suslin complex} $C^S(Y,r)^*(X)$ as the complex associated to the cubical object
\[
n\mapsto z^S_\equi(Y,r)(X\times_S\square^n_S),
\]
i.e.
\[
C^S(Y,r)^n(X):=z^S_\equi(Y,r)(X\times_S\square^{-n}_S)/\text{degn}.
\]
For $f:X'\to X$, define
\[
f^*:C^S(Y,r)(X)\to C^S(Y,r)(X')
\]
via the pull-back maps
\[
(f\times\id_{\square^n})^*:z^S_\equi(Y,r)(X\times\square^n)\to 
z^S_\equi(Y,r)(X'\times\square^n);
\]
this defines the presheaf of complexes $C^S(Y,r)$ on $\Sm/S$. 

Suppose that $Y$ is in $\Prj/S$ and we are given an embedding $Y\hookrightarrow \P^N_S$ over $S$. Let $C^S(Y,r)_{\le e}(X)\subset C^S(Y,r)(X)$ be the subcomplex corresponding to the cubical abelian group 
\[
n\mapsto z^S_\equi(Y,r)_{\le e}(X\times_S\square^n_S)\subset z^S_\equi(Y,r)(X\times_S\square^n_S).
\]
\end{definition}

\begin{rems} 1. Suppose that $Y$ is in $\Prj/S$. Then 
\[
C^S(Y,0)^*(X)=\sHom_{\DGCor_S}(X,Y)^*,
\]
and 
$C^S(Y,0)$ is the presheaf $\sHom_{\DGCor_S}(-,Y)^*$ on $\Sm/S$.\\
\\
2. Let $f:Y\to Y'$ be a proper morphism. Push-forward by the projection $\id\times f:X\times\square^n\times Y\to X\times\square^n\times Y'$ defines the map of complexes
\[
f_*(X):C^S(Y,r)(X)\to C^S(Y',r)(X)
\]
giving us the map of presheaves $f_*:C^S(Y,r)\to C^S(Y',r)$. Thus $Y\mapsto C^S(Y,r)$ defines a functor from $\Prj/S$ to complexes of sheaves on $\Sm/S$.\\
\\
3.  Correspondences act on $z^S_\equi(Y,r)(X)$, both in $Y$ (covariantly) and in $X$ (contravariantly). Thus sending $(Y,X)$ to $z^S_\equi(Y,r)(X)$ extends to a functor
\[
z^S_\equi(-,r)(-):\Cor_S\times\Cor_S^\op\to \Ab
\]
We have a similar extension of the complexes $C^S(Y,r)(X)$ to
\[
C^S(-,r)(-):\Cor_S\times\Cor_S^\op\to C^-(\Ab)
\]
As $\Ab$ is abelian, the bi-functors $z^S_\equi(-,r)(-):$ and $C^S(-,r)(-)$ extend to the idempotent completion of $\Cor_S$; in particular, the presheaves $z^S_\equi(Y\otimes\L^n,r)$ and 
$C^S(Y\otimes\L^n,r)$ are defined. 
\end{rems}

Now take $Y, X\in\Sm/S$ with $X\to S$ equi-dimensional of dimension $p$ over $S$. Each integral $W\in z^S_\equi(Y,r)(U\times_SX)$ gives us an integral subscheme $W$ of $U\times_SX\times_SY$ which is equi-dimensional of relative dimension $r+p$ over some component of $U$. This defines the map of complexes
\[
\int_X:C^S(Y,r)(U\times_SX)\to C^S(X\times_SY,r+p)(U)
\]

We can now state our main result, to be proven in \S\ref{sec:FLVMov}.

\begin{thm}\label{thm:Duality} Let  $S$ be a regular semi-local $k$-scheme, essentially of finite type over $k$. Then for all $X, Y\in\Prj/S$, $U\in\Sm/S$, with $X\to S$ of relative dimension $p$, the map
\[
\int_X:C^S(Y,r)(U\times_SX)\to C^S(X\times_SY,r+p)(U)
\]
is a quasi-isomorphism.
\end{thm}

 In addition, we will need a computation of $C_*(Y\times\P^n,r)$ for $r\ge n$. Fix linear inclusions $\iota_j:\P^j\to \P^n$, $j=0,\ldots, n$.  For each $i$, $0\le i\le n\le r$, define the map
\[
\alpha_j:C^S(Y,r-j)(X)\to C^S(Y\times\P^n,r)(X)
\]
by sending $W\subset Y\times_SX\times\square^n$ to  $(\iota_j\times\id)_*(\P^j\times W)$.

\begin{thm}\label{thm:PBF} Let  $S$ be a regular semi-local $k$-scheme, essentially of finite type over $k$. Then for all $X, Y\in\Prj/S$, and for $r\ge n$, the map
\[
\sum_{j=0}^n\alpha_j:\oplus_{j=0}^n C^S(Y,r-j)(X)\to C^S(Y\times\P^n,r)(X)
\]
is a quasi-isomorphism.  
\end{thm}
Theorem~\ref{thm:PBF} will also be proven in \S\ref{sec:FLVMov}. As consequence, we have

\begin{cor} \label{cor:Duality} Let  $S$ be a regular semi-local $k$-scheme, essentially of finite type over $k$.\\
\\
1. For $Y\in\Prj/S$, $X\in\Sm/S$ and $n\ge0$, let 
\[
\phi:C^S_*(Y,r)(X)\to C^S(Y\otimes \L^n,r+n)(X)
\]
be the map sending  $W\subset X\times\square^n\times_SY$ to $W\times(\P^1)^n\subset X\times(\P^1)^n\times\square^n_SY$. Then $\phi$   is an isomorphism in $D^-(\Ab)$.\\
\\
2. For $Y,Z\in\Prj/S$, $X\in\Sm/S$, with $Y$ of dimension $d$ over $S$, there are natural isomorphisms in $D^-(\Ab)$
\[
C^S(Y\times_SZ,r)(X)\cong C^S(Z\otimes\L^d,r)(X\times_SY).
\]
For $Z=S\times_kZ_0$, with $Z_0\in\Prj/k$, we have an isomorphism
\[
C^S(Z\otimes\L^d,r)(X\times_SY)\cong C^k(Z_0\otimes\L^d,r)(p_*(X\times_SY)).
\]
3. For $X\in\Sm/S$, $Y\in\Prj/S$ and $r\ge n\ge0$ integers, there are natural isomorphisms in 
$D^-(\Ab)$
\[
C^S(Y,r)(X\otimes\L^n)\cong C^S(Y,r+n)(X).
\]
\end{cor}

\begin{proof} For (1), the isomorphism 
\[
C^S(Y\otimes \L^n,r+n)(X)\cong C^S(Y,r)(X).
\]
follows by induction and the projective bundle formula (theorem~\ref{thm:PBF}). Indeed, it suffices to define a natural isomorphism
\begin{equation}\label{eqn:Formula1}
C^S(Y\otimes \L^n,r+1)(X)\to C^S(Y\otimes \L^{n-1},r)(X).
\end{equation}
Since the projection $\P^1\to\Spec k$ is split by the inclusion $i_\infty:\Spec k\to \P^1$, we get a direct sum decomposition
\[
C^S(Y\times\P^1,r)\cong C^S(Y,r)\oplus C^S(Y\otimes\L,r)
\]
By induction, we have a similar direct sum decomposition of $C^S(Y\times(\P^1)^n,r)$ for all $n$; this reduces the proof of \eqref{eqn:Formula1} to showing that
\[
C^S(Y\times(\P^1)^{n-1}\otimes \L,r+1)(X)\cong C^S(Y\times(\P^1)^{n-1},r)(X),
\]
which reduces us to the case $n=1$.
Comparing with the isomorphism
\[
\alpha_0+\alpha_1: C^S(Y,r+1)\oplus C^S(Y,r)\to C^S(Y\times\P^1,r+1)
\]
and noting that $p_*\circ \alpha_1=0$,  $p_*\circ\alpha_0=\id$, we see that $\alpha_1$ gives an isomorphism
\[
\alpha_1:C^S(Y,r)\to C^S(Y\otimes\L,r+1),
\]
completing the proof of the first assertion.

The first isomorphism in (2)  follows from the first assertion and theorem~\ref{thm:Duality}:
\[
C^S(Z\otimes\L^d,r)(X\times_SY)\cong C^S(Y\times_SZ\otimes \L^d,d+r)(X)\cong 
C^S(Y\times_SZ,r)(X).
\]
For the second isomorphism,  $C^S(Z\otimes\L^d,r)(X\times_SY)=C^k(Z_0\otimes\L^d,r)(X\times_SY)$, since we have the isomorphism of schemes (over $\Spec k$)
\[
T\times_S(S\times_kZ_0)\times_S\P^1_S\cong p_*T\times_kZ_0\times_k\P^1_k
\]
for all $S$-schemes $T$.

For (3), it suffices to prove the case $n=1$. By theorem~\ref{thm:Duality}, we have the quasi-isomorphism
\[
\int_{\P^1}:C^S(Y,r)(X\times\P^1)\to C^S(Y\times\P^1, r+1)(X).
\]
Since $Y\times\P^1\cong Y\oplus Y\otimes\L$ in $\Cor_S^\natural$, and similarly for $X\times\P^1$, we have the quasi-isomorphism
\[
\int_{\P^1}:C^S(Y,r)(X)\oplus C^S(Y,r)(X\otimes\L)\to C^S(Y\times\P^1, r+1)(X).
\]
Since $\L=(\P^1,1-i_{\infty}p)$, the summand $C^S(Y,r)(X\otimes\L)$ of 
$C^S(Y,r)(X\times\P^1)$ is the image of $\id-p^*i_\infty^*$, which is the same as the kernel of $i_\infty^*$. Similarly, the map  $C^S(Y,r)(X)\to C^S(Y,r)(X\times\P^1)$ is defined by $p^*$. Using the flag $0\subset\P^1$ to define the quasi-isomorphism  of theorem~\ref{thm:PBF},
\[
\alpha_0+\alpha_1:C^S(Y, r+1)(X)\oplus C^S(Y, r)(X)\to 
C^S(Y\times\P^1, r+1)(X),
\]
the inverse  isomorphism (in $D^-(\Ab)$) is  given by the map of complexes
\[
C^S(Y\times\P^1, r+1)_{Y\times \infty}(X)\xrightarrow{(p_*, i_\infty^*)}
C^S(Y, r+1)(X)\oplus C^S(Y, r)(X).
\]
composed with the inverse of the quasi-isomorphism 
\[
C^S(Y\times\P^1, r+1)_{Y\times \infty}(X)\hookrightarrow
C^S(Y\times\P^1, r+1)(X).
\]
We note that the image of $\int_{\P^1}$ lands in $C^S(Y\times\P^1, r+1)_{Y\times \infty}(X)$, and sends
$C^S(Y,r)(X\otimes\L)$ to $\ker i_\infty^*$ and $C^S(Y,r)(X)$ to $\ker p_*$. Thus $p_*\circ \int_{\P^1}$ defines a quasi-isomorphism
\[
p_*\circ \int_{\P^1}:C^S(Y,r)(X\otimes\L)\to C^S(Y, r+1)(X),
\]
as desired.
\end{proof}

\begin{rem}\label{rem:duality} For later use, we extract from the proof  a description of the isomorphism 
in corollary~\ref{cor:Duality}(2,3).  In (2), the isomorphism is the composition of
\[
\int_Y:C^S(Z\otimes\L^d,r)(X\times_SY)\to 
C^S(Y\times_SZ\otimes\L^d,r+d)(X)
\]
with the inverse of
\[
-\times(\P^1)^n:C^S(Y\times_SZ,r)(X)\to
C^S(Y\times_SZ\otimes\L^d,r+d)(X).
\]

To describe the map in (3), let $p:Y\times\P^1\to Y$ be the projection. We have the composition
\[
C^S(Y,r)(X\times\P^1)\xrightarrow{\int_{\P^1}}C^S(Y\times\P^1,r+1)(X)
\xrightarrow{p_*}C^S(Y,r+1)(X).
\]
Then the restriction of $p_*\circ\int_{\P^1}$ to the summand $C^S(Y,r)(X\otimes\L)$ of $C^S(Y,r)(X\times\P^1)$,
\[
p_*\circ\int_{\P^1}:C^S(Y,r)(X\otimes\L)\to C^S(Y,r+1)(X),
\]
is the isomorphism in (3). Iterating, we have the map
\[
[p_*\circ\int_{\P^1}]^n:C^S(Y,r)(X\otimes\L^n)\to C^S(Y,r+n)(X),
\]
giving the isomorphism in (3).

\end{rem}

For $Y,Z\in\Prj/S$, $X\in \Sm/S$, we have the natural map
\[
\times_SZ:z_\equi^S(Y,r)(X)\to z_\equi^S(Y\times_SZ,r)(X\times_SZ)
\]
defined by sending a cycle $W$ on $X\times_SY$ to the image of $W$ in $X\times_SZ\times_SY\times_SZ$  via the ``diagonal" embedding $X\times_SY\to X\times_SZ\times_SY\times_SZ$. This gives us the map
\[
\times_SZ:C^S(Y,r)(X)\to C^S(Y\times_SZ,r)(X\times_SZ).
\]
Taking $Z=\P^1$ and extending to the pseudo-abelian hull gives the map
\[
\otimes\L:C^S(Y,r)(X)\to C^S(Y\otimes\L, r)(X\otimes\L),
\]
sending $W\subset X\times\square^n\times_SY$ to $W\times\Delta_{\P^1}\subset X\times\P^1\times\square^n\times_SY\times\P^1$.

\begin{cor}\label{cor:Cancellation}  Let  $S$ be a regular semi-local $k$-scheme, essentially of finite type over $k$. For $X, Y\in\Prj/S$, the map $\otimes\L:C^S(Y,r)(X)\to C^S(Y\otimes\L,r)(X\otimes\L)$ is a quasi-isomorphism.
\end{cor}

\begin{proof} By the projective bundle formula (theorem~\ref{thm:PBF}), the map $W\mapsto W\times\P^1$ gives a quasi-isomorphism
\[
-\times\P^1:C^S(Y,r)(X)\to C^S(Y\otimes\L,r+1)(X).
\]
By corollary~\ref{cor:Duality}(3) and remark~\ref{rem:duality}, the map
\[
p_*\circ\int_{\P^1}:C^S(Y\otimes\L,r)(X\otimes\L)\to C^S(Y\otimes\L,r+1)(X)
\]
is a quasi-isomorphism. Now, for $W\in C^S(Y,r)(X)$, we have
\[
p_*\circ\int_{\P^1}(W\otimes\L)=p_*\circ\int_{\P^1}(W\times\Delta_{\P^1})=
W\times\P^1,
\]
hence
\[
\otimes\L:C^S(Y,r)(X)\to C^S(Y\otimes\L,r)(X\otimes\L)
\]
is a quasi-isomorphism.
\end{proof}

\subsection{Duality for smooth motives} The duality results of the previous section extend to the category $\PMot(S)$ and defines a duality on the tensor triangulated category $\PMot(S)_\Q$.

\begin{prop}  Let $S$ be a regular scheme, essentially of finite type over a field $k$. For $X\in\Prj/S$, the natural map
\[
\sHom_{\DGPCor_S^\natural}(X,\L^d)^*\to \sHom_{R\Gamma(S, \un{\DGPMot}_S)^\natural}(X,\L^d)^*
\]
is a quasi-isomorphism.
\end{prop}

\begin{proof} 
Let $p_*X\in\Sm/k$ denote the scheme $X$, considered as a $k$-scheme via the structure morphism $p:S\to\Spec k$. We have a canonical isomorphism
\[ 
\sHom_{\DGPCor_S^\natural}(X,\L^d)^*\cong \sHom_{\DGPCor_k^\natural}(p_*X,\L^d)^*,
\]
induced by the isomorphisms
\[
(\P^1)^d_S\times_S \square^n_S\times_SX\cong (\P^1_k)^d\times_k\square^n_k\times_kp_*X.
\]
On the other hand, the complex $\sHom_{\DGPCor_k^\natural}(p_*X,\L^d)^*$ is just a cubical version of the weight $d$ Friedlander-Suslin complex of the $k$-scheme $p_*X$, hence we have isomorphisms \cite{VoevodskyChow}
\[
H^n(\sHom_{\DGPCor_k^\natural}(p_*X,\L^d)^*)\cong   H^{2d+n}(X,\Z(d))\cong \CH^d(X,-n).
\]
Since the  higher Chow groups of smooth $k$-schemes satisfies the Mayer-Vietoris property for the Zariski topology \cite{BlochMovLem},  the presheaf of complexes on $S$
\[
U\mapsto \sHom_{\DGPCor_U^\natural}(X\times_SU,\L^d)^*
\]
satisfies the Mayer-Vietoris property on $S_\Zar$. Thus, by Thomason's theorem \cite{Thomason},  the natural map
\[
\sHom_{\DGPCor_S^\natural}(X,\L^d)^*\to R\Gamma(S,  [U\mapsto \sHom_{\DGPCor_U^\natural}(X\times_SU,\L^d)^*])
\]
is a quasi-isomorphism. Since $R\Gamma(S,  [U\mapsto \sHom_{\DGPCor_U^\natural}(X\times_SU,\L^d)^*])$ is by definition equal to $\sHom_{R\Gamma(S, \un{\DGPMot}_S)^\natural}(X,\L^d)^*$, the result is proven.
\end{proof}

Following remark~\ref{rem:action}, the tensor structure on $\Cor_S$ extends to an action of $\Cor_S$ on the presheaf  of DG categories $\un{\DGPMot}_S$, giving us an action of $\Cor^\natural_S$ on the DG categories 
$R\Gamma(S, \un{\DGPMot}_S)^\natural$ and $\DGPMot^{\eff\natural}_S$. Thus, we have an action
of $\Cor^\natural_S$ on the triangulated category $\PMot^\eff(S)$:
\[
\otimes:\Cor_S^\natural\otimes \PMot^\eff(S)\to \PMot^\eff(S).
\]
In particular, each object $A\in \Cor_S^\natural$ gives an exact functor
\[
(-)\otimes A:\PMot^\eff(S)\to \PMot^\eff(S)
\]
sending a morphism $f:X\to Y$ in $\PMot^\eff(S)$ to $f\otimes\id:X\otimes A\to Y\otimes A$.

\begin{thm}\label{thm:DualityExt} Let $S$ be a regular scheme, essentially of finite type over a field $k$. \\
\\
1. Let $X, Y$ and $Z$ be in $\Prj/S$, $d=\dim_SY$. Then there is a natural quasi-isomorphism
\[
\sHom_{R\Gamma(S, \un{\DGPMot}_S)}(X,Y\times_SZ)^*\cong
\sHom_{R\Gamma(S, \un{\DGPMot}_S)^\natural}(X\times_SY,Z\otimes\L^d)^*\]
2.  Let $X$ and $Y$ be in $\Prj/S$, $d=\dim_SY$, and take $Z$ in $\Prj/k$. Then there is a natural quasi-isomorphism
\[
\sHom_{R\Gamma(S, \un{\DGPMot}_S)}(X,Y\times_kZ)^*\cong
\sHom_{\DGPMot_k^\natural}(p_*(X\times_SY),Z\otimes \L^d)^*
\]
3. Let $X$ and $Y$ be in $\Prj/S$. Then the map
\[
\otimes\L:\sHom_{R\Gamma(S, \un{\DGPMot}_S)}(X,Y)^*\to \sHom_{R\Gamma(S, \un{\DGPMot}_S)^\natural}(X\otimes \L,Y\otimes\L)^*
\]
is a quasi-isomorphism.
\end{thm}

\begin{proof} For $X, Y\in \Prj/S$, let $\sHom_{\un{\DGPMot}_S}(X,Y)$ denote the Zariski sheaf  on $S$ associated to the presheaf
\[
U\mapsto \sHom_{\DGPMot_U}(X_U,Y_U)^*;
\]
extend the notation to define the sheaf $\sHom_{ \un{\DGPMot}^\natural_S}(X,Y\otimes\L^d)$, etc.

From corollary~\ref{cor:Duality} and remark~\ref{rem:duality}, we have  
 isomorphisms in $D^-(\Sh_S^\Zar)$
\[
\sHom_{\un{\DGPMot}_S}(X,Y\times_SZ)^*\cong
\sHom_{\un{\DGPMot}_S^\natural}(X\times_SY,Z\otimes\L^d)^*
\]
and 
\[
\otimes\L:\sHom_{\un{\DGPMot}_S}(X,Y)^*\to \sHom_{\un{\DGPMot}_S^\natural}(X\otimes \L,Y\otimes\L)^*
\]
These induce isomorphisms (in $D(\Ab)$) after applying $\Gamma(S,G^*(-))\cong R\Gamma(S,-)$, proving (1) and (3). (2) follows from (1), noting that we have the isomorphism of 
$\sHom_{\un{\DGPMot}_S}(X\times_SY,S\times_kZ\otimes\L^d)^*$ with the constant sheaf (on $S_\Zar$) with value
$\sHom_{\DGPMot_k^\natural}(p_*(X\times_SY),Z\otimes \L^d)^*$, and that for a constant sheaf of complexes $\sF$ on $S_\Zar$, the natural map
\[
\Gamma(S,\sF)\to R\Gamma(S,\sF)
\]
is an isomorphism in $D(\Ab)$.  
\end{proof}

\begin{definition} The DG category of motives over $S$, $\DGPMot_S$, is the DG category formed by inverting $\otimes\L$ on $\DGPMot^{\eff\natural}_S$. The triangulated category of motives over $S$, 
$\PMot(S)$ is the triangulated category formed by inverting $\otimes\L$ on $\PMot^\eff(S)$ (we will see in 
corollary~\ref{cor:InvertL} below that $\PMot(S)$ has a natural triangulated structure).

Explicitly, $\DGPMot_S$ has objects $X\otimes\L^n$, $n\in\Z$, $X$ in $\DGPMot^{\eff\natural}_S$, with
\[
\sHom_{\DGPMot_S}(X\otimes\L^n, Y\otimes\L^m)^*:=\colim_N
\sHom_{\DGPMot^{\eff\natural}_S}(X\otimes\L^{N+n}, Y\otimes\L^{N+m})^*
\]
Similarly, $\PMot(S)$  has objects 
 $X\otimes\L^n$, $n\in\Z$, $X$ in $\PMot^\eff(S)$, with
\[
\sHom_{\PMot(S)}(X\otimes\L^n, Y\otimes\L^m)^*:=\colim_N
\sHom_{\PMot^\eff(S)}(X\otimes\L^{N+n}, Y\otimes\L^{N+m})^*.
\]
\end{definition}

\begin{rems} We can also invert $\otimes\L$ on ${\DGPCor}_U^\natural$ for all open $U\subset S$, giving us the presheaf $\un{\DGPCor}_S^\natural[\otimes\L^{-1}]$ on $S_\Zar$, and the associated DG category
$R\Gamma(S,\un{\DGPCor}_S^\natural[\otimes\L^{-1}])$. We have an isomorphism of DG categories
\[
R\Gamma(S,\un{\DGPCor}_S^\natural[\otimes\L^{-1}])\cong R\Gamma(S,\un{\DGPCor}_S)^\natural[\otimes\L^{-1}]).
\]
2. Inverting $\otimes\L$ on the various Lefschetz/Tate categories of definition~\ref{definition:EffectiveTate} gives us full  sub-DG categories 
 $\DGTCor_S$, $\DGTMot(S)$ of $R\Gamma(S,\un{\DGPCor}_S^\natural[\otimes\L^{-1}])$ and 
 $\DGPMot_S$, resp., with
 \[
 \DGTMot(S)\cong C^b(\DGTCor_S),
 \]
and the full triangulated subcategory $\PTMot(S)$ of $\PMot(S)$.
 
 With $\Q$-coefficients we have the analogous DG tensor sub- categories  $\DGTCor^\alt_{S\Q}$, $\DGTMot_{S\Q}$ of $R\Gamma(S, \un{\DGPMot}^\alt_S)^{\otimes\natural}[\otimes\L^{-1}])$, 
 $\DGPMot_S$, resp., with
 \[
 \DGTMot_{S\Q}\cong C^b(\DGTCor^\alt_{S\Q}),
 \]
  and the full tensor triangulated subcategory $\PTMot(S)_\Q$ of $\PMot(S)_\Q$.
\end{rems}

\begin{thm} \label{thm:Cancel2} For $X, Y\in \DGPMot^{\eff\natural}_S$, the natural map
\[
\sHom_{\DGPMot^{\eff\natural}_S}(X, Y)^*\to \sHom_{\DGPMot_S}(X, Y)^*
\]
is a quasi-isomorphism, and the natural map
\[
\Hom_{\PMot^\eff(S)}(X, Y[n])\to \Hom_{\PMot(S)}(X, Y[n])
\]
is an isomorphism for all $n$. Furthermore, the isomorphism
\[
H^n\sHom_{\DGPMot_S^{\eff\natural}}(X, Y)^*\cong \Hom_{\PMot^\eff(S)^\natural}(X, Y[n])
\]
induces an isomorphism
\[
H^n\sHom_{\DGPMot_S}(X\otimes\L^p, Y\otimes\L^m)^*\cong \Hom_{\PMot(S)}(X\otimes\L^p, Y\otimes\L^m[n]).
\]
for each $n,m,p\in\Z$.
\end{thm}

\begin{proof} For $X, Y\in\Prj/S$, the first assertion is theorem~\ref{thm:DualityExt}(2). Since 
$\DGPMot^\eff_S$ is generated by $\Prj/S$ by taking translation, cone sequences and isomorphisms, the first assertion for $X, Y\in \DGPMot^\eff_S$ follows; this immediately implies the result for 
$\DGPMot^{\eff\natural}_S$. 

Since cohomology commutes with filtered inductive limits, we have the isomorphism
\[
H^n\sHom_{\DGPMot_S}(X, Y)^*\cong \Hom_{\PMot(S)}(X, Y[n]).
\]
as claimed. Thus the first assertion implies the second.
\end{proof}

As immediate consequence we have

\begin{cor}\label{cor:InvertL} 1.  $\PMot(S)$ is equivalent to the idempotent completion of $H^0\DGPMot_S$. In particular, 
$\PMot(S)$ has the natural structure of a triangulated category.\\
\\
2. The canonical functor $\PMot^\eff(S)\to \PMot(S)$ is an exact fully faithful embedding.
\end{cor}

\begin{rem} We can similarly invert $\otimes\L$ on the DG tensor categories $\DGCor^{\alt\natural}_S$, $R\Gamma(S, \un{\DGPCor}^\alt_S)^{\otimes\natural}$ and
$\DGPMot^{\eff\natural}_{S\Q}$, and on the tensor triangulated category $\PMot^\eff(S)_\Q$. Setting
\[
\DGPMot_{S\Q}:=\DGPMot^{\eff\natural}_{S\Q}[\otimes\L^{-1}],\
\PMot(S)_\Q:=\PMot^\eff(S)_\Q[\otimes\L^{-1}],
\]
this gives us the DG tensor functor $\DGPMot^{\eff\natural}_{S\Q}\to \DGPMot_{S\Q}$, and the exact tensor functor $\PMot^\eff(S)_\Q\to \PMot(S)_\Q$. The analogs of theorem~\ref{thm:Cancel2} and corollary~\ref{cor:InvertL} hold in this setting. Similarly, the analog of  theorem~\ref{thm:Cancel2} and corollary~\ref{cor:InvertL} hold for the respective subcategories of Tate motives.
\end{rem}

\subsection{Chow motives}

We recall the definition of the category of Chow motives over $S$.
\begin{definition} Let $S$ be a regular scheme. Let $\tilde\CM^\eff(S)$ be the category with the same objects as $\Prj/S$, and with morphisms (for $X\to S$ of pure dimension $d_X$ over $S$)
\[
\Hom_{\tilde\CM^\eff(S)}(X,Y):=\CH_{d_X}(X\times_SY).
\]
The composition law is the usual one of composition of correspondence classes: for $W\in \Hom_{\tilde\CM^\eff(S)}(X,Y)$, $W'\in \Hom_{\tilde\CM^\eff(S)}(Y,Z)$, define
\[
W'\circ W:=p_{13*}(p_{12}^*(W)\cdot_{XYZ}p_{23}^*(W')),
\]
where $p_{ij}$ is the projection of $X\times_SY\times_SZ$ on the $ij$ factors. The operation of product over $S$ makes $\tilde\CM^\eff(S)$ a tensor category. Sending $f:X\to Y$ to the graph of $f$ defines a functor 
\[
\tilde{m}_S:\Prj/S\to \tilde\CM^\eff(S).
\]

The category $\CM^\eff(S)$ of {\em effective Chow motives over $S$} is the idempotent completion of $\tilde\CM^\eff(S)$. We let $\L=(\P^1, 1-i_\infty\circ p)$, and define the category of Chow motives over $S$ as
\[
\CM^\eff(S):=\tilde\CM^\eff(S)[\otimes\L^{-1}].
\]
\end{definition}

Take $X, Y\in\Prj/S$. By theorem~\ref{thm:DualityExt} and \cite{VoevodskyChow}, we have the isomorphism
\begin{multline*}
\Hom_{\PMot^\eff(S)}(X,Y)=H^0\sHom_{R\Gamma(S, \un{\DGPMot}_S)}(X,Y)^*\\
\cong 
H^0\sHom_{\DGPMot_k}(p_*(X\times_SY),\L^{d_Y})^*
= H^{2d_Y}(X\times_SY,\Z(d_Y))\\\cong\CH_{d_X}(X\times_SY)
=\Hom_{\CM^\eff(S)}(X,Y).
\end{multline*}
which we denote as
\[
\phi_{X,Y}:\Hom_{\PMot^\eff(S)}(X,Y)\to Hom_{\CM^\eff(S)}(X,Y).
\]
\begin{lem}\label{lem:CycleSurjectivity} $\phi_{**}$ respects the composition: for $X,Y,Z\in\Prj/S$,  
\[
f\in\Hom_{\PMot^\eff(S)}(X,Y),\ g\in \Hom_{\PMot^\eff(S)}(Y,Z), 
\]
we have
\[
\phi_{X,Z}(g\circ f)=\phi_{Y,Z}(g)\circ\phi_{X,Y}(f).
\]
\end{lem}

\begin{proof} If $f$ is in the image of $z^S_\equi(Y,0)(X)\to \Hom_{\PMot^\eff(S)}(X,Y)$ and 
$g$ is in the image of $z^S_\equi(Z,0)(Y)\to \Hom_{\PMot^\eff(S)}(Y,Z)$, the result is obvious, so it suffices to show that
\[
z^S_\equi(Y,0)(X)\to \Hom_{\PMot^\eff(S)}(X,Y)=\CH_{d_X}(X\times_SY)
\]
is surjective for all $X,Y\in\Prj/S$; this is lemma~\ref{lem:FLMovOrig} below.
\end{proof}

\begin{lem}\label{lem:FLMovOrig}Take $X\in\Sm/S$  of dimension $d_X$ over $S$. For $Y\in \Prj/S$, , the map
\[
z^S_\equi(Y,r)(X)\to \CH_{d_X+r}(X\times_SY),
\]
sending a cycle $W$ on $X\times_SY$ to its cycle class, is surjective for all $r\ge0$.
\end{lem}

\begin{proof} Take a locally closed immersion $X\times_SY$ in a projective space $\P^N_k$ and let $T\subset\P^N_k$ be the closure of $X\times_SY$. Let $\gamma$ be a dimension $d_X+r$ cycle on $X\times_SY$ and let $\bar\gamma$ be the closure of $\gamma$ on $T$. By \cite[theorem 1.7]{FriedlanderLawson}, $\bar\gamma$ is rationally equivalent to a cycle $\bar\gamma'$ such that each component of $\bar\gamma'$ intersects the cycle $x\times Y$ properly on $T$, for each point $x\in X$ (note that $x\times Y$ is closed on $T_{k(x)}$ and is contained in the smooth locus of $T_{k(x)}$). Thus, the restriction $\gamma'$ of $\bar\gamma'$ to $X\times_SY$ is rationally equivalent to $\gamma$ and is in the image of 
$z^S_\equi(Y,r)(X)\to \CH_{d_X+r}(X\times_SY)$, completing the proof.
\end{proof}

Thus, we have shown

\begin{prop}\label{prop:ChowEmbed} 1. There is a fully faithful embedding $\psi^\eff:\CM^\eff(S)\to \PMot^\eff(S)$, such that the diagram
\[
\xymatrix{
\Prj/S\ar[r]\ar[dr]&\CM^\eff(S)\ar[d]^{\psi^\eff}\\
& \PMot^\eff(S)
}
\]
commutes, and such that 
\[
\psi^\eff(X,Y):\Hom_{\CM^\eff(S)}(X,Y)\to \Hom_{\PMot^\eff(S)}(X,Y)
\]
 is the inverse of the isomorphism $\phi_{X,Y}$.\\
\\
2. $\psi^\eff$ extends to a fully faithful embedding $\psi:\CM(S)\to \PMot(S)$. \\
\\
3. The maps $\psi^\eff$ and $\psi$ are compatible with the $\otimes$ action of $\Cor_S$. Also, the maps $\psi^\eff:\CM^\eff(S)\to \PMot^\eff(S)_\Q$,  $\psi^\eff:\CM(S)\to \PMot(S)_\Q$ induced by $\psi^\eff$ and $\psi$ are tensor functors.
\end{prop}

\begin{lem}\label{lem:Preduality} For $X\in \Prj/S$, there are maps
\[
\tilde\delta_X:\L^d\to X\otimes X;\ \tilde\epsilon_X:X\otimes X\otimes \L^d
\]
in $\PMot^\eff(S)$ such that the composition
\[
X\otimes\L^d\xrightarrow{\id\otimes\delta_X}X\otimes X\otimes X\xrightarrow{\epsilon_X\otimes\id}
\L^d\otimes X\xrightarrow{\tau}X\otimes\L^d
\]
is the identity.
\end{lem}

\begin{proof} By proposition~\ref{prop:ChowEmbed}, it suffices to construct $\tilde\delta_X$ and $\tilde\epsilon_X$ in $\CM(S)$. Identifying $\Hom_{\CM(S)}(\L^d,X\otimes X)$ with the appropriate summand  of $\CH_d((\P^1)^d\otimes X\times_SX)$, $\tilde\delta_X$ is represented by the cycle $0\times\Delta_X$. Similarly 
\[
\tilde\epsilon_X\in \Hom_{\CM(S)}(X\otimes X,\L^d)\subset \CH_{2d}(X\times_SX\times(\P^1)^d)
\]
is represented by the cycle $\Delta_X\times(\P^1)^d$; the desired identity is a straightforward computation.
\end{proof}

Finally, we have the duality theorem:

\begin{thm} Sending $X\in\Prj/S$ of dimension $d$ over $S$ to $X^D:=X\otimes\L^{-d}$ extends to an exact duality
\[
D: \PMot(S)_\Q^\op\to  \PMot(S)_\Q.
\]
\end{thm}

\begin{proof}  For $X\in\Prj/S$, the maps $\tilde\delta_X$ and $\tilde\epsilon_X$ define maps
\[
\delta_X:S\to X\otimes(X\otimes\L^{-d_X});\ \epsilon_X:(X\otimes\L^{-d})\otimes X\to S.
\]
The identity of lemma~\ref{lem:Preduality} tells us that $(X\otimes\L^{-d},\delta_X, \epsilon_X)$ defines a dual of $X$ in $\PMot(S)_\Q$. Since $\PMot(S)_\Q$ is generated by the Tate twists of objects in $\Prj/S$ (as an idempotently complete triangulated category), this implies that every object of $\PMot(S)_\Q$ admits a dual (see e.g. \cite[Part I, Chap. IV, theorem 1.2.5]{MixMot}). Since a dual of an object in a tensor category is unique up to unique isomorphism, this suffices to define the exact tensor duality involution $D$.
\end{proof}

\section{Moving lemmas} \label{sec:FLVMov} We proceed to extend the Friedlander-Lawson-Voevodsky moving lemmas to the case of a semi-local regular base scheme, and use these results to prove theorem~\ref{thm:Duality} and theorem~\ref{thm:PBF}.

\subsection{Friedlander-Lawson-Voevodsky moving lemmas}

The Friedlander-Law\-son moving lemmas \cite{FriedlanderLawson} are applied in \cite{FriedlanderVoevodsky} to prove the main moving lemmas and duality statements for the cycle complexes of equi-dimensional cycles. This involves a two-step process: one moves cycles on $\P^n$ by one process (which we will not need to recall explicitly) and then one uses a variation of the projecting cone argument to extend the moving process to a smooth projective variety. The main result following from the first step is

\begin{thm}[\hbox{\cite[theorem 6.1]{FriedlanderVoevodsky}}]\label{thm:FV1} Let $k$ be a field, and $n,r, d, e$ integers. Then there are maps of presheaves abelian monoids on $\Sm/k$
\[
H^\pm_U:z_\equi^k(\P^n,r)(U)\to z_\equi^k(\P^n,r)(U\times\A^1)
\]
such that
\begin{enumerate}
\item There is an integer $m\ge0$ such that
\begin{align*}
&i_0^*\circ H^+_U=(m+1)\cdot\id_{z_\equi^k(\P^n,r)(U)}\\
&i_0^*\circ H^-_U=m\cdot\id_{z_\equi^k(\P^n,r)(U)}
\end{align*}
where $i_0:U\to U\times\A^1$ is the zero-section.
\item Let $i_x:\Spec k'\to U\times \A^1\setminus\{0\}$ be a $k'$-point of $U\times \A^1\setminus\{0\}$, where $k'\supset k$ is some extension field of $k$. Then for each $W\in z^k_\equi(\P^n,s)^\eff_{\le d}(k')$, $Z\in 
z^k_\equi(\P^n,r)^\eff_{\le e}(U)$, with $0\le r,s\le n$, $r+s\ge n$, the cycles $x^*(H^\pm_U(Z))$ and $W$ intersect properly on $\P^n_{k'}$.
\end{enumerate}
\end{thm}

The second step involves the projecting cone construction, which enables one to go from moving cycles on $\P^n$ to moving cycles on a smooth projective $X\subset \P^N$ of dimension $n$. We first recall the situation over a base-field $k$.

Let $D_0,\ldots, D_n\subset \P^N$ be degree $f$ hypersurfaces,  and let $F_0,\ldots, F_n$ be the corresponding defining equations. If $X\cap D_0\cap \ldots\cap D_n=\0$, then 
\[
F:=(F_0:\ldots:F_n):\P^N\setminus\cap_{i=0}^nD_i\to  \P^n
\]
restricts to $X$ to define a finite surjective morphism
\[
F_X:X\to \P^n
\]
For $Z$ an effective cycle of dimension $r$ on $X$, we thus have the effective cycle $F_X^*(F_{X*}(Z))$, which can be written as
\[
F_X^*(F_{X*}(Z))=Z+R_F(Z)
\]
for a uniquely determined effective cycle $R_F(Z)$ of dimension $r$. Sending $Z$ to $R_F(Z)$ thus gives a map
\[
R_F:z^k_\equi(X,r)^\eff(k)\to z^k_\equi(X,r)^\eff(k).
\]

Let $\sU_X(d)$ be the open subscheme of $H^0(\P^N,\sO(d))^n$ consisting of $D_0,\ldots, D_n$ such that $\cap_{i=0}^n\cap X=\0$. It is easy to see that the complement of $\sU_X(d)$ in $H^0(\P^N,\sO(d))^n$ is a hypersurface (in fact, the defining equation of this complement is exactly the classical Chow form of $X$, whose coefficients give the Chow point of $X$ in the Chow variety of dimension $n$, degree $\Deg X$ effective cycles on $\P^N$). 

For $F\in \sU_X(f)$, let $\text{ram}_X(F)\subset X$ be the ramification locus of $F_X$. Choose integers $f_0,\ldots, f_n\ge 1$ and let $\sR_X(d_*)\subset \prod_{i=0}^n\sU_X(d_i)$ be the open subscheme consisting of those tuples $(F^0,\ldots, F^n)$, $F^i:=(F^i_0:\ldots :F^i_n)\in \sU_X(d_i)$ for which
\[
\cap_{i=0}^n\text{ram}_X(F^i)=\0.
\]

\begin{rem} It follows from Bertini's theorem that, if $d_i\ge2$ for all $i$, then $\sR_X(d_*)$ is a dense open subscheme of $\prod_{i=0}^n\sU_X(d_i)$ (even in positive characteristic).
\end{rem}

The next moving result is a translation of \cite[proposition 1.3 and theorem 1.7]{FriedlanderLawson}.

\begin{prop}\label{prop:FLMov2} Let $X\subset \P_k^N$ be a smooth closed subscheme of dimension $n$. Fix dimensions $0\le r,s\le n$  with $r+s\ge n$ and a degree $e\ge1$.  Then there is an increasing function
\[
G:=G_{r,s,N,e,X}:\N\to\N,
\]
depending  on $X$ and the integers $r, s, N$ and $e$,
and,  for each $d_*:=(d_0,\ldots, d_n)$,  there is a closed subset $\sB_X(d_*, e, r,s))$ of $\sR_X(d_*)$,
such that the following holds: Let $k'$ be a field extension of $k$, and take $F^0,\ldots, F^n\in \sR_X(d_*)(k')\setminus  \sB_X(d_*, e, r,s)(k')$. Suppose that $d_0\ge G(e)$ and $d_i\ge G(d^n_{i-1}\cdot e)$ for $i=1,\ldots, n$. Then
\begin{enumerate}
\item  For each pair of effective cycles $Z, W$ on $X_{k'}$ with $Z$ of dimension $r$, $W$ of dimension $s$ and both having degree $\le e$, the cycles $W$ and $R_X(F^n)\circ\ldots\circ R_X(F^0)(Z)$ intersect properly on $X_{k'}$.
\item   Let  $Z, W$ be a  pair of effective cycles on $X_{k'}$ with $Z$ of dimension $r$, $W$ of dimension $s$ and both having degree $\le e$. Then for every $j=0,\dots, n$, the cycles
 $W$ and $(F_X^{j*}F^j_{X*})\circ\ldots\circ(F^{0*}_XF^0_{X*})(Z)$ intersect properly on $X_{k'}$.
\end{enumerate}
\end{prop}

\begin{rem}\label{rem:Minimal} If we have two closed subsets $\sB^1_X(d_*, e, r,s), \sB^2_X(d_*, e, r,s)$ satisfying the conditions of proposition~\ref{prop:FLMov2}, then the intersection 
$\sB^1_X(d_*, e, r,s)\cap\sB^2_X(d_*, e, r,s)$ also satisfies proposition~\ref{prop:FLMov2}. Thus, without loss of generality, we may assume that $\sB_X(d_*, e, r)$ is the {\em minimal} closed subset of 
$\sR_X(d_*)$ satisfying the conditions of  proposition~\ref{prop:FLMov2}.

Assuming this to be the case, we have the following description of  $\sB_X(d_*, e, r)$:
\begin{align*}
\sB_X(d_*, e, r)=&\{F^*:=(F^0,\ldots, F^n)\in \sR_X(d_*)\ |\ \exists  \text{ an extension field } k'\supset k(F^*), \\&Z\in z_\equi(X,r)^\eff_{\le e}(k'), 
W\in z_\equi(X,s)^\eff_{\le e}(k') \text{ such that either} \\
&\text{1. }W \text{ and } R_X(F^n)\circ\ldots\circ R_X(F^0)(Z)\text{ do not intersect properly on }\\
&\hskip20pt X_{k'}\\
&\text{or}\\
&\text{2. }W\text{ and }Z\text{ intersect properly on }X_{k'}\text{ and }Z\cdot_{X_{k'}}W\text{ is in } \\
&\hskip20pt z_\equi(X,r+s-n)^\eff(k'), \text{ but this is not the case for }\\
&\hskip20pt W\text{  and }(F_X^{j*}F^j_{X*})\circ\ldots\circ(F^{0*}_XF^0_{X*})(Z)\text{ for some } j, 0\le j\le n.
\}
\end{align*}
To see this, it suffices to see that the set described is closed; for this it suffices to see that this set is closed under specialization. This follows from the fact the the Chow varieties parametrizing effective  cycles of fixed degree and dimension on $X$ are proper over $k$.
\end{rem}

\subsection{Extending the moving lemma} We now consider the two moving lemmas in the setting of a smooth projective scheme over a regular   base-scheme $B$. The extension of the first moving lemma is a direct corollary of theorem~\ref{thm:FV1}.

\begin{cor}\label{cor:FVMov} Let $k$ be a field, and $n,r, d, e$ integers, $B$ a regular $k$-scheme, essentially of finite type over $k$. Then there are maps of presheaves of abelian monoids on $\Sm/B$
\[
H^\pm_U:z_\equi^B(\P_B^n,r)^\eff(U)\to z_\equi^B(\P_B^n,r)^\eff(U\times\A^1)
\]
such that
\begin{enumerate}
\item Let $i_0:U\to U\times\A^1$ be the zero-section.There is an integer $m\ge0$ such that
\begin{align*}
&i_0^*\circ H^+_U=(m+1)\cdot\id_{z_\equi^k(\P^n,r)(U)}\\
&i_0^*\circ H^-_U=m\cdot\id_{z_\equi^k(\P^n,r)(U)}.
\end{align*}
\item Let $i:T\to U\times \A^1\setminus\{0\}$ be a morphism of $B$-schemes. Then for each $W\in z^B_\equi(\P_B^n,s)^\eff_{\le d}(T)$, $Z\in 
z^k_\equi(\P^n,r)^\eff_{\le e}(U)$, with $0\le r,s\le n$, $r+s\ge n$, the cycles $i^*(H^\pm_U(Z))$ and $W$ intersect properly on $\P^n_T$.
\end{enumerate}
\end{cor}

\begin{proof} First of all, it suffices to prove the result if $B$ is of finite type over $k$. Let $p:B\to \Spec k$ be the structure morphism and $p_*:\Sm/B\to \Sm/k$ the  base-restriction functor, $p^*:\Sm/k\to \Sm/B$ the pull-back functor $p^*(Y):=Y\times_kB$. Then we have the evident identifications
\[
z_\equi^B(\P_B^n,r)(U)=z_\equi^k(\P_k^n,r)(p_*U),
\]
induced by the canonical isomorphism $U\times_B\P^n_B\cong p_*U\times_k\P^n_k$. Thus the maps $H^\pm_{p_*U}$ given by theorem~\ref{thm:FV1} give rise to maps
\[
H^\pm_U:z_\equi^B(\P_B^n,r)(U)\to z_\equi^B(\P_B^n,r)(U\times\A^1)
\]
which clearly satisfy our condition (1). 

For (2), if we have two cycles $A\in z_\equi^B(\P^n_B,r)^\eff(T)$ and $A'\in z_\equi^B(\P^n_B,s)^\eff(T)$ such that, for each point $t\in T$, the fibers $t^*(A), t^*(A')$ intersect properly on $\P^n_t$, then $A$ and $A'$ intersect properly on $T\times_B\P^n_B$ and $A\cdot A'$ is in $z_\equi^B(\P^n_B,r+s-n)^\eff(T)$. This, together with theorem~\ref{thm:FV1}(2), proves (2).
\end{proof}

Take $X\in \Prj/B$ with a fixed embedding  $X\hookrightarrow \P^N_B$ over $B$. For each $b\in B$, let $\sC_{X_b}(d)=\P(H^0(\P^N_s,\sO(d)))^{n+1}\setminus \sU_{X_b}(d)$.  Let $V\subset \Proj_B(p_*\sO(d))^{n+1}\times\P^N$ be the incidence correspondence with points $\{(f_0:\ldots:f_n, x)\ |\ x\in\cap_{i=}^n(f_i=0)\}$. Let $\sC_X(d)=p_1(V\cap p_2^{-1}(X)$, and let $\sU_X(d)=\P(p_*\sO(d)))^{n+1}\setminus \sC_X(d)$.  Then clearly $\sC_X(d)\subset \P(H^0(\P^N_s,\sO(d)))^{n+1}$ is a closed subset with fiber $\sC_{X_b}(d)$ over $b\in B$, hence $\sU_X(d)$ has fiber $\sU_{X_b}(d)$ over $b\in B$.

 Let $F_d:\sU_X(d)\times_B X\to \sU_X(d)\times_B \P^n$ be the $\sU_X(d)$-morphism parametrized by 
$\sU_X(d)$, i.e., $F_d((f_0:\ldots:f_n),x):=((f_0:\ldots:f_n), (f_0(x):\ldots, f_n(x)))$. We have the ramification locus $\text{ram}_X(F)_{\sU}\subset \sU_X(d)\times_BX$. 

For a sequence $d_*:=(d_0,\ldots, d_n)$, let 
\[
\sU_X(d_*):=\sU_X(d_0)\times_B\ldots\times_B\sU_X(d_n)
\]
let $p_i:\sU_X(d_*)\to \sU_X(d_i)$ be the projection, and let  $\sF_X(d_*)\subset \sU_X(d_*)\times_BX$ be the intersection
$\cap_{i=0}^n(p_i\times\id_X)^{-1}(\text{ram}_X(F_{d_i})_\sU)$. Finally, let $q:\sU_X(d_*)\times_BX\to \sU_X(d_*)$ be the projection and set
\[
\sR_X(d_*):= \sU_X(d_*)\setminus q(\sF_X(d_*))
\]
Clearly $\sR_X(d_*)$ is an open subscheme of $ \sU_X(d_*)$ with fiber $\sR_{X_b}(d_*)$ over $b\in B$.

 Each $F\in  \sU_X(d_*)(B)$ thus determines a finite $B$-morphism
\[
F_X:X\to \P^n_B
\]
and gives rise to the map
\[
R_{F}:z^B(X,r)^\eff\to z^B(X,r)^\eff
\]
with
\[
F_X^*(F_{X*})(Z)=Z+R_F(Z)
\]
for each $Z\in z^B(X,r)^\eff$.

\begin{lem} Fix integers $e,r,s$ and a sequence of integers $d_0,\ldots, d_n$, with $d_i\ge2$, $0\le r,s\le n$, $r+s\ge n$ and $e\ge1$. For each $b\in B$, we have the closed subset $\sB_{X_b}(d_*, e, r)$ of
$\sR_{X_b}(d_*)\subset \sR_X(d_*)$ given by proposition~\ref{prop:FLMov2}; we assume following remark~\ref{rem:Minimal} that $\sB_{X_b}(d_*, e, r)$ is minimal for each $b$.  Let
\[
\sB_X(d_*, e, r,s):=\cup_{b\in B}\sB_{X_b}(d_*, e, r,s).
\]
Then $\sB_X(d_*, e, r,s)$ is closed in $\sR_X(d_*)$.
\end{lem}

\begin{proof} Let $b\to b'$ be a specialization of points of $B$, $x$ a point of $\sB_{X_b}(d_*, e, r)$, and $x\to x'$ an extension of $b\to b'$ to a specialization of $x$ to a point of 
$\sR_{X_{b'}}(d_*)$. It follows from the description of $\sB_{X}(d_*, e, r,s)$ in remark~\ref{rem:Minimal} that $x'$ is in $\sB_{X_{b'}}(d_*, e, r)$, proving the result.
\end{proof}

 Recall the function $G_{r,s,N,e,X}:\N\to \N$ from proposition~\ref{prop:FLMov2}, defined for $X\in \Prj/k$, with a given embedding $X\hookrightarrow\P^N_k$, and integers $r,s,e$. If now we have $X\in \Prj/B$ with a given embedding over $B$, $X\hookrightarrow\P^N_B$, let $b_1,\ldots, b_m$ be the closed points of $B$ and define
 \[
 G_{r,s,N,e,X}(m):=\max_i\{G_{r,s,N,e,X_{b_i}}(m)\}.
 \]

 \begin{prop}\label{prop:FLMov2Ext} Let $B$ be a semi-local regular $k$-scheme, with $k$ an infinite field.  Take $X$ in $\Prj/B$ of relative dimension $n$ over $B$, with a fixed closed immersion $X\hookrightarrow \P^N_B$ over $B$. Fix dimensions $0\le r,s\le n$ and a degree $e\ge1$ with $r+s\ge n$ and let $G:=G_{r,s,N,e,X}$. Then   for all tuple of integers $d_0,\ldots, d_n$ with $d_0\ge G(e)$,  $d_i\ge G(d_{i-1}^n\cdot e)$ for $i=1,\ldots, n$, there is a point $(F^0,\ldots, F^n)\in \sR_X(d_*)(B)\setminus \sB_X(d_*)(B)$.
 \end{prop}
 
 \begin{proof}
We note that the fiber of $\sR_X(d_*)$ over $b\in B$ is $\sR_{X_b}(d_*)$ and similarly for $\sB_X(d_*)$.  
Then $\sR_X(d_*)\setminus \sB_X(d_*)\to B$ is an open subscheme of the product of projective spaces $\prod_{i=0}^n\Proj_B(p_*\sO_{\P^N}(d_i))$ over $B$; by  proposition~\ref{prop:FLMov2}, the fiber
$\sR_X(d_*)\setminus \sB_X(d_*)_{b}$ is non-empty for each closed point $b$ of $B$. As $k$ is assumed infinite, $B$ has infinite residue fields over each of its closed points. Thus, for each closed point $b$ of $B$, there is a $k(b)$-point $F^*_b$ in  $\sR_X(d_*)_{b}\setminus \sB_X(d_*)_{b}$. Since $B$ is semi-local,  there is a $B$-point $F^*$ of 
$\prod_{i=0}^n\Proj_B(p_*\sO_{\P^N}(d_i))$ with $F^*(b)=F^*_b$ for each closed point $b$; $F^*$ is automatically a $B$ point of $\sR_X(d_*)\setminus \sB_X(d_*)$, using again the fact that $B$ is semi-local.
\end{proof}

\begin{rem}\label{rem:GoodIntersection} Take $(F^0,\ldots, F^n)\in \sR_X(d_*)(B)\setminus \sB_X(d_*)(B)$ as given by proposition~\ref{prop:FLMov2Ext}. Then  for each $B$-scheme $T\to B$, and each pair of cycles 
\[
Z\in z^B(X,r)^\eff_{\le e}(T),\ W\in  z^B(X,s)^\eff_{\le e}(T), 
\]
the cycles  $W$ and $R_X(F^n)\circ\ldots\circ R_X(F^0)(Z)$ intersect properly on $X_T$ and the intersection $W\cdot_{X_T}R_X(F^n)\circ\ldots\circ R_X(F^0)(Z)$ is in 
$z^B(X,r+s-n)^\eff(T)$.

Indeed, this follows from the fact that the operation $R_X(F)$ is compatible with taking fibers, hence, for all $t\in T$, the cycles $W_t$ and $R_{X_t}(F^n_t)\circ\ldots\circ R_{X_t}(F_t^0)(Z_t)$ intersect properly on $X_t$. Thus $W$ and $R_X(F^n)\circ\ldots\circ R_X(F^0)(Z)$ intersect properly on $X_T$ and the intersection $W\cdot_{X_T}R_X(F^n)\circ\ldots\circ R_X(F^0)(Z)$ is equi-dimensional (of dimension $r+s-n$) over $T$.

Similarly, if $W$ and $Z$ intersect properly on $X_T$ and $W\cdot_{X_T}Z$ is in $z^B(X,r+s-n)^\eff(T)$, then for each $j=0,\ldots, n$, $W$ and $Z_j:=(F^{j*}_XF^j_{X*})\circ (F^{0*}_XF^0_{X*})(Z)$ intersect properly on $X_T$ and $W\cdot Z_j$ is in $z^B(X,r+s-n)^\eff(T)$.
\end{rem}

We can now prove our extension of \cite[theorem 6.3]{FriedlanderVoevodsky}.

\begin{thm}\label{thm:FVMovExt} Let $k$ be a field, $B$ a regular $k$-scheme, essentially of finite type over $k$, $X\in\Prj/B$ of dimension $n$ over $B$ with a given embedding $X\hookrightarrow\P^N_B$. Let $r,s, e,d$ be given integers with $e\ge1$, $0\le r,s\le n$, $r+s\ge n$. Then is a map of presheaves
\[
H_{X,U}:z^B_\equi(X,r)(U)\to z^B_\equi(X,r)(U\times\A^1);\ U\in \Sm/B,
\]
such that
\begin{enumerate}
\item   Let $i_0:U\to U\times\A^1$ be the zero-section.Then 
$i_0^*\circ H_{X,U }=\id_{z_\equi^B(X,r)(U)}$.
\item Let $i:T\to U\times \A^1\setminus\{0\}$ be a morphism of $B$-schemes. Then for each $W\in z^B_\equi(X,s)^\eff_{\le e}(T)$, $Z\in 
z^B_\equi(X,r)^\eff_{\le e}(U)$, with $0\le r,s\le n$, $r+s\ge n$, the cycles $i^*(H_{X,U}(Z))$ and $W$ intersect properly on $X_T$ and $i^*(H_{X,U}(Z))\cdot_{X_T} W$ is in 
$z^B_\equi(X,r+s-n)^\eff(T)$.
\item Let $i:T\to U\times \A^1$ be a morphism of $B$-schemes and take 
\begin{align*}
&W\in z^B_\equi(X,s)^\eff_{\le e}(T), \\
&Z\in 
z^B_\equi(X,r)^\eff_{\le e}(U),
\end{align*}
 with $0\le r,s\le n$, $r+s\ge n$. Suppose that $(p_U\circ i)^*(Z)$ and $W$ intersect properly   on $X_T$ and $(p_U\circ i)^*(Z)\cdot_{X_T} W$ is in 
$z^B_\equi(X,r+s-n)^\eff(T)$. Then  the cycles $i^*(H_{X,U}(Z))$ and $W$ intersect properly on $X_T$ and $i^*(H_{X,U}(Z))\cdot_{X_T} W$ is in 
$z^B_\equi(X,r+s-n)^\eff(T)$.
\end{enumerate}
\end{thm}

\begin{proof}   The proof follows that in \cite{FriedlanderVoevodsky} and \cite{FriedlanderLawson}. If $k$ is a finite field, we use the fact that $k$ admits a infinite pro-$l$ extension $k_l$ for every prime $l$ prime to the characteristic, plus the usual norm argument, to reduce to the case of an infinite field. We may assume that $B$ is irreducible. Thus for each $b\in B$, the embedded subscheme $X_b\subset\P^N_b$ has degree $d_X$ independent of $b$. 

We denote the maps 
\[
H^\pm_U:z_\equi^B(\P_B^n,r)^\eff(U)\to z_\equi^B(\P_B^n,r)^\eff(U\times\A^1)
\]
given by corollary~\ref{cor:FVMov} by $H^\pm_U(d)$, noting explicitly the dependence on the degree $d$. Similarly,we write $m(d)$ for the integer $m$ that appears in corollary~\ref{cor:FVMov}(1). 

 Choose some $F^*\in \sR_X(d_*)(B)\setminus \sB_X(d_*)(B)$, as given by  proposition~\ref{prop:FLMov2Ext}.   Let $p_U:U\times\A^1\to U$ be the projection. We have the maps of presheaves
 \begin{align*}
& F^i_{X*}:z^B_\equi(X,r)^\eff_d\to z^B_\equi(\P^n,r)^\eff_d\\
& F^{i*}_{X}:z^B_\equi(\P^n,r)^\eff_d\to z^B_\equi(X,r)^\eff_{d_i^ned}
\end{align*}
and similarly without the degree restrictions. Also, for $W\in z^B_\equi(X,s)^\eff(U)$, $Z\in z^B_\equi(\P^n_B,r)^\eff(U)$, $W\cdot_{U\times_BX}F^{i*}_X(Z)$ is defined and is in $z^B_\equi(X,s+r-n)^\eff(U)$ if and only if $F^i_{X*}(W)\cdot_{U\times_B\P^n_B}Z$ is defined and is in  
$z^B_\equi(\P^n_B,s+r-n)^\eff(U)$. Finally, we note that, for $Z\in z^B_\equi(X,r)^\eff_{d}(U)$, 
$R_{F^i}(Z)$ is in $z^B_\equi(X,r)^\eff_{d(d_i^ne-1)}(U)$.

We now define a sequence of maps of presheaves on $\Sm/B$,
\[
H_{X,U,j}^\pm:z^B_\equi(X,r)^\eff(U)\to z^B_\equi(X,r)^\eff(U\times\A^1);\ U\in \Sm/B.
\]
Define the integers $\delta_j$ inductively by $\delta_0=e$, and
\[
\delta_j=(m(\delta_{j-1})+1)d_i^ne\delta_{j-1}
\]
for $j\ge1$. These numbers have the property that, if $Z$ is in $z^B_\equi(X,r)^\eff_{\le \delta_j}(U)$, then
$(m(\delta_j)+1)\cdot R_{F^j}(Z)$  is in $z^B_\equi(X,r)^\eff_{\le \delta_{j+1}}(U)$. For $Z\in z^B_\equi(X,r)^\eff(U)$, define
\[
H_{X,U,0}^\pm(Z):=F^{0*}_X\circ H^\pm_U(\delta_0)\circ F^0_{X*}(Z).
\]
For $j=1,\ldots, n$, let
\[
H_{X,U,j}^\pm(Z):=F^{j*}_X\circ H^\pm_U(\delta_j)\circ F^j_{X*}(R_{F^{j-1}}\circ\ldots\circ R_{F^0}(Z)).
\]
Finally, we let
\[
H_{X,U,n+1}^+(Z)=p_U^*(R_{F^{n}}\circ\ldots\circ R_{F^0}(Z))
\]
and $H_{X,U,n+1}^-(Z)=0$. Set
\[
H_{X,U}(Z):=\sum_{j=0}^{n+1}(-1)^j(H^+_{X,U,j}(Z)-H^-_{X,U,j}(Z)).
\]
Thus, $U\mapsto H_{X,U}$ defines a map of presheaves on $\Sm/B$
\[
H_{X}:z^B_\equi(X,r)\to z^B_\equi(X,r)((-)\times\A^1).
\]

We have
\[
i_0^*(H_{X,U,j}^+(Z)-H_{X,U,j}^-(Z))=\begin{cases}
F^{*0}_X\circ F^0_{X*}(Z)&\text{ for }j=0\\
F^{j*}_X\circ F^j_{X*}(R_{F^{j-1}}\circ\ldots\circ R_{F^0}(Z))&\text{ for }j=1,\ldots, n\\
R_{F^{n}}\circ\ldots\circ R_{F^0}(Z)&\text{ for }j=n+1.
\end{cases}
\]
Since $F^{j*}_X\circ F^j_{X*}=\id+R_{F^j}$, and $i_0^*\circ (H^+_U-H^-_U)=\id$, this implies that
 \[
 i_0^*\circ H_{X,U}(Z)=Z
 \]
 for all $Z\in z^B_\equi(X,r)(U)$. This verifies the property (1).
 
For (2), let $i:T\to U\times \A^1\setminus\{0\}$ be a morphism of $B$-schemes, take $W\in z^B_\equi(X,s)^\eff_{\le e}(T)$ and  $Z\in 
z^k_\equi(X,r)^\eff_{\le e}(U)$, with $0\le r,s\le n$, $r+s\ge n$. Take $j=0,\ldots, n$. By  corollary~\ref{cor:FVMov} and our remarks above,  the cycles $i^*(H^\pm_{X,U,j}(Z))$ and $W$ intersect properly on $X_T$ and $i^*(H^\pm_{X,U,j}(Z))\cdot_{X_T} W$ is in 
$z^B_\equi(X,r+s-n)^\eff(T)$. By remark~\ref{rem:GoodIntersection}, the same holds for $j=n+1$. This proves (2).

For (3), under the given assumptions for $Z$ and $W$, it follows from remark~\ref{rem:GoodIntersection} that
$W$ and $Z_j:=i^*\circ p_U^*(F^{j*}_XF^j_{X*})\circ (F^{0*}_XF^0_{X*})(Z)$ intersect properly on $X_T$ for all $j=0,\ldots, n$. Since $H^\pm_{X,U,j}(Z)$ is effective and 
$i_0^*\circ H^\pm_{X,U,j}(Z)$ plus some effective cycle  is a multiple of $F^{j*}_XF^j_{X*})\circ (F^{0*}_XF^0_{X*})(Z)$, (3) follows from (2).
\end{proof}

Fix a field extension $k'$ of $k$, and let $Y\subset \P^N_{k'}$ be a closed subset. Let $\sC_Y(s,e)$ be the Chow scheme parametrizing effective cycles of dimension $s$ and degree $\le e$ on $Y$. For $L\supset k'$ an extension field and $W$ an effective dimension $s$ cycle on $\P^N_L$ of degree $\le e$, and with support in $Y_L$, we let $\chow(W)\in \sC_(s,e)(\bar{L})$ denote the Chow point of $W$. 

For $Y\subset \P^N_B$, we have the relative Chow scheme $\sC_{Y/B}(s,e)\to B$, with (reduced) fiber over $b\in B$ the Chow scheme of $Y_b\subset \P^N_b$.

\begin{definition} For a fixed regular noetherian base-scheme $B$, take $Y\in\Prj/B$ and fix an embedding $Y\hookrightarrow\P^N_B$ over $B$. Let $\sC\subset \sC_{Y/B}(s,e)$ be any collection of locally closed subsets of $\sC_{Y/B}(s,e)$. 

Let 
\[
z^B_\equi(Y,r)_\sC(X)\subset z_\equi^B(Y,r)(X)
\]
 be the subgroup generated by integral closed subschemes  $Z\subset X\times_BY$ such that, 
\begin{enumerate}
\item $Z$ is in   $z^B_\equi(Y,r)(X)$.
\item Take $W\in z_\equi^B(Y,s)^\eff_{\le e}(X)$ and suppose that, for each $x\in X$, the cycle $i_x^*(W)$ on $Y_{k(x)}$  has Chow point $\chow(W)$ in $\sC(\bar{k(x)})$. Then the intersection $W\cdot_{X\times_BY} Z$ is defined and is in $z^B_\equi(Y,r+s-n)(X)$.
\end{enumerate}
This defines the  subpresheaf $z^B_\equi(Y,r)_\sC$ of $z^B_\equi(Y,r)$. The subpresheaves
\[
z^B_\equi(Y,r)^\eff_{\sC,\le e}\subset z^B_\equi(Y,r)^\eff_{\le e};\
z^B_\equi(Y,r)_{\sC,\le e}\subset z^B_\equi(Y,r)_{\le e}, 
\]
etc., are defined similarly.

We let $C^B(Y,r)_\sC(X)\subset C^B(Y,r)(X)$ be the subcomplex associated to the cubical object
\[
n\mapsto  z^B_\equi(Y,r)_\sC(X\times\square^n) 
\]
This gives us the presheaf of subcomplexes  $C^B(Y,r)_\sC\subset C^B(Y,r)$.
The subcomplex $C^B(Y,r)_{\sC,\le e}\subset C^B(Y,r)_{\le e}$ is defined similarly.
\end{definition}

\begin{thm}\label{thm:MLExt} Let $B$ be a semi-local regular  scheme, essentially of finite type over some  field $k$. Take $X\in\Prj/B$ of relative dimension $n$ over $B$,  and integers $r,s,e$ with $e\ge1$, $0\le r,s\le n$, $r+s\ge n$. Fix an embedding $X\hookrightarrow \P^N_B$.  Let $\sC\subset\sC_X(s,e)$ be a collection of locally closed subsets. Then the inclusion $C^B(X,r)_\sC\subset C^B(X,r)$ is a quasi-isomorphism.
\end{thm}

\begin{proof} Since
\[
C^B(X,r)_\sC=\cup_{e\ge1}C^B(X,r)_{\sC,\le e};\ C^B(X,r)=\cup_{e\ge1}C^B(X,r)_{\le e},
\]
it suffices to show that the map
\[
\iota:\frac{C^B(X,r)_{\le e}}{C^B(X,r)_{\sC,\le e}}\to
\frac{C^B(X,r)}{C^B(X,r)_{\sC}}
\]
induced by the inclusions $C^B(X,r)_{\sC,\le e}\subset C^B(X,r)_{\sC}$ and
$C^B(X,r)_{\le e}\subset C^B(X,r)$
gives the zero-map on homology.

Let
\[
H_X: z^B_\equi(X,r)\to z^B_\equi(X,r)((-)\times\A^1)
\]
be the map given by theorem~\ref{thm:FVMovExt} for the given values of $r,s, e$. By theorem~\ref{thm:FVMovExt}(3),  $H_X$ restricts to a map
\[
H_X:z^B_\equi(X,r)_\sC\to z^B_\equi(X,r)((-)\times\A^1)_\sC
\]
By   theorem~\ref{thm:FVMovExt}(2), $i_1^*\circ H_X$ defines a map
\[
G_X:z^B_\equi(X,r)\to z^B_\equi(X,r)_\sC
\]
and by  theorem~\ref{thm:FVMovExt}(1), $i_0^*\circ H_X=\id$. 

Identifying $U\times\A^1\times_B X\times\square^n$ with $U\times_BX\times\square^{n+1}$ via the exchange of factors
\[
U\times\A^1\times_B X\times\square^n\to U\times_B X\times\square^n\times \A^1=
U\times_BX\times\square^{n+1},
\]
$(-1)^nH_X$ gives us maps
\[
h_X^{-n}: C^B(X,r)^{-n}_{\le e}\to C^B(X,r)^{-n-1},\ h_{X\sC}^{-n}: C^B(X,r)^{-n}_{\le e,\sC}\to C^B(X,r)^{-n-1}_\sC
\]
and thus induces a degree -1 map
\[
\bar{h}_X:\frac{C^B(X,r)_{\le e}}{C^B(X,r)_{\le e,\sC}}\to
\frac{C^B(X,r)}{C^B(X,r)_{\sC}}
\]
which gives a homotopy between the map $\iota$ and the zero-map, completing the proof.
\end{proof}

Theorem~\ref{thm:MLExt} allows us to prove theorem~\ref{thm:Duality}:

\begin{cor}\label{cor:duality} Let $B$ be a semi-local regular  scheme, essentially of finite type over some  field $k$. Take $X,Y\in\Prj/B$, $U\in\Sm/B$ and let $p=\dim_BX$. Then for $0\le r\le\dim_BY$, the map
\[
\int_X:C^B(Y,r)(U\times_BX)\to C^B(X\times_BY,r+p)(U)
\]
is a quasi-isomorphism.
\end{cor}

\begin{proof} Let $s=\dim_BY$. Fix embeddings of $Y, X$ into some $\P^N_B$, with $Y$ of degree $e$. The Segre embedding $\P^N\times\P^N\to \P^M$ gives us an embedding of $X\times_BY$ in $\P^M_B$ such that $x\times Y$ has degree $e$ for each $x\in X$.  Take $\sC\subset \sC_{X\times_BY/B}(s, e)$ to be the family of cycles $x\times Y$, $x\in X$.  Note that the map $\int_X$ identifies $z_\equi(Y,r)(U\times_BX)$ with $z_\equi(X\times_BY, r+p)_\sC(U)$, and thus gives an isomorphism
\[
\int_X:C^B(Y,r)(U\times_BX)\to C^B(X\times_BY,r+p)_\sC(U).
\]
Since $r+p+s\ge p+s= \dim_BX\times_BY$, we may apply theorem~\ref{thm:MLExt} to conclude that the inclusion
\[
C^B(X\times_BY,r+p)_\sC(U)\subset
C^B(X\times_BY,r+p)(U)
\]
is a quasi-isomorphism, completing the proof.
\end{proof}

To complete this section, we prove the projective bundle formula.

\begin{proof}[Proof of theorem~\ref{thm:PBF}] We proceed by induction on $n$. Let $\pi:Y\times\P^n\to Y$ be the projection. By theorem~\ref{thm:MLExt}, the inclusion
\[
C^B(Y\times\P^n,r)_{\{Y\times\P^{n-1}\}}\subset C^B(Y\times\P^n,r)
\]
is a quasi-isomorphism for $r\ge1$ Here $\P^{n-1}\subset \P^n$ is the hyperplane $X_n=0$. We have the intersection map
\[
i_{\P^{n-1}}^*:  z_\equi(Y\times\P^n,r)_{\{Y\times\P^{n-1}\}}\to
z_\equi(Y\times\P^{n-1},r-1) 
\]
and the cone-map
\[
C_{p_0}(-):z_\equi(Y\times\P^{n-1},r-1) \to z_\equi(Y\times\P^n,r)_{\{Y\times\P^{n-1}\}}
\]
where $p_0:=(0:\ldots, 0:1)$. Let $i_0:\Spec k\to \P^n$ be the inclusion of the point $p_0$ and let $\pi:\P^n\to\Spec k$ be the projection. Let
\[
\mu:\P^n\times(\A^1\setminus\{0\})\to \P^n
\]
be the multiplication map
\[
\mu((x_0:,\ldots:x_n), t):=(x_0:\ldots :x_{n-1}:tx_n).
\]
We have as well the natural transformation
\[
H_U:z_\equi(Y\times\P^n,r)_{\{Y\times\P^{n-1}\}}(U)\to
z_\equi(Y\times\P^n,r)_{\{Y\times\P^{n-1}\}}(U\times\A^1)
\]
which sends a cycle $Z\in z_\equi(Y\times\P^n,r)_{\{Y\times\P^{n-1}\}}(U)$ to the closure of $\mu^*(Z)$. One checks that this is well-defined and satisfies
\[
i_1^*\circ H_U=\id_{z_\equi(Y\times\P^n,r)_{\{Y\times\P^{n-1}\}}(U)};\ 
i_0^*\circ H_U=C_{p_0}\circ i_{\P^{n-1}}^*+\alpha_0\circ\pi_*,
\]
where   $\alpha_0=i_{0*}$.
As in the proof of theorem~\ref{thm:MLExt}, the maps 
\[
h_U:=(-1)^nH_U:C^B_n(Y\times\P^n,r)_{\{Y\times\P^{n-1}\}}(U)\to
C^B_{n+1}(Y\times\P^n,r)_{\{Y\times\P^{n-1}\}}(U)
\]
define a homotopy between the identity and $C_{p_0}\circ i_{\P^{n-1}}^*+\alpha_0\circ\pi_*$. 

On the other hand, since $p_0\cap\P^{n-1}=\0$, $ i_{\P^{n-1}}^*$ is the zero map on the image of $\alpha_0$. Also, since $\pi(C_{p_0}(Z))$ has dimension $<\dim_BZ$ for any equi-dimensional closed subset $Z$ of $Y\times_B\P^{n-1}$, $\pi_*$ is zero on the image of $C_{p_0}$.  Finally, 
\[
\pi_*\circ\alpha_0=\id;\ i_{\P^{n-1}}^*\circ C_{p_0}=\id
\]
hence, for $r\ge n\ge 1$, 
\[
\alpha_0+C_{p_0}:C^B(Y,r-n)\oplus C^B(Y\times\P^{n-1},r-1)\to
C^B(Y\times\P^n,r)_{\{Y\times\P^{n-1}\}}
\]
is a homotopy equivalence. By induction
\[
\sum_{j=0}^{n-1}\alpha_j:C^B(Y,r-j-1)\to  C^B(Y\times\P^{n-1},r-1)
\]
is a quasi-isomorphism; since $C_{p_0}\circ \alpha_j=\alpha_{j+1}$ (where we use $p_0$ and the  subspaces $C_{p_0}(\P^j\subset \P^{n-1})$ for the flag of linear subspaces of $\P_k^n$ needed to define the maps $\alpha_j$), the induction goes through.
\end{proof}

\section{Smooth motives and motives over a base}\label{sec:Motives}
Cisinski-D\'eglise have defined a triangulated tensor category of effective motives over a base-scheme $S$, $\DM^\eff(S)$, and a triangulated tensor category of motives over $S$, $\DM(S)$, with an exact tensor functor $\DM^\eff(S)\to \DM(S)$ that inverts $\otimes\L$. In this section, we show how to define exact functor
\[
\rho^\eff_S:\PMot^\eff(S)\to \DM^\eff(S),\ \rho_S:\PMot(S)\to \DM(S)
\]
which give equivalences of $\PMot^\eff(S)$, $\PMot(S)$ with the full triangulated subcategories of $\DM^\eff(S)$ and $\DM(S)$ generated by the motives of smooth projective $S$ schemes, resp. the Tate twists of smooth projective $S$-schemes.  Working with $\Q$ coefficients, we have the same picture, with $\rho^\eff_S$, $\rho_S$ replaced by exact tensor functors
\[
\rho^\eff_{S\Q}:\PMot^\eff(S)_\Q\to \DM^\eff(S)_\Q^\natural,\ \rho_{S\Q}:\PMot(S)_\Q\to \DM(S)_\Q^\natural,
\]
giving analogous equivalences.

\subsection{Cisinski-D\'eglise categories of motives} We summarize the main points of the construction of the category $\DM^\eff(S)$ of effective motives over $S$, and the category $\DM(S)$ of motives over $S$, from \cite{CisinskiDeglise}. Although $S$ is allowed to be a quite general scheme in \cite{CisinskiDeglise}, we restrict ourselves to the case of  a  base-scheme $S$ that is  separated, smooth and essentially of finite type over a field. 

Define the abelian category of {\em presheaves with transfer} on $\Sm/S$, $\PST(S)$, as the category of presheaves of abelian groups on $\Cor_S$. We have the representable presheaves $\Z_S^{tr}(Z)$ for $Z\in \Sm/S$ by 
$\Z_S^{tr}(Z)(X):=\Cor_S(X,Z)$ and pull-back maps given by the composition of correspondences. 

One gives the category of complexes $C(\PST(S))$  the  {\em Nisnevich local} model structure (which we won't need to specify). The homotopy category is equivalent to the (unbounded) derived category $D(\Sh^{tr}_\Nis(S))$, where $\Sh^{tr}_\Nis(S)$ is the full subcategory of $\PST(S)$  consisting of the presheaves with transfer which restrict to Nisnevich sheaves on $\Sm/S$.  

The  operation 
\[
\Z_S^{tr}(X) \otimes_S^{tr}\Z_S^{tr}(X') :=\Z_S^{tr}(X\times_SX') 
\]
extends to  a tensor structure $\otimes^{tr}_S$ making $\PST(S)$ a tensor category: one forms the {\em canonical left resolution} $\sL(\sF)$ of a presheaf $\sF$ by taking the canonical surjection \[
\sL_0(\sF):=\bigoplus_{ X\in \Sm/S, s\in \sF(X)}\Z_S^{tr}(X)\xrightarrow{\phi_0} \sF
\]
setting $\sF_1:=\ker\phi_0$ and iterating. One then defines
\[
\sF\otimes_S^{tr}\sG:=H_0(\sL(\sF)\otimes_S^{tr} \sL(\sG))
\]
noting that $\sL(\sF)\otimes_S^{tr} \sL(\sG)$ is defined since both complexes are degreewise direct sums of representable presheaves.

The restriction of $\otimes^{tr}_S$ to the subcategory of cofibrant objects in $C(\Sh^{tr}_\Nis(S))$ induces a tensor operation  $\otimes^L_S$ on $D(\Sh^{tr}_\Nis(S))$ which makes $D(\Sh^{tr}_\Nis(S))$ a triangulated tensor category.

\begin{definition}[\hbox{\cite[definition 10.1]{CisinskiDeglise}}]   $\DM^\eff(S)$ is  the localization of the triangulated category $D(\Sh^{tr}_\Nis(S))$ with respect to the localizing category generated by the complexes $\Z_S^{tr}(X\times\A^1)\to \Z_S^{tr}(X)$.  Denote  by $m_S(X)$ the image of $\Z_S^{tr}(X)$ in $\DM^\eff(S)$.
\end{definition}

\begin{rem} 1.  $\DM^\eff(S)$ is a triangulated tensor category  with tensor product $\otimes_S$  induced from the tensor product   $\otimes^L_S$  via the localization map
\[
Q_S:D(\Sh^{tr}_\Nis(S))\to \DM^\eff(S),
\]
   and satisfying $m_S(X)\otimes_S  m_S(Y)=m_S(X\times_SY)$.\\
\\
2. There is a   model category $C(\PST_{\A^1}(S))$ having  $C(\PST(S))$ as underlying category, defined as the left Bousfield localization of $C(\PST(S))$ with respect to the  following
complexes
\begin{enumerate}
\item For each {\em elementary Nisnevich square}  with $X\in\Sm/S$:
\[
\xymatrix{
W\ar@{^{(}->}[r]\ar@{=}[d]&X'\ar[d]^f\\
W\ar@{^{(}->}[r]&X}
\]
one has the complex 
\[
\Z^{tr}_S(X'\setminus W)\to \Z^{tr}_S(X\setminus W)\oplus \Z^{tr}_S(X')\to \Z^{tr}_S(X)
\]
Recall that the square above is an elementary Nisnevich square if $f$ is \'etale, the horizontal arrows are closed immersions of reduced schemes and the square is cartesian.
\item For $X\in \Sm/S$, one has the complex $\Z_S^{tr}(X\times\A^1)\to \Z_S^{tr}(X)$.
\end{enumerate}

 The homotopy category of $C(\PST_{\A^1}(S))$ is equivalent to $\DM^\eff(S)$.
\end{rem}

\begin{definition}\label{def:TateObj} Let $T^{tr}$ be the presheaf with transfers
\[
T^{tr}:=\coker(\Z_S^{tr}(S)\xrightarrow{i_{\infty*}}\Z_S^{tr}(\P_S^1))
\]
and let $\Z_S(1)$ be the image in $\DM^\eff(S)$ of $T^{tr}[-2]$. Let 
\[
\otimes T^{tr}: C(\PST(S))\to C(\PST(S))
\]
be the functor $C\mapsto C\otimes_S^{tr}T^{tr}$.
\end{definition}

Let $\Spt_{T^{tr}}(S)$ be the model category of $\otimes T^{tr}$ spectra in $C(\PST_{\A^1}(S))$, i.e., objects are sequence $E:=(E_0, E_1, \ldots)$, $E_n\in C(\PST(S))$, with bonding maps
\[
\epsilon_n:E_n\otimes^{tr}_S T^{tr}\to E_{n+1}.
\]
Morphisms are given by sequences of maps in $C(\PST(S))$ which strictly commute with the respective bonding maps.

 The model structure on the category of $T^{tr}$-spectra is defined by following the construction of Hovey \cite{HoveySymSpec}. The weak equivalences are the {\em stable weak equivalences}: for each $E\in \Spt_{T^{tr}}(S)$ there is a canonical fibrant model $E\to E^f$, where $E^f:=(E^f_0, E^f_1,\ldots)$ with each $E^f_n$ fibrant in 
$C(\PST_{\A^1}(S))$ and the map
\[
E^f_n\to \sHom(T^{tr}, E^f_{n+1})
\]
adjoint to the bonding map $E^f_n\otimes^{tr}_S T^{tr}\to E^f_{n+1}$ is a weak equivalence in the model category  $C(\PST_{\A^1}(S))$.

\begin{definition} The category of triangulated motives over $S$, $\DM(S)$, is the homotopy category of $\Spt_{T^{tr}}(S)$.
\end{definition}

We will use the following results from \cite{CisinskiDeglise}.
\begin{thm}[\hbox{\cite[section~10.3, corollary 6.12]{CisinskiDeglise}}]\label{thm:MotivesMain} Suppose that $S$ is in $\Sm/k$ for a field $k$, take $X$ in $\Sm/S$, and let $m_k(X)$, $m_S(X)$ denote the motives of $X$ in $\DM(k)$, $\DM(S)$, respectively. Then there is a natural isomorphism
\[
\Hom_{\DM(S)}(m_S(X), \Z(n)[m])\cong \Hom_{\DM(k)}(m_k(X), \Z(n)[m])
\]
In addition, the natural map
\[
\colim_N\Hom_{\DM^\eff(S)}(m_S(X)\otimes\Z(N), \Z(n+N)[m])\to 
\Hom_{\DM(S)}(m_S(X), \Z(n)[m])
\]
is an isomorphism. Finally, the cancellation theorem holds in this setting: the natural map
\[
\Hom_{\DM^\eff(S)}(m_S(X), \Z(n)[m])]\to\Hom_{\DM^\eff(S)}(m_S(X)\otimes\Z(1), \Z(n+1)[m])
\]
is an isomorphism.
\end{thm}

\begin{rems} \label{rem:MotCohIso} 1. By   \cite{VoevodskyChow}
 $\Hom_{\DM(k)}(m_k(X), \Z(n)[m])$ is motivic cohomology in the sense of Voevodsky \cite[chapter~V]{FSV},  that is  
\[
\Hom_{\DM(k)}(m_k(X), \Z(n)[m])=H^m(X,\Z(n))\cong \CH^n(X,2n-m).
\]
In fact, this follows immediately from \cite{VoevodskyChow} in the case of a perfect field; the general case follows by using the usual trick of viewing a field $k$ as a limit of finitely generated extensions of the prime field $k_0$, and the fact that,  for a projective systems of schemes $S_\alpha$ with affine transition maps,  the functor 
\[
S_\alpha\mapsto \Hom_{\DM(S)}(m_{S_\alpha}(X_{S_\alpha}), \Z(n)[m])
\]
 transforms the projective limit $\lim S_\alpha$ to the inductive limit (see \cite[prop 4.2.19]{Deglise}).\\
\\
2. The isomorphisms in theorem~\ref{thm:MotivesMain} are proven as follows: Let $p:S\to\Spec k$ be the (smooth) structure morphism. The limit argument mentioned above reduces us to the case of $k$ a perfect field. The restriction of base functor induces an exact functor
\[
p_{\sharp}:\DM^\eff(S)\to \DM^\eff(k)
\]
and the pull-back $X\mapsto X\times_kS$ induces an exact functor
\[
p^*:\DM^\eff(k)\to \DM^\eff(S),
\]
right adjoint to $p_{\sharp}$. In addition, one has $p^*(\Z(n))\cong \Z(n)$, and $p_\sharp(m_S(X)\otimes\Z(N))\cong m_k(p_*X)\otimes\Z(N)$ for $X\in\Sm/S$, $N\in \Z$. We have similar functors on $\DM(S)$, $\DM(k)$. Thus, we have
\begin{align*}
\Hom_{\DM^\eff(S)}&(m_S(X)\otimes\Z(N), \Z(n+N)[m])\\
&\cong
\Hom_{\DM^\eff(S)}(m_S(X)\otimes\Z(N), p^*\Z(n+N)[m])\\
&\cong
\Hom_{\DM^\eff(k)}(p_\sharp (m_S(X)\otimes\Z(N)), \Z(n+N)[m])\\
&\cong
\Hom_{\DM^\eff(k)}(m_k(p_*X)\otimes\Z(N), \Z(n+N)[m]),
\end{align*}
and similarly with $\DM^\eff$ replaced by $\DM$. This already proves the first isomorphism. For the next two,  the above isomorphisms reduce us to the case of $S=\Spec k$. Also, $\DM^\eff_-(k)$ is a full subcategory of $\DM^\eff$. Finally,
 Voevodsky's identification of motivic cohomology with the higher Chow groups \cite{VoevodskyChow}, together with the projective bundle formula, gives the limited version of the cancellation theorem we need to finish the proof.
 \end{rems}

 \subsection{Tensor structure}
The tensor structure on $C(\PST(S))$ induces a ``tensor  operation" on the spectrum category by the usual device of choosing a cofinal subset $\N\subset \N\times\N$, $i\mapsto (n_i,m_i)$, with $n_{i+1}+m_{i+1}=n_i+m_i+1$ for each $i$: each pair of $T^{tr}$ spectra $E:=(E_0, E_1,\ldots)$ and $F:=(F_0, F_1,\ldots)$ gives rise to a $T^{tr}$ bispectrum
\[
E\boxtimes^{tr}_S F:=\begin{pmatrix}&\vdots&\\\ldots&E_i\otimes^{tr}_S F_j&\ldots\\&\vdots&\end{pmatrix}
\]
with vertical and horizontal bonding maps induced by the bonding maps for $E$ and $F$, respectively. The vertical bonding maps use in addition the symmetry isomorphism in $C(\PST_{\A^1}(S))$. Finally, the choice of the cofinal $\N\subset \N\times\N$ converts a bispectrum to a spectrum.

Of course, this is not even associative, so one does not achieve a tensor operation on $\Spt_{T^{tr}}(S)$, but $\boxtimes^{tr}_S$ (on cofibrant objects) does pass to the localization $\DM(S)$, and gives rise there to a tensor structure, making $\DM(S)$ a tensor triangulated category. We write this tensor operation as $\otimes_S$, as before.

\begin{rem} One can also define a ``Spanier-Whitehead" category   $\DM^\SW(S)$ by inverting  the functor $-\otimes T^{tr}=-\otimes\Z_S(1)[2]$ on $\DM^\eff(S)$; this is clearly equivalent to inverting $-\otimes\Z_S(1)$. Concretely, $\DM^\SW(S)$ has objects $X(n)$, $n\in\Z$, with morphisms
\[
\Hom_{\DM^\SW(S)}(X(n),Y(m)):=\colim_N\Hom_{\DM^\eff(S)}(X\otimes\Z_S(N+n),Y\otimes\Z_S(N+m)).
\]
Sending $X$ to $X(n)$ clearly defines an auto-equivalence of $\DM^\SW(S)$.

$\DM^\SW(S)$ inherits the structure of a triangulated category from $\DM^\eff(S)$, by declaring a triangle $\sT$ in $\DM^\SW(S)$ to be distinguished if $\sT(N)$ is the image of a distinguished triangle in $\DM^\eff(S)$ for some $N>>0$. One shows that the symmetry isomorphism $\Z_S(1)\otimes\Z_S(1)\to \Z_S(1)\otimes\Z_S(1)$ is the identity, which implies that $\DM^\SW(S)$ inherits a tensor structure from $\DM^\eff(S)$.  

Since $-\otimes T^{tr}$ is isomorphic to the shift operator in $\DM(S)$, this functor is invertible on $\DM(S)$, hence $\DM^\eff(S)\to \DM(S)$ factors through a canonical exact functor $\phi_S:\DM^\SW(S)\to \DM(S)$. Giving $\DM(S)$ the tensor structure described above, it is easy to see that $\phi_S$ is a tensor functor.
\end{rem}

\subsection{Motives and smooth motives} 

For $X\in\Sm/S$, we have the presheaf $U\mapsto C^S(X,0)(U)$, giving us the object $C^S(X)$ in $C^-(\Sh^{tr}_\Nis(S))$. As $C^S(X)^0=\Z^{tr}_S(X)$, we have the natural map
\[
\iota_X:\Z^{tr}_S(X)\to C^S(X)
\]
\begin{lem}\label{lem:SuslinEquiv} $\iota_X$ is an isomorphism in $\DM^\eff(S)$.
\end{lem}

\begin{proof} Let    $\Z^{tr}_S(X)^*$ be the complex which is $\Z^{tr}_S(X)$ in each degree $n\le 0$, and with differential $d^n:\Z^{tr}_S(X)^n\to \Z^{tr}_S(X)^{n+1}$ the identity if $n$ is even and 0 if $n$ is odd. The projection $\Z^{tr}_S(X\times\square^n)\to \Z^{tr}_S(X)$ gives a map of  complexes
\[
\pi:C^S(X)\to \Z^{tr}(X)^*.
\]
As $\Z^{tr}_S(X\times\square^n)\to \Z^{tr}_S(X)$ is an isomorphism in $\DM^\eff(S)$ for all $n$, it follows that $\pi$ is an isomorphism. Since $\pi\iota$ is a homotopy equivalence, $\iota$ is an isomorphism as well.
\end{proof}

Let $G^*_S:C^-(\Sh^{tr}_\Nis(S))\to C(\Sh^{tr}_\Nis(S))$ be the Godement resolution functor, with respect to $S_\Zar$. Concretely, for $\sF\in Sh^{tr}_\Nis(S)$, $G^0_S(\sF)$ is the sheaf
\[
G^0_S(\sF)(U):=\prod_{s\in S}\sF(U\otimes_S\sO_{S,s}).
\]
where we define $\sF(U\otimes_S\sO_{S,s})$ as the limit of $\sF(U\otimes_SV)$, where $V$ runs over the Zariski open neighborhoods of $s$ in $S$. We have the natural transformation $\id\to G_S$ induced by the inclusion $\sF\to G^0_S\sF$.  Since $\sF\to G^*_S\sF$ is a Zariski local weak equivalence, this map induces an isomorphism in $D(\Sh^{tr}_\Nis)$.  Thus, lemma~\ref{lem:SuslinEquiv} gives

\begin{lem}\label{lem:SuslinEquiv2} For each $X\in\Sm/S$, the composition
\[
\Z^{tr}_S(X)\xrightarrow{\iota_X} C^S(X)\to G_S^*C^S(X)
\]
define an isomorphism $m_S(X)\cong G_S^*C^S(X)$ in $\DM^\eff(S)$.
\end{lem}

\begin{rem} Since $S$ has finite Krull dimension, say $d$, Grothendieck's theorem \cite{GrothCohDim} tells us that $S$ has Zariski cohomological dimension $\le d$. Thus, for $\sF\in C^-(\Sh^{tr}_\Nis(S))$, $G_S(\sF)$ is cohomologically bounded above. Therefore, the composition of $G_S$ with the quotient functor $C(\Sh^{tr}_\Nis(S))\to D(\Sh^{tr}_\Nis(S))$ factors (up to natural isomorphism) through a modified Godement resolution
\[
G_S^-:C^-(\Sh^{tr}_\Nis(S))\to D^-(\Sh^{tr}_\Nis(S)).
\]
\end{rem}

We now proceed to construct an exact functor
\[
\rho_S:\PMot^\eff(S)\to \DM^\eff(S).
\]
For this, consider the presheaf of DG categories $\un{\DGCor}_S$ on $S_\Zar$
\[
U\mapsto \DGCor_U.
\]
Recall that $\DGCor_U$ has objects $X\in\Sm/U$, and the full subcategory with objects $X\in\Prj/U$ is our category $\DGPCor_U$. We have the ``Godement extension" $R\Gamma(S, \un{\DGCor}_S)$ with objects $X\in\Sm/S$, containing $R\Gamma(S, \un{\DGPCor}_S)$ as the full DG subcategory with objects $X\in\Prj/S$. 

Take $X\in\Sm/S$. By construction, the presheaf on $\Sm/S$ of Hom-complexes
\[
U\mapsto \sHom_{R\Gamma(S, \un{\DGCor}_S)}(U,X)^*
\]
is just $G_S^*C^S(X)$. Thus, for $X,Y\in\Sm/S$, the composition law in $R\Gamma(S, \un{\DGCor}_S)$ defines a natural map
\[
\tilde\rho_{X,Y}:\sHom_{R\Gamma(S, \un{\DGCor}_S)}(X,Y)^*\to
\sHom_{C(\Sh^{tr}_\Nis(S))}(G_S^*C^S(X),G_S^*C^S(Y)).
\]

This defines for us the functor of DG categories
\[
\tilde\rho:R\Gamma(S, \un{\DGCor}_S)\to C(\Sh^{tr}_\Nis(S)).
\]
Applying the functor $K^b$ and composing with the total complex functor
\[
\Tot:K^b(C(\Sh^{tr}_\Nis(S)))\to K(\Sh^{tr}_\Nis(S))
\]
and the quotient functor
\[
K(\Sh^{tr}_\Nis(S))\to \DM^\eff(S)
\]
gives us the exact functor
\[
K^b(\tilde\rho):K^b(R\Gamma(S, \un{\DGCor}_S))\to \DM^\eff(S)
\]
If we restrict to the full subcategory $K^b(R\Gamma(S, \un{\DGPCor}_S))$ of $K^b(R\Gamma(S, \un{\DGCor}_S))$ and extend canonically to the idempotent completion, we have defined an exact functor
\[
\rho^\eff_S:\PMot^\eff(S)\to  \DM^\eff(S).
\]
with  $\rho^\eff_S(X)=G_S^*C^S(X)$ for all $X\in\Prj/S$. The natural isomorphism $m_S(X)\to G_S^*C^S(X)$
in $\DM^\eff(S)$ constructed in lemma~\ref{lem:SuslinEquiv2} shows that the diagram
\[
\xymatrix{
\Prj/S\ar[r]\ar[dr]_{m_S}&\PMot^\eff(S)\ar[d]^{\rho^\eff_S}\\
&\DM^\eff(S)}
\]
commutes up to natural isomorphism.

We note that $\rho^\eff_S(\L)=T^{tr}$, hence the composition
\[
\PMot^\eff(S)\xrightarrow{\rho_S} \DM^\eff(S)\to \DM^\SW(S)
\]
sends $-\otimes\L$ to the invertible endomorphism $-\otimes T^{tr}$ on $\DM^\SW(S)$, hence this composition factors through a canonical extension
\[
\rho^\SW_S:\PMot(S)\to  \DM^\SW(S).
\]
Let
\[
\rho_S:\PMot(S)\to  \DM(S).
\]
be the composition of $\rho^\SW$ with the canonical functor $\DM^\SW(S)\to\DM(S)$.

\begin{rem} It is easy to check that the restriction of $\rho^\eff_S$ to $H^0\DGCor_S$ is a tensor functor.  By the  surjectivity of the cycle class map (lemma~\ref{lem:FLMovOrig}) this  shows that
the composition
\[
\rho^\eff_S\circ \psi^\eff:\CM^\eff(S)\to  \DM^\eff(S)
\]
is a tensor functor. Similarly, 
\[
\rho_S\circ \psi:\CM(S)\to  \DM(S)
\]
is a tensor functor. 
\end{rem}

\begin{thm}\label{thm:MotDuality} Let $X,Y,Z$ be in $\Prj/S$, $d=\dim_SY$, and take $A, B$ in $\DM(S)$. Then for every $n\in\Z$, there is a natural isomorphism
\begin{multline*}
\Hom_{\DM^\eff(S)}(m_S(X), m_S(Y)\otimes m_S(Z)[n])\\\cong \Hom_{\DM^\eff(S)}(m_S(X)\otimes m_S(Y), m_S(Z)\otimes \Z(d)[2d+n]),
\end{multline*}
and a natural isomorphism
\[
\Hom_{\DM(S)}(A, m_S(Y)\otimes B[n])\cong \Hom_{\DM(S)}(A\otimes m_S(Y), B\otimes \Z(d)[2d+n])
\]
\end{thm}
\begin{proof}  Since $Y\in \CM(S)$ has the dual $(Y\otimes\L^{-d},\delta_Y,\epsilon_Y)$, it follows that
$m_S(Y)\cong \rho_S\psi_S(Y)$ has the dual $(m_S(Y)(-d)[-2d],  \rho_S\psi_S(\delta_Y),  \rho_S\psi_S(\epsilon_Y))$ in $\DM(S)$. This gives us the second isomorphism. For the first, we have the limited cancellation theorem (see theorem~\ref{thm:MotivesMain}), giving isomorphisms
\begin{align*}
&\Hom_{\DM^\eff(S)}(m_S(X), m_S(Y)\otimes m_S(Z)[n])\\
&\hskip 100pt\cong
\Hom_{\DM(S)}(m_S(X), m_S(Y)\otimes m_S(Z)[n])\\
&\Hom_{\DM^\eff(S)}(m_S(X)\otimes m_S(Y), m_S(Z)\otimes \Z(d)[2d+n])\\
&\hskip 100pt\cong
\Hom_{\DM(S)}(m_S(X)\otimes m_S(Y), m_S(Z)\otimes \Z(d)[2d+n]).
\end{align*}
Thus, the duality in $\DM(S)$ yields the desired isomorphism of Hom-groups in $\DM^\eff(S)$.
\end{proof}

\begin{cor}\label{cor:Main} The functors
\begin{align*}
&\rho^\eff_S:\PMot^\eff(S)\to  \DM^\eff(S)\\
&\rho_S:\PMot(S)\to  \DM(S).
\end{align*}
are fully faithful embeddings. 
\end{cor}

\begin{proof} We give the proof for $\rho^\eff_S$; the proof for $\rho_S$ is the same.

 Take $X, Y\in \Prj/S$, let $d=\dim_SY$. By theorem~\ref{thm:MotDuality}, we have the duality isomorphism
\[
\Hom_{\DM^\eff(S)}(m_S(X), m_S(Y)[n])\cong \Hom_{\DM^\eff(S)}(m_S(X)\otimes m_S(Y), \Z^{tr}(d)[2d+n])
\]
As the construction of the duality isomorphism arises from the duality in $\CM(S)$, the diagram
\[
\xymatrixcolsep{15pt}
\xymatrix{
\Hom_{\PMot^\eff(S)}(X,Y[n])\ar[r]^-{\sim}\ar[d]_{\rho^\eff_S}&\Hom_{\PMot^\eff(S)}(X\times_SY,\L^d[n])\ar[d]^{\rho^\eff_S}\\
\Hom_{\DM^\eff(S)}(m_S(X), m_S(Y)[n])\ar[r]_-\sim&
 \Hom_{\DM^\eff(S)}(m_S(X\times_SY), \Z^{tr}(d)[2d+n])
 }
 \]
 commutes. But by theorem~\ref{thm:MotivesMain} and remark~\ref{rem:MotCohIso}, we have
 \begin{multline*}
  \Hom_{\DM^\eff(S)}(m_S(X\times_SY), \Z^{tr}(d)[2d+n])\\\cong
  \Hom_{\DM^\eff(k)}(m_k(p_*X\times_SY), \Z^{tr}(d)[2d+n])\\\cong
  H^{2d+n}(p_*X\times_SY,\Z(d)).
  \end{multline*}
  It is easy to check that the natural map
  \begin{multline*}
  H^{2d+n}(p_*X\times_SY,\Z(d))\cong  H^{n}(C^k(\L^d)(p_*(X\times_SY))\\
  =H^n(C^S(\L^d)(X\times_SY))\to
  \Hom_{\DM^\eff(S)}(m_S(X\times_SY), \Z^{tr}(d)[2d+n])
  \end{multline*}
inverts this isomorphism, from which it follows that the composition
\begin{multline*}
   H^{2d+n}(X\times_SY,\Z(d))\cong  H^{n}(C^k(\L^d))(p_*(X\times_SY))=H^n(C^S(\L^d)(X\times_SY))\\
   \to \Hom_{\PMot^\eff(S)}(X\times_SY,\L^d[n])\xrightarrow{\rho_S}
    \Hom_{\DM^\eff(S)}(m_S(X\times_SY), \Z^{tr}(d)[2d+n])\\
    \cong  \Hom_{\DM^\eff(k)}(m_k(p_*X\times_SY), \Z^{tr}(d)[2d+n])\cong
  H^{2d+n}(X\times_SY,\Z(d))
\end{multline*}
is the identity. Therefore,
\[
\rho^\eff_S:\Hom_{\PMot^\eff(S)}(X\times_SY,\L^d[n])\to  \Hom_{\DM^\eff(S)}(m_S(X\times_SY), \Z^{tr}(d)[2d+n])
\]
is an isomorphism.
\end{proof}

\begin{definition} Let $\DTM(S)$ be the full triangulated subcategory of $\DM(S)$ generated by the Tate objects $\Z^{tr}_S(n)$, $n\in\Z$, and let $\DTM^\eff(S)$ be the full triangulated subcategory of $\DM^\eff(S)$ generated by the Tate objects $\Z^{tr}_S(n)$, $n\ge0$
\end{definition}

As immediate consequence of corollary~\ref{cor:Main}, we have
\begin{cor}\label{cor:MainTate} The restriction of $\rho_S$ to 
\[
\rho_S^{\text{Tate}}:\PTMot(S)\to \DM(S)
\]
is a fully faithful embedding, giving an equivalence of $\PTMot(S)$ with $\DTM(S)$. The version with $\Q$-coefficients
\[
\rho_S^{\text{Tate}}:\PTMot(S)_\Q\to \DM(S)_\Q
\]
defines an equivalence of $\PTMot(S)_\Q$ with $\DTM(S)_\Q$ as tensor triangulated categories.

Similarly, we have the equivalences
\[
\rho_S^{\eff\text{Tate}}:\PTMot^\eff(S)\to \DTM^\eff(S)
\]
and
\[
\rho_{S\Q}^{\eff\text{Tate}}:\PTMot^\eff(S)_\Q\to \DTM^\eff(S)_\Q
\]
\end{cor}

\end{document}